\documentclass[English,11pt]{amsart}
\usepackage[T1]{fontenc}
\usepackage[utf8]{inputenc} % may help to print special character in .bib

\usepackage{amssymb}
\usepackage{amsmath,amsthm, amsfonts, thmtools, thm-restate}

\usepackage[linkcolor=red!80!black,colorlinks=true,citecolor=green!65!blue, urlcolor=blue!70!black]{hyperref}
\usepackage{cleveref}

\usepackage{enumitem}

\usepackage{url}
\usepackage[backend=biber, style=ieee-alphabetic, maxnames=5, maxalphanames=5, giveninits=true, sorting=nyt, isbn=false, doi=false, url=false, eprint=false]{biblatex} %Imports biblatex package

\usepackage{xfrac}

\usepackage{multirow}

\usepackage{lmodern} % to deal with sizes of font shape
\usepackage{anyfontsize} % to deal with sizes of font shape

\DeclareNameAlias{author}{given-family}
\AtEveryBibitem{\clearlist{language}}
\AtEveryBibitem{\clearlist{month}}
\newtheorem{thm}{Theorem}[section]

\newtheorem{lem}[thm]{Lemma}

\newtheorem{dfn}[thm]{Definition}
\newtheorem{rmk}[thm]{Remark}

\newtheorem{theorem}{Theorem}

\usepackage{graphicx}
\usepackage[left=3cm,right=3cm,top=4cm,bottom=3cm, marginparwidth=2.35cm]{geometry} % Pour fixer les marges du document
\usepackage{tabularx}
\usepackage{mathrsfs} % pour les lettres caligraphiées

\usepackage{caption}
\usepackage{subcaption}

\usepackage{pgf,tikz}
\usetikzlibrary{arrows}
\usetikzlibrary{arrows.meta} %for tip kind
\usepackage{cleveref}
\usepackage{color}
\newcommand{\RR}{\mathbb{R}}
\newcommand{\NN}{\mathbb{N}}
\newcommand{\ZZ}{\mathbb{Z}}

\newcommand{\GL}{\mathrm{GL}}

\definecolor{T-MC1}{RGB}{255,255,255} %white
\definecolor{T-MC2}{RGB}{238,204,102} %light brown
\definecolor{T-MC3}{RGB}{238,153,170} %pink
\definecolor{T-MC4}{RGB}{102,153,204} %lightblue
\definecolor{T-MC5}{RGB}{153,119,0} %dark brown
\definecolor{T-MC6}{RGB}{153,68,85} %purple
\definecolor{T-MC7}{RGB}{0,68,136} %deep blue
\definecolor{T-MC0}{RGB}{0,0,0} %black

\definecolor{colR}{rgb}{0.6,0.2,0} % color of the rectangles
\definecolor{colright}{RGB}{0,125,150} % light blue for some traj (reversed)
\definecolor{colleft}{rgb}{0.05,0,0.5} % deep blue for some traj (reversed)
\definecolor{colmid}{rgb}{0.9,0.5,0} % orange for some traj (translated)
\definecolor{colmid2}{rgb}{0.75,0.05,0} % deeper orange for some traj (translated)

\title{Wind-tree tiling billiards and their trapping strips}
\author{Magali Jay}

\begin{document}
\maketitle

\textbf{Abstract.} 
We introduce a new dynamical system: the wind-tree tiling billiards. This system studies trajectories of a ray in Euclidean space which has a negative refractive index when encountering rectangular obstacles located at lattice points. We show that for almost every configuration of the system, trajectories with initial vertical direction are trapped in an infinite strip of the plane. This result is reminiscent of the propagation of light rays in Eaton lenses, as shown by Fr\k{a}czek and Schmoll \cite{FS}.

\section{Introduction}

We will define the wind-tree tiling billiards, which was motivated by combining two dynamical systems. The first is the wind-tree model. Introduced in 1912 by Tatiana and Paul Ehrenfest \cite{Ehrenfest} to study the movement of gas molecules, the wind-tree model is a billiards in the infinite Euclidean plane with rectangular obstacles on lattice points. 
This system has been extensively studied over the past decades. 
A generic trajectory is recurrent \cite{AvilaHubert}, has a diffusion rate of $\frac{2}{3}$ \cite{DHL}, which also holds in average \cite{Barazer}.

The second dynamical system is tiling billiards. Tiling billiards are a class of dynamical systems introduced by Davis et al. in 2018 \cite{DDiPRStL}, motivated by the discovery of metamaterials with negative refractive index \cite{SSS01,SPW04}. They have been studied in several but yet not many configurations:
\begin{enumerate}
    \item In the tri-hexagonal tiling \cite{DH}, where generic trajectories are dense in the plane except in a periodic family of triangles;
    \item In a triangle tiling (where each tile is a $180^\circ$ rotation of each of its neighbors), a generic trajectory is either periodic or linearly escaping (meaning that there exists a direction $v\in\RR^2$ such that the position $p_n$ of the trajectory after $n$ refractions is approximated b $nv$: the sequence $(|p_n-nv|)_n$ is bounded), see \cite{BSDFI18,PR19,HPR22};
    \item In a cyclic quadrilateral tiling (meaning that the quadrilateral is inscribed in a circle), see \cite{DHMPRS};
    \item With polygons inscribed in a half circle (which do not tile the plane but still allow to define a dynamical system), an open set of trajectories are escaping to infinity, with unbounded deviations upon the asymptotic direction, see \cite{jay}.
\end{enumerate}

In \Cref{subsec:statement-thm}, we will introduce the wind-tree tiling billiards, which takes the wind-tree model, but consider the obstacles as metamaterials of opposite refraction index: the trajectory cross the obstacle and is refracted.
In contrast to the study of tiling billiards and the wind-tree model, we show that for almost every configuration of the system, trajectories with initial vertical direction are trapped in an infinite strip of the plane (see \Cref{thm:tb_in_wt}). In \Cref{subsec:connections-other-work} we will give more context for this result in the scope of the field, and in \Cref{subsec:strategy} we will give the strategy and overview of the proof, including a key theorem (\Cref{thm:bounded_intersection}) on bounding algebraic intersection number between homology cycles and trajectories.

\subsection{Statement of the main theorem}\label{subsec:statement-thm}
Let $\Lambda \subset \RR^2$ be a lattice, let $a,b>0$, and let $\theta\in \left(0,\pi\right/2)$. We consider an array of rectangles of size $a\times b$ with their centers placed on points of the lattice $\Lambda$, and making an angle $\theta$ with the horizontal (see \Cref{fig:setting}).
We restrict to tuples $(\Lambda,a,b,\theta)$ where the rectangles are pairwise disjoint in the plane, and say these tuples are \textit{admissible}.
\begin{figure}[h!]
\centering
\definecolor{coltheta}{rgb}{0,0,0}%{0,0.2,0.6}
\def\eps{0.03}
\def\xa{3.9}
\def\ya{1.3}
\def\xb{-1}
\def\yb{3}
\def\xp{-0.72}
\def\xpp{3.62}
\def\ux{7.8} % x coordinate of the first basis vector of \Lambda
\def\uy{-1.2} % y coordinate of the first basis vector of \Lambda
\def\vx{-1} % x coordinate of the 2nd basis vector of \Lambda
\def\vy{6} % y coordinate of the 2nd basis vector of \Lambda
\begin{tikzpicture}[line cap=round,line join=round,>=stealth,x=0.4cm,y=0.4cm]
%\clip(-1.2,-0.5) rectangle (4.1,4.5);
\foreach \k in {0,1}
\foreach \j in {0,1}
{\draw[color=colR,opacity=0.5,fill=colR, fill opacity=0.15, shift={(\j*\ux+\k*\vx,\j*\uy+\k*\vy)}] (0,0) -- (3.9,1.3) -- (2.9,4.3) -- (-1,3) -- cycle; %midle = 1.45, 2.15
}
\draw[gray!50!black] (1.45-\ux/2-\vx/2,2.15-\uy/2-\vy/2) -- ++(\ux,\uy) -- ++(\vx,\vy) -- ++(-\ux,-\uy) -- cycle; %midle = \ux-\vx, \vy-\ux
%\draw[gray] (2.45-\ux/2+\vx/2,2.15+\uy/2-\vy/2) -- (2.45+\ux/2+\vx/2,2.15+1.5*\uy-\vy/2) -- (2.45+\ux/2+\vx/2,2.15+1.5*\uy+\vy/2) -- (2.45-\ux/2+1.5*\vx,2.15+\uy/2+\vy/2) -- cycle;
\begin{footnotesize}
\draw [color=coltheta,fill=coltheta,fill opacity=0.1] (0,0) -- (0.8,0) arc (0:18.43:0.8) -- cycle;
\draw [color=coltheta,right] (0.8,0) node {$\theta$};
\draw [<->,shift={(0.05*\xb,0.05*\yb)}] (\xb,\yb)--(\xb+\xa,\yb+\ya);
\draw [above=4] (\xa/2+\xb,\ya/2+\yb) node {$a$};
\draw [<->,shift={(0.05*\xa,0.05*\ya)}] (\xa,\ya)--(\xb+\xa,\yb+\ya);
\draw [right=4] (\xa+\xb/2,\ya+\yb/2) node {$b$};
\draw [gray!50!black, left] (1.45-\ux/2+\vx/2,2.15-\uy/2+\vy/2) node {$\Lambda$};
\end{footnotesize}
\end{tikzpicture}
\caption{Setting of our systems}\label{fig:setting}
{\small
The parallelogram in grey is a fundamental domain for the lattice $\Lambda$ and the vectors generating its two sides are in $\Lambda$.
 }
\end{figure}
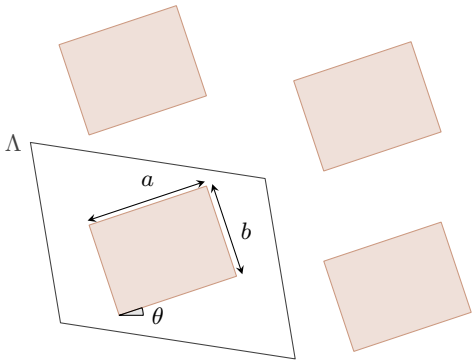

The trajectories we study go in a straight line outside the rectangles, and are refracted when crossing the rectangles. See Figure \ref{fig:tb}. We call them \emph{tiling billiards trajectories}.
We consider the tiling billiards trajectories with an initial \emph{vertical} direction 
(going either upward or downward) from any point of the plane outside the rectangles. 
If a trajectory reaches a corner, it stops. We call this system \emph{the wind-tree tiling billiards}, which we denote by $W(\Lambda,a,b,\theta)$.
\begin{figure}[h!]
\centering
\includegraphics[scale=0.2]{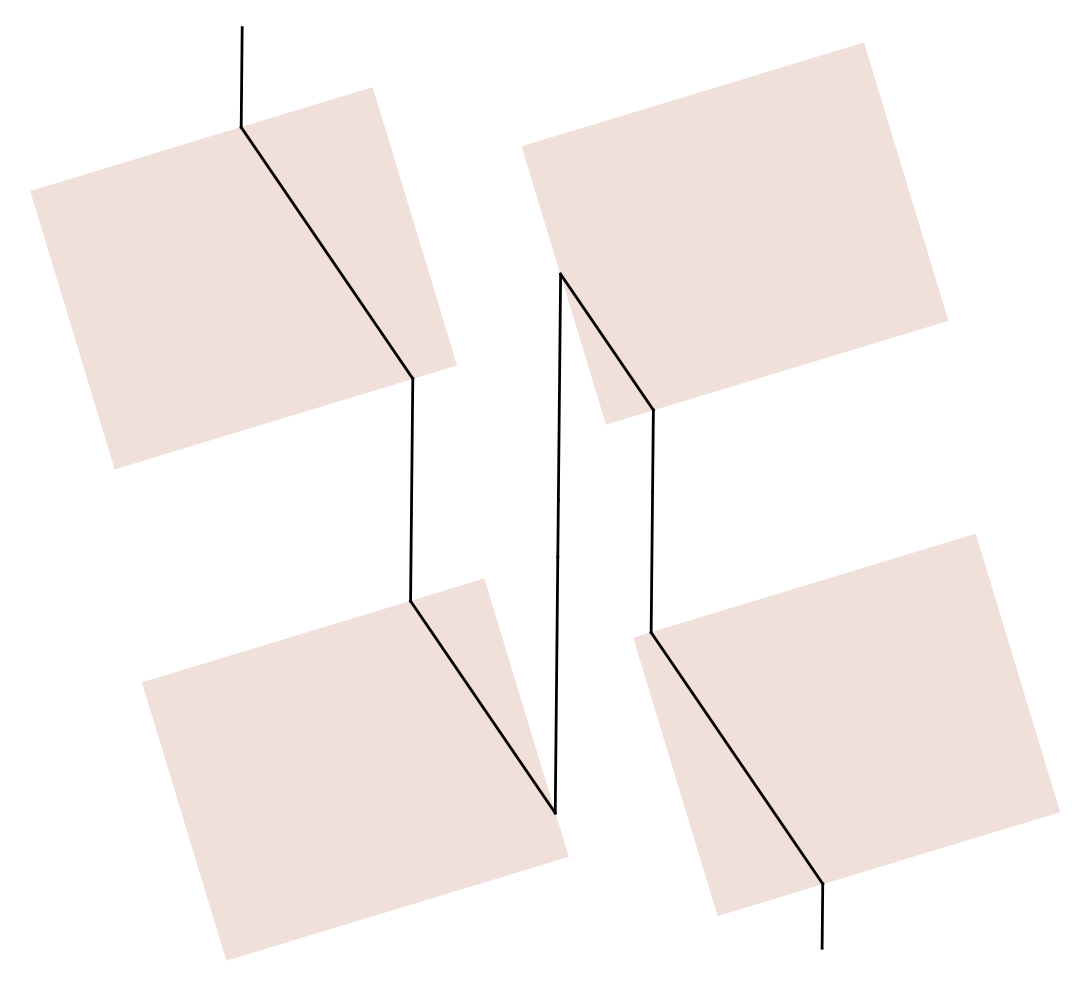}
\caption{A few steps of tiling billiards trajectory}\label{fig:tb}
\label{fig:traj_first_step}
\end{figure}
 
Note that we can always parameterize the system with $\theta\in[0,\frac{\pi}{2})$ and that we can exclude $\theta=0$ because then all trajectories are just vertical lines.
Note also that considering a vertical initial direction is not a restriction since a non-vertical direction, say making an angle $\alpha$ with the vertical, in the system $(\Lambda,a,b,\theta)$ corresponds to a vertical trajectory in the system $(R_{-\alpha}\Lambda,a,b,\theta-\alpha)$ where $R_{-\alpha}\Lambda$ denotes the lattice $\Lambda$ rotated by the angle $-\alpha$.
We prove the following Theorem, illustrated by Figure~\ref{fig:ex_traj_strip}.

    \begin{theorem}\label{thm:tb_in_wt}
        For almost every admissible tuple $(\Lambda,a,b,\theta)$, there exist constants  ${C  > 0}$ and $ \Theta \in [0, \pi) $, such that every vertical trajectory in $W(\Lambda,a,b,\theta)$ is trapped in an infinite band of width $C > 0$ in direction making an angle $\Theta$ with the horizontal. 
    \end{theorem}

\begin{figure}[h!]
\centering
\includegraphics[scale=0.25]{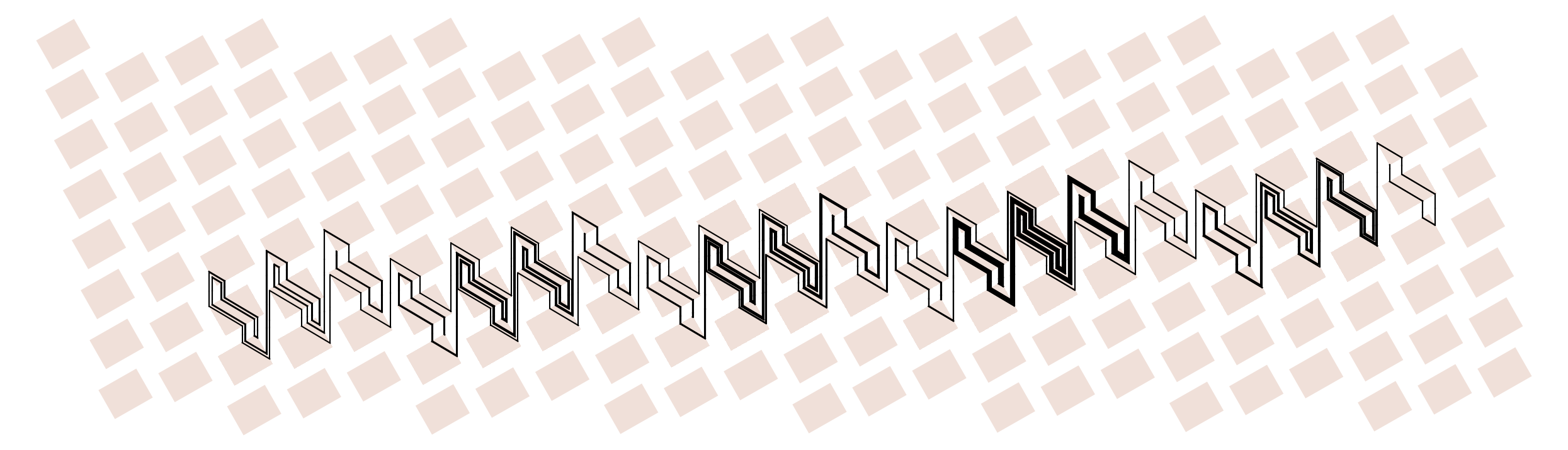}
\caption{Example of a trajectory, trapped in a strip}
\label{fig:ex_traj_strip}
\end{figure}

\subsection{Connections to other work: Eaton lenses}\label{subsec:connections-other-work}
Eaton lenses \cite{eaton} are glass spheres whose optical index depends on the radius in such a way that an incident ray goes out parallel to itself after a half turn, see \Cref{fig:eaton_lens}.
In \cite{FS}, Fr\k{a}czek and Schmoll studied a periodic arrangement of Eaton lenses (see \Cref{fig:eaton_lenses_arrangement}) and showed that almost every light ray trajectory is trapped in a strip.
The rectangle obstacles that we consider here refract the light in such a way that they behave around two of their corners as Eaton lenses do: they reflect the light parallel to itself, see \Cref{fig:rect-similar-to-eaton}. However, they also translate the light when the ray hits some part of the obstacle, see \Cref{fig:rect-different-from-eaton}.
We know from \cite{FS} that the first behavior leads to trajectories trapped in a strip, and a pure translation behavior would also lead to trajectories trapped in a strip. The combination of these two behaviors could \emph{a priori} produce trajectories exploring the entire plane. This is not the case, \Cref{thm:tb_in_wt} states that they generically remain in a strip.
\begin{figure}[h]
    \centering
    \begin{subfigure}[b]{0.4\textwidth}
    \def\r{1} % the raidus
\def\p{0.65} % the x-coordinate of the initial point of the light ray
\begin{tikzpicture}[line cap=round,line join=round,>=stealth,x=1cm,y=1cm]
\clip (-3,-2) rectangle (3,2);
\coordinate (a) at (\p,-{sqrt(\r-\p*\p)}) ;
\coordinate (b) at (-\p,-{sqrt(\r-\p*\p)}) {};
\coordinate (c) at (0,\r/3) {} ;
\coordinate (center) at (0,0) {} ;
\draw[color=colR,opacity=0.5,fill=colR, fill opacity=0.15] (center) circle(\r);

\draw [->] (\p,-\r-0.5) -- (\p,-\r) ;
\draw (\p,-\r-0.1) -- (a) ;
\draw (a) to[out=90, in=0] (c);
\draw (c) to[out=180, in=90] (b);
\draw [->] (b) -- (-\p,-\r-0.1) ;
\draw (-\p,-\r) -- (-\p,-\r-0.5) ;

% construction lines
\begin{footnotesize}
\draw [gray] (center) node {$\times$};
\end{footnotesize}
\draw [gray] (0,-{sqrt(\r-\p*\p)}+0.15) --++ (0.15,0) --++ (0,-0.15); % pernpendicular sign
\draw [densely dotted, gray] (center) -- (0,-{sqrt(\r-\p*\p)});
\draw [ultra thin,gray] (a) -- (b);
\draw [gray] (\p/2-0.05,-{sqrt(\r-\p*\p)}-0.1) --++ (0.1,0.2) ;
\draw [gray] (-\p/2-0.05,-{sqrt(\r-\p*\p)}-0.1) --++ (0.1,0.2) ;
\end{tikzpicture}
    \caption{Refraction through an Eaton lens (the bending is not accurate)}\label{fig:eaton_lens}
    \end{subfigure}
    \hfill
    \begin{subfigure}[b]{0.55\textwidth}
            \def\r{1}
\def\ux{2.5} % x coordinate of the first basis vector of \Lambda
\def\uy{-0.7} % y coordinate of the first basis vector of \Lambda
\def\vx{-1} % x coordinate of the 2nd basis vector of \Lambda
\def\vy{3} % y coordinate of the 2nd basis vector of \Lambda
\begin{tikzpicture}[line cap=round,line join=round,>=stealth,x=0.5cm,y=0.5cm]
%\clip(-1.2,-0.5) rectangle (4.1,4.5);
\foreach \k in {0,1,2}
\foreach \j in {0,1,2}
{\draw[color=colR,opacity=0.5,fill=colR, fill opacity=0.15, shift={(\j*\ux+\k*\vx,\j*\uy+\k*\vy)}] (0,0) circle(\r);
}

\draw[->] (0.8,0.8) -- (0.8,1.2);
\draw (0.8,1.1) -- (0.8,1.59);
\draw[gray,densely dotted] (0.8,1.59) -- (2.2,1.59);
\draw (2.2,1.59) -- (2.2,0.25);
\draw[gray,densely dotted] (2.2,0.25) -- (2.8,0.25);
\draw (2.8,0.25) -- (2.8,3.62);
\draw[gray,densely dotted] (2.8,3.62) -- (3.2,3.62);
\draw (3.2,3.62) -- (3.2,2.2);
\draw[gray,densely dotted] (3.2,2.2) -- (4.8,2.2);
\draw[->] (4.8,2.2) --++ (0,0.7); 
%%\draw (0.8,0.8) -- (0.8,1.59);
\end{tikzpicture}
            \caption{Exemple of a light ray's trajectory in an Eaton lenses arrangement (outside of the lenses)}\label{fig:eaton_lenses_arrangement}
    \end{subfigure}
    \caption{A dynamical system with Eaton lenses}\label{fig:system_eaton}
\end{figure}
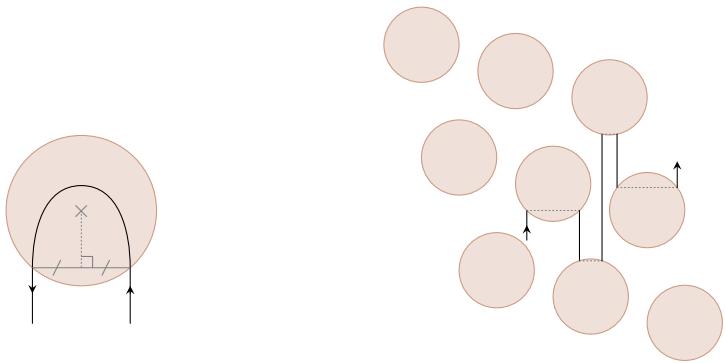
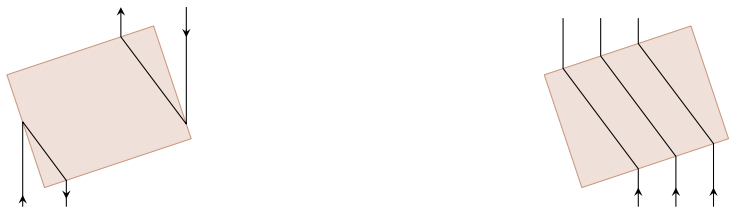
\begin{figure}[h]
    \centering
    \begin{subfigure}[b]{0.45\textwidth}
    {\center
    % just for this figure - traj in black
\definecolor{colright}{rgb}{0,0,0}
\definecolor{colleft}{rgb}{0,0,0}
\def\eps{0.03}
\def\xa{3.9}
\def\ya{1.3}
\def\xb{-1}
\def\yb{3}
\def\xp{-0.72}
\def\xpp{3.62}
\begin{tikzpicture}[line cap=round,line join=round,>=stealth,x=0.5cm,y=0.5cm]
\clip (-5,-0.5) rectangle (4.9,4.8);
\draw[color=colR,opacity=0.5,fill=colR, fill opacity=0.15] (0,0) -- (3.9,1.3) -- (2.9,4.3) -- (-1,3) -- cycle;

% traj left
\draw [colleft,->] (-0.58,-.5) -- (-0.58,-0.2) ;
\draw [colleft] (-0.58,-0.3) -- (-0.58,1.75) -- (0.58,0.19);
\draw [colleft,->] (0.58,0.19)-- (0.58,-0.3) ;
\draw [colleft] (0.58,-0.2)-- (0.58,-0.5);

% traj right
\draw [colright,->] (3.77,4.8) -- (3.77,4);
\draw [colright,->] (3.77,4.2) -- (3.77,1.69) -- (2.03,4.01)-- (2.03,4.8);

\end{tikzpicture}}
    \caption{The behavior similar to the Eaton lens}\label{fig:rect-similar-to-eaton}
    \end{subfigure}
    \begin{subfigure}[b]{0.45\textwidth}
    % just for this figure - traj in black
\definecolor{colmid}{rgb}{0,0,0}
\def\eps{0.03}
\def\xa{3.9}
\def\ya{1.3}
\def\xb{-1}
\def\yb{3}
\def\xp{-0.72}
\def\xpp{3.62}
\begin{tikzpicture}[line cap=round,line join=round,>=stealth,x=0.5cm,y=0.5cm]
\clip (-5,-0.5) rectangle (4.9,4.8);
\draw[color=colR,opacity=0.5,fill=colR, fill opacity=0.15] (0,0) -- (3.9,1.3) -- (2.9,4.3) -- (-1,3) -- cycle;

% traj middle 1
\draw [colmid,->]  (1.5,-0.5) --++ (0,0.5);
\draw [colmid] (1.5,-0.1) -- (1.5,0.5) -- (-0.5,3.17) -- (-0.5,4.5);

% traj middle 2
\draw [colmid,->]  (2.5,-0.5) --++ (0,0.5);
\draw [colmid] (2.5,-0.1) -- (2.5,0.83) -- (0.5,3.5) -- (0.5,4.5);

% traj middle 3
\draw [colmid,->]  (3.5,-0.5) --++ (0,0.5);
\draw [colmid] (3.5,-0.1) -- (3.5,1.17) -- (1.5,3.83) -- (1.5,4.5);

\end{tikzpicture}
    \caption{The additional behavior (translation)}\label{fig:rect-different-from-eaton}
    \end{subfigure}
    \caption{Comparison of the refraction behavior of our rectangle with an Eaton lens}
    \label{fig:comp_rect-eaton}
\end{figure}

\subsection{Strategy}\label{subsec:strategy}
To establish Theorem~\ref{thm:tb_in_wt}, we transform (in \Cref{sec:corresponding_surf_and_question}) our broken-line trajectory in the plane into a geodesic on some half-translation surface $X_{(\Lambda,a,b,\theta)}$ that corresponds to the system $W(\Lambda,a,b,\theta)$. 
This idea was already used by Fr{\k{a}}czek and Schmoll to study light rays in an Eaton lenses array \cite{FS}. One of the main points of this article is to build the surface $X_{(\Lambda,a,b,\theta)}$ corresponding to our case and show that almost every such surface has a generic behavior in the stratum, meaning that it is generic for Oseledets theorem (see \Cref{subsec:Oseledets}) and that its orbit under the Teichmüller flow $g_t$ (see \Cref{subsec:background_teich}) is recurrent in the stratum.  
This surface has four singularities, two of conical angle $\pi$ and two of conical angle $3\pi$. In other words $X_{(\Lambda,a,b,\theta)}$ lies in the stratum $\mathcal{Q}(-1,-1,1,1)$, which we denote by $\mathcal{Q}(-1^2,1^2)$ in the following. Section~\ref{sec:decompo_traj} introduces the tools to prove the following theorem in \Cref{sec:ccl}, where $\gamma_T$ is a cycle corresponding to the trajectory, defined in Section~\ref{subsec:new_question}. The following theorem we will consider Masur and Oseledets generic surfaces (introduced in \Cref{subsec:generic_surf} and \Cref{subsec:Oseledets}, respectively).
\begin{restatable}{theorem}{BoundedIntersectionOnSurf}\label{thm:bounded_intersection}
Let $X\in\mathcal{Q}(-1^2,1^2)$ be a Masur and Oseledets generic surface. There exist a cycle ${w\in H_1(X,\RR)}$ and a constant $C>0$ such that for every $p\in X$, and every $T>0$, the (algebraic) number of intersections between $w$ and $\gamma_T$ is bounded by $C$:
$$
| \iota(w,\gamma_T) | < C.
$$
\end{restatable}
To get the bound, we decompose the cycle $\gamma_T$, as Forni did in~\cite{Forni} and after him Delecroix, Hubert and Lelièvre in~\cite{DHL} and we find a good cycle $w$ with the Oseledets theorem. To adapt the results of \cite{Forni, DHL} to quadratic differentials, we provide a proof using the fact that the points in the stratum are tori. We believe \Cref{thm:bounded_intersection} is in fact true in other strata, but constructing such a proof is an interesting problem outside of the scope of this paper. 
We prove that Theorem~\ref{thm:bounded_intersection} implies Theorem~\ref{thm:tb_in_wt} with Lemma~\ref{lem:transition} in \Cref{subsec:new_question}.

\section{Background and notation} 

\subsection{Translation surfaces} We recall three equivalent definitions for a translation surface:
\begin{enumerate}
\item
It is a finite collection of polygons, with sides identified by translation.
\item
It is a topological orientable surface $M$ of genus $g\geqslant 1$ where
\begin{itemize}
\item there exists a finite set of points $\Sigma =\{ s_1,\dots,s_n\} \subset M$, called the \emph{singularities},
\item there exists a tuple $k=(k_1,\dots,k_n)$ of integers,
\item the surface $M\setminus\Sigma$ is a manifold, and the maximal atlas contains only charts $M\setminus\Sigma \rightarrow \mathbb{C}$ that are translations,
\item for every singularity $s_i$, there exists a neighborhood $\mathcal{U}_i$ and a ramified covering $\pi_i:\mathbb{D} \rightarrow \mathcal{U}_i$ of degree $k_i$ such that the restriction of $\pi_i$ on a sheet is a chart.
\end{itemize}
\item It is a pair $(X,\omega)$ where $X$ is a compact Riemann surface and $\omega$ a holomorphic $1$-form on $X$.
\end{enumerate}
We refer for example to \cite[Section 2.5]{athreya-masur} for a proof of the equivalence of the definitions. For other surveys on the subject, the interested reader can look at \cite{SurveyYoccoz} or \cite{Zorich_flat_surf}. 
A translation is a map $z\mapsto z+c$ with $c\in\mathbb{C}$. 
If we replace the condition "by translation" in the first definition with the condition "by a map of the form $z\mapsto \pm z + c$", then we get a half-translation surface.
It changes in the same way the second point.
In the third point, it is equivalent to replace "holomorphic $1$-form" with "meromorphic quadratic form with simple poles".
Many properties of the translation surfaces hold for half-translation as well.

\subsection{Zippered rectangles}\label{subsec:background_zip-rect}
We first give a definition of interval exchange transformations.
\begin{dfn}[IET] An \emph{interval exchange transformation} (IET) on an interval $I$ is a injection from $I\setminus \Delta$ to $I$ that is a translation on each connected component of $I\setminus \Delta$, where $\Delta$ is a finite set.
\end{dfn}
Roughly speaking, an IET chopes its interval of definition into finitely many intervals, and rearranges them (by translation, without overlapping).
They arise as first return maps of a directional flow on a translation surface to a transversal segment.
The zippered rectangles have been first introduced by Veech \cite{Veech_mv} to build a translation surface $S$ from an IET, that we call $T$, such that $T$ is the first return map of the horizontal flow on $S$ to some transversal interval.
Conversely, from a given translation surface $S$ without either horizontal or vertical saddle connection (that is, a geodesic segment between two singularities), we can build a zippered rectangles representation. We quickly explain how and refer to \cite[Section 5]{SurveyYoccoz} for more detail. 

We pick a singularity of $S$.
We can always pick an outgoing horizontal geodesic segment $I$ from this singularity such that $I$ intersects every vertical leaf. If $S$ has no vertical saddle connection, then we can show that every vertical leaf is dense on $S$ and $I$ can be arbitrarily small.  If $S$ has no horizontal saddle connection, then we can pick $I$ as long as we want, we will never reach a singularity, and taking it long enough ensures to intersect every vertical leaf.
Moreover, we choose $I$ such that its other endpoint is either a singularity or on a singular leaf (ie. a leaf going to or out from a singularity). When we follow the vertical flow on $S$ from a point $x\in I$, it will either reach a singularity (in this case, $x$ is at the "boundary between" two rectangles) or it will eventually come back to $I$. Points come back to $I$ in cylinder, forming the rectangles. See \Cref{fig:zippered_rectangles}. 

\begin{figure}[h!]
\begin{tikzpicture}[line cap=round,line join=round,x=0.5cm,y=0.5cm]
\definecolor{colI}{rgb}{0.7,0,0}
\def \a{2}
\def \b{1}
\def \c{3}
\def \d{4}
\def \ya{4}
\def \yb{-0.5}
\def \yc{1}
\def \yd{-2}
\def \eps{0.05}

\draw [dashed] (0,0)--(\a,\ya)--(\a+\b,\ya+\yb)--(\a+\b+\c,\ya+\yb+\yc)--(\a+\b+\c+\d,\ya+\yb+\yc+\yd)--(\b+\c+\d,\yb+\yc+\yd)--(\c+\d,\yc+\yd)--(\d,\yd)--cycle;

\draw (0,0)--(0,\ya+\yb+\yc+\yd)--(\a+\a*\yb/\ya+\a*\yc/\ya+\a*\yd/\ya,\ya+\yb+\yc+\yd);
\draw (\a,\ya+\yb+\yc+\yd)--(\a,\ya-\yd-\yc)--(\a+\b-\eps,\ya-\yd-\yc)--(\a+\b,\ya+\yb)--(\a+\b+\eps,\ya+\yb-\yd)--(\a+\b+\c-\eps,\ya+\yb-\yd)--(\a+\b+\c,\ya+\yb+\yc)--(\a+\b+\c,\ya+\yb+\yc)--(\a+\b+\c+\d-\eps,\ya+\yb+\yc)--(\a+\b+\c+\d,\ya+\yb+\yc+\yd)--(\a+\b+\c+\d,0);

\draw [dotted] (\a,0)--(\a,\ya+\yb+\yc+\yd);
\draw [dotted] (\a+\b,0)--(\a+\b,\ya+\yb);
\draw [dotted] (\a+\b+\c,0)--(\a+\b+\c,\ya+\yb+\yc);

\draw [dotted] (\d,0)--(\d,\yd);
\draw [dotted] (\d+\c,0)--(\d+\c,\yd+\yc);
\draw [dotted] (\d+\b+\c,0)--(\d+\b+\c,\yd+\yb+\yc);

\fill (0,0) circle (1.5pt);
\fill (\a,\ya) circle (1.5pt);
\fill (\a+\b,\ya+\yb) circle (1.5pt);
\fill (\a+\b+\c,\ya+\yb+\yc) circle (1.5pt);
\fill (\a+\b+\c+\d,\ya+\yb+\yc+\yd) circle (1.5pt);

\draw [colI] (0,0)--(\a+\b+\c+\d,0);
\draw [colI] (\a+\a*\yb/\ya+\a*\yc/\ya+\a*\yd/\ya,\ya+\yb+\yc+\yd)--(\a,\ya+\yb+\yc+\yd);

\begin{footnotesize}
\draw [above left] (\a/2,\ya+\yb+\yc+\yd) node {$A$};
\draw [above] (\a+\b/2,\ya-\yd-\yc) node {$B$};
\draw [above] (\a+\b+\c/2,\ya+\yb-\yd) node {$C$};
\draw [above] (\a+\b+\c+\d/2,\ya+\yb+\yc) node {$D$};

\draw [below right] (\d+\c+\b+\a/2,0) node {$A$};
\draw [below] (\d+\c+\b/2,0) node {$B$};
\draw [below] (\d+\c/2,0) node {$C$};
\draw [below] (\d/2,0) node {$D$};
\draw [colI,above] (\a/2,0) node {$I$};
\end{footnotesize}
\end{tikzpicture}
\caption{Representation of a surface by zippered rectangles}\label{fig:zippered_rectangles}
{\small
The polygon in dashed lines defines the surface. The rectangles are build "above" the cross section $I$ (depicted in deep red). The labels $A$, $B$, $C$, and $D$ show where the rectangles come back to $I$.}
\end{figure}
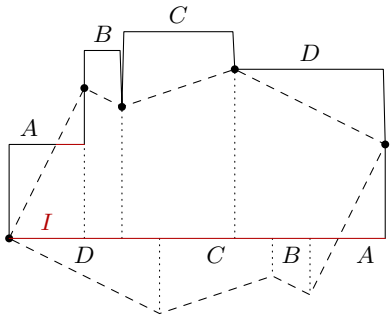

\subsection{Teichmüller spaces and moduli spaces}\label{subsec:background_teich}

Let $S$ be a topological surface of genus $g\geqslant 1$. The \emph{Teichmüller space of Riemann surfaces} of genus $g$ is
$$
\mathscr{T}_g = \mathscr{T}(S) = \{(X,f) \, | \, X \text{ a Riemann surface}, f: S \rightarrow X \text{ an homeomorphism} \} /\backsim
$$
where two pairs
$(X_1,f_1)$ and $(X_2,f_2)$ are equivalent if there exists a biholomorphic map $h :X_1\rightarrow X_2$ such that $f_2^{-1}\circ h \circ f_1$ is homotopic to the identity.

The \emph{mapping class group} $\Gamma(S)$ is the group of the orientation-preserving diffeomorphisms on $S$ quotiented by the group of those which are homotopic to the identity.
$$
\Gamma(S) = \text{Diff}^+ (S) / \text{Diff}_0^+ (S).
$$

The group $\Gamma(S)$ is discrete and finitely generated. It acts on $\mathscr{T}(S)$ properly discontinuously. The \emph{moduli space of Riemann surfaces} of genus $g$ is the quotient
$$
\mathscr{M}_g = \mathscr{T}_g/\Gamma(S). 
$$

A half-translation surface is a Riemann surface together with a meromorphic quadratic form with simple poles. Thus, 
the Teichmüller and moduli spaces of translation surfaces also contains the data of the quadratic form.
The Teichmüller space $\mathscr{T}_g$ admits a vector bundle, denoted by $\mathscr{T}\Omega_g$, whose fiber at the point $X$ is the set 
$$\Omega(X)=\{(X,\omega) \, | \, \omega \text{ a  meromorphic quadratic form with simple poles defined on } X\}$$
 of pairs of $X$ and a quadratic form $\omega$ defined on $X$.	
The group $\Gamma(S)$ acts on $\mathscr{T}\Omega_g$ as well.
The \emph{Teichmüller space of half-translation surfaces} is $\mathscr{T}\Omega_g$ and the \emph{moduli space of translation surfaces} is the quotient
$$
\Omega_g = \mathscr{T}\Omega_g/\Gamma(S)
$$
which is a vector bundle over $\mathscr{M}_g$ whose fiber in point $X$ is $\Omega(X)$.
In the point of view of the first given definition of half-translation surfaces, 
the moduli space $\Omega_g$ is the set of polygons with identified sides modulo cut and paste operations.

The group $SL_2(\RR)$ acts naturally on the moduli space as follows.
Let $M$ be a half-translation surface given by a polygon $P$ with pairs of sides identified by translation or by $z\mapsto\pm z +c$. 
Let $g\in SL_2(\RR)$. Then, $g.P$ is a polygon and the identifications are still by "half-translation", since $SL_2(\RR)$ preserves the parallelism.
Thus $g.P$ defines a half-translation surface, which is $g.M$.
The \emph{Teichmüller flow} is defined as the restriction of the action of $SL_2(\RR)$ to the diagonal subgroup:
$$ g_t=\left(\begin{array}{cc}
e^t & 0 \\
0 & e^{-t}
\end{array}\right)  .$$
It corresponds to the geodesic flow on the moduli space.

\subsection{The Kontsevich-Zorich cocycle (geometric point of view)}
\label{sec:background_KZ_geometric}

To avoid confusion between the genus and the Teichmüller flow, we now simply write $\mathscr{T}$ and $ \mathscr{M}$ for the Teichmüller space and the moduli space of translation surfaces of genus $g$, $g$ being fixed.
Take a half-translation surface $M$ (of genus $g$), defined by a polygon $P$ with identified sides. Take also a tuple of closed curves  $(\zeta_1,\dots,\zeta_d)$ such that the tuple of the cycles $([\zeta_1],\dots,[\zeta_d])$ forms a basis of the homology $H_1(M,\RR)$. 
Consider the orbit $(g_t. M)_{t\in\RR}$  of $M$ along the Teichmüller flow. For a time $t>0$, the polygon $g_t.P$ that defines the surface $g_t .M$ is stretched in the horizontal direction. We can cut and paste it into another polygon that still defines $g_t.M$. For some\footnote{the time $t$ can be arbitrarily large by ergodicity of the Teichmüller flow in $\mathscr{M}$} well chosen time $t$, we can obtain a polygon whose shape is approximately the one of $P$, which allows to identify $H_1(g_t.M,\RR)$ and $H_1(M,\RR)$.
The Kontsevich--Zorich cocycle expresses the cycles $([g_t.\zeta_1],\dots,[g_t.\zeta_d])$ in the basis $([\zeta_1],\dots,[\zeta_d])$.

We now give a more formal definition, following the nice presentation given in \cite{FSU}.
Let $\pi : \mathscr{T} \rightarrow \mathscr{M}$ and $\Delta\subset\mathscr{T}$ a fundamental domain for the action of $\Gamma$. 
Let $t\in\RR$, $x\in\mathscr{M}$, and $\tilde{x}$ be any point in the fiber $\pi^{-1}(x)$ of $x$. There is an only element  $\psi\in\Gamma$ such that $\psi(g_t . \tilde{x}) \in \Delta $.
The Kontsevich--Zorich cocycle is
$$
A_{KZ}:\begin{array}{ccc}
\RR \times \mathscr{M}(M) &\rightarrow &\GL(H_1(M,\RR) \\
(t,x) &\mapsto& (\psi_{g_t,\tilde{x}})_* 
\end{array}.
$$
 
\subsection{The Oseledets theorem}
\label{subsec:Oseledets}

The Oseledets theorem is a multiplicative version of the ergodic theorem. We consider a dynamical system $(E,\mu,\varphi)$, that is $(E,\mu)$ a probability space and $\varphi:E\rightarrow E$ a $\mu$-invariant transformation, and a cocycle 

\begin{equation}\label{eq:cocycle}
\begin{split}
A: \begin{array}{ccc}
X & \longrightarrow & M_d (\RR) \\
x & \mapsto & A_x
\end{array} .
\end{split}
\end{equation}

We are interested in the growth of vectors under the action of the product of matrices $A_x^{(n)} =  A_{\varphi^{n-1}(x)}...A_{\varphi(x)}A_{x}$.

\begin{thm}[Oseledets]\label{thm:Oseledets}
Let $E$ be a topological space with a probability measure $\mu$. Let $\varphi : E\longrightarrow E$ be an ergodic function.
Let $A$ be a cocycle as in \eqref{eq:cocycle}.
Assume that the cocycle is log-integrable: $\int_E \log |||A_x||| \text{d}\mu < \infty$.

There exist $\theta_1>\theta_2>...>\theta_k$, called \emph{Lyapunov exponents}, such that for almost every $x\in E$, there exists a flag $ \RR^d = \mathcal{H}_k^{(x)} \supset \mathcal{H}_{k-1}^{(x)} \supset ... \supset \mathcal{H}_1^{(x)} $ such that:

$$
A_x \mathcal{H}_j^{(x)} = \mathcal{H}_j^{(\varphi(x))}
$$

and for all $v\in\mathcal{H}_j\setminus\mathcal{H}_{j-1}$, 
$$
 \frac{1}{n}\log || A_x^{(n)} v || \underset{n\rightarrow\infty}{\longrightarrow} \theta_j
$$
where $ A_x^{(n)} = A_{\varphi^{n-1}(x)}...A_{\varphi(x)}A_{x}$.
\end{thm}

The set of surfaces for which the conclusion of \Cref{thm:Oseledets} holds are called \emph{Oseledets generic} surfaces. In this article, we apply the Oseledets theorem in $E=\mathcal{Q}(-1^2,1^2)$ the stratum of half-translation surfaces with four singularities (two with conical angle $\pi$ and two with conical angle $2\pi$). It is endowed with the Masur-Veech measure $\mu$, which is ergodic for the Teichmüller flow \cite{Masur_mv,Veech_mv}.
We will apply the theorem (in \Cref{subsec:ccl_Oseledets}) to a discrete version of the Kontsevich-Zorich cocycle (introduced in \Cref{sec:background_KZ_geometric}), which is log-integrable \cite{Forni}.

\subsection{Notation}
In \Cref{tab:notation} we provide a table of notations for the convenience of the reader.

\begin{table}[h!]
    \centering
    \begin{tabular}{lrl}
\textbf{Setting} & $\Lambda$& lattice  \\
       & $a,b$ & size of the rectangular obstacle \\
        &$\theta$ & inclination of the obstacle \\
\textbf{Surfaces} &        $S$& infinite half-translation surface associated to $(\Lambda,a,b,\theta)$\\
   (see Sections     &$X=S/\Lambda$& finite half-translation surface associated to $(\Lambda,a,b,\theta)$\\
    \ref{subsec:equiv-traj}, \ref{subsec:finite_HTS} and \ref{subsec:generic_surf})      &$\mathcal{I} \subset \RR^2$& set of periodic slits without identification \\
        &$\mathcal{P}  \subset \RR^2 $& set of periodic parallelogram-shape cuts without identification\\
\textbf{Projections} & $\pi:$ & $S\rightarrow X$ \\
      (see Sections      &$\pi_S:$ & $\RR^2\setminus\mathcal{P}\rightarrow S$\\
    \ref{subsec:equiv-traj} and \ref{subsec:finite_HTS})      &$\pi_X:$ & $\RR^2\setminus\mathcal{P}\rightarrow X$\\
\textbf{Set of surfaces}
    &$\mathcal{A} $&  $=\{X_{(\Lambda,a,b,\theta)} \, | \, (\Lambda,a,b,\theta) \text{ admissible parameters} \}$\\
  (see Section \ref{subsec:generic_surf})    &$\mathcal{X} $& $=\{ X_{(\Lambda,a,b,\theta)}  \, | \, (\Lambda,a,b,\theta) \text{ possibly non admissible parameters} \}$\\
      &\multirow{2}{5em}{$\mathcal{S}=\mathcal{S}_1\sqcup\mathcal{S}_2 $}& set of surfaces that are quotient of $\RR^2$ by $\Lambda$ with\\
      &&parallelogram-shape cuts with identification \\
\textbf{Basis}        &$(e_1,e_2)$& basis of $\Lambda$ \\
    (see Section \ref{subsec:new_question})       &  $(\xi_1,\xi_2)$& closed geodesic curves on $X$ induced by $(e_1,e_2)$,  form a basis of $\pi_1(X)$\\
       & $(\zeta_1,\zeta_2)$& homology classes of $(\xi_1,\xi_2)$, form a basis of $H_1(X,\RR)$\\
\textbf{Flows}    &    $\Psi$ & flow of the wind-tree tiling billiards \\
        &$\Phi$ & vertical flow on $X$\\
\textbf{Curves}&        $\varphi_s^t(x)$ & part of the vertical foliation of $x$ between $\Phi_s(x)$ and $\Phi_t(x)$ (on $X$) \\
   (see Section \ref{subsec:new_question})        &$\gamma_s^t(x)$ & closed curve obtained by closing $\varphi_s^t(x)$ without crossing $\xi_1$ nor $\xi_2$ \\ 
        &$\gamma_t(x)$ & $=\gamma_0^t(x)$\\
\textbf{Other objects}   &     $\mathcal{U}$ & small neighborhood of $X$ in $\mathcal{Q}(-1^2,1^2)$  \\
  &    \multirow{2}{0.25cm}{$I$} & segment on $X$, transverse to the vertical flow,\\
  &&on which we build a zippered rectangles representation \\
    \end{tabular}
    \caption{Notations}
    \label{tab:notation}
\end{table}

\section{The corresponding surface} \label{sec:corresponding_surf_and_question}
With each wind-tree tiling billiards with parameters $(\Lambda, a, b, \theta)$, we associate a (finite) half-translation surface in the stratum $\mathcal{Q}(-1^2,1^2)$ such that the geodesic flow corresponds to the wind-tree tiling billiards flow. We first build an infinite half-translation surface related to our system in \Cref{subsec:infinite_HTS}. In \Cref{subsec:equiv-traj}, we explain how both flows relate to each other. In \Cref{subsec:finite_HTS}, we quotient the infinite surface to get a finite surface.
The set $\mathcal{A}$ of surfaces that we obtain is an open set 
of a 7 real dimensional subset $\mathcal{X}$ of the stratum $\mathcal{Q}(-1^2,1^2)$ (which has 8 real dimensions). However, in \Cref{subsec:generic_surf} we show that almost all surfaces in $\mathcal{X}$ are generic in $\mathcal{Q}(-1^2,1^2)$. Finally, in \Cref{subsec:new_question}, we transpose the question of being trapped in a strip in the plane to a question of algebraic intersection numbers in the associated surface $X\in\mathcal{Q}(-1^2,1^2)$.

\subsection{From foliations to an appropriate slit surface}\label{subsec:infinite_HTS}
We first give the idea behind the construction. Start with a fixed rectangular obstacle $\mathcal{R} \subseteq \mathbb{R}^2$ with side lengths $a,b$ and angle $\theta$ from the horizontal. Consider a trajectory moving in the vertical direction. There are two ways for a trajectory to pass through $\mathcal{R}$: either it goes from one side to the opposite side, or it goes from one side to an adjacent side (and in this case, its direction is reversed), as seen on the left of Figure~\ref{fig:ways_to_cross}. If a trajectory hits a corner of $\mathcal{R}$ there is no unique choice of path, so we call this a singular leaf. 

\begin{figure}[h!]
\begin{center}
\begin{subfigure}[b]{0.44\textwidth}
\def\colsg{black}
\def\eps{0.03}
\def\xa{3.9}
\def\ya{1.3}
\def\xb{-1}
\def\yb{3}
\def\xp{-0.72}
\def\xpp{3.62}
\begin{tikzpicture}[line cap=round,line join=round,>=stealth,x=1.2cm,y=1.2cm]
\clip(-1.2,-0.5) rectangle (4.1,4.5);
\draw[color=colR,opacity=0.5,fill=colR,fill opacity=0.15] (0,0) -- (3.9,1.3) -- (2.9,4.3) -- (-1,3) -- cycle;

% feuilles singulières
\draw [densely dotted,\colsg] (-1,3) -- (-1,4.93);
\draw [densely dotted,\colsg] (-1,-0.98) -- (-1,3) -- (1,0.33) -- (1,-0.98);
\draw [densely dotted,\colsg] (3.9,1.3) -- (3.9,4.93);
\draw [densely dotted,\colsg] (3.9,-0.98) -- (3.9,1.3)-- (1.9,3.97) -- (1.9,4.93);

% traj middle 1
\draw  [thick,dashed,colmid] (1.54,-0.5) -- (1.54,0.51) -- (-0.46,3.18) -- (-0.46,4.6);
\draw  [thick,colmid,->] (1.54,-0.8) -- (1.54,-0.9);
\draw  [thick,colmid,->] (-0.46,4) -- (-0.46,4.1);

% traj middle 2
\draw [thick,dashed,colmid] (1.13,4.76)-- (1.13,3.71) -- (3.13,1.04)-- (3.13,-1.56);
\draw  [thick,colmid,->] (3.13,-0.8) -- (3.13,-0.9);
\draw  [thick,colmid,->] (1.13,4) -- (1.13,4.1);

% traj left
\draw [thick,dotted,colleft] (-0.58,-0.5)-- (-0.58,1.75) -- (0.58,0.19)-- (0.58,-0.5);
\draw [thick,colleft,->,>=latex] (-0.58,0.8)-- (-0.58,0.9);
\draw [thick,colleft,->,>=latex] (0.58,0)-- (0.58,-0.1);

% traj right
\draw [thick,dash pattern=on 10pt off 3pt on 1pt off 2pt on 1pt off 2pt on 1pt off 3pt,colright] (3.77,4.5) -- (3.77,1.69) -- (2.03,4.01)-- (2.03,4.84);
\draw [thick,colright,->,>=Straight Barb] (3.77,3.5) -- (3.77,3.4);
\draw [thick,colright,->,>=Straight Barb] (2.03,4.2) -- (2.03,4.3);

\end{tikzpicture}
    \caption{Different ways to go through the \mbox{rectangular} obstacle $\mathcal{R}$}
\end{subfigure}
\hfill
\begin{subfigure}[b]{0.44\textwidth}
\def\colsg{black}
\def\eps{0.06}
\def\xa{3.9}
\def\ya{1.3}
\def\xb{-1}
\def\yb{3}
\def\xp{-0.72}
\def\xpp{3.62}
\begin{tikzpicture}[line cap=round,line join=round,>=stealth,x=1.2cm,y=1.2cm]
\clip(-1.2,-0.5) rectangle (4.1,4.5);
\draw[color=colR,opacity=0.5,fill=colR,fill opacity=0.15] (0,0) -- (3.9,1.3) -- (2.9,4.3) -- (-1,3) -- cycle;

% feuilles singulières
\draw [densely dotted,\colsg] (-1,-0.5) -- (-1,4.8);
\draw [densely dotted,\colsg]  (1,0.33) -- (1,-0.5);
\draw [densely dotted,\colsg] (1,2.30612-\eps) -- (1,0.33);
\draw [densely dotted,\colsg] (3.9,-0.5) -- (3.9,4.8);
\draw [densely dotted,\colsg] (1.9,1.99388+\eps) -- (1.9,4.8); 

% slit
\draw [line cap=rect,line join=rect,color=colleft,shift={(0,-\eps)}, line width= 2pt] (-1,3) -- (1,2.30612);
\draw [line width=0.75pt] (0,2.65-\eps-0.05)--(0,2.65-\eps+0.03);
\draw [line cap=rect,line join=rect,color=colright,shift={(0,+\eps)}, line width= 2pt] (1.9,1.99388) -- (3.9,1.3);
\draw [line width=0.75pt] (2.9,1.65+\eps-0.03)--(2.9,1.65+\eps+0.05);
\draw [line cap=rect,line join=rect,color=colmid,shift={(0,+\eps)}, line width= 2pt] (-1,3) -- (1.9,1.99388);
\draw [line cap=rect,line join=rect,color=colmid,shift={(0,-\eps)}, line width= 2pt] (1,2.30612) -- (3.9,1.3);

% traj left
\draw [thick,dotted,colleft,->,>=latex] (-0.58,-0.5)-- (-0.58,0.95);
\draw [thick,dotted,colleft] (-0.58,0.9)-- (-0.58,2.825);
\draw [thick,dotted,colleft,->,>=latex] (0.58,2.4)-- (0.58,0.75);
\draw [thick,dotted,colleft] (0.58,1)-- (0.58,-0.5);

% traj middle 1
\draw  [thick,dashed,colmid,->] (1.54,-0.5) -- (1.54,0.9);
\draw  [thick,dashed,colmid] (1.54,0.8) -- (1.54,2.05);
\draw  [thick,dashed,colmid,->] (-0.46,2.85) -- (-0.46,3.5);
\draw  [thick,dashed,colmid] (-0.46,3.4) -- (-0.46,4.5);

% traj middle 2
\draw [thick,dashed,colmid,->]  (3.13,-0.5)--(3.13,0.9);
\draw [thick,dashed,colmid]  (3.13,0.8)--(3.13,1.45);
\draw [thick,dashed,colmid,->] (1.13,2.3) -- (1.13,3.5);
\draw [thick,dashed,colmid] (1.13,3.4) -- (1.13,4.5);

% traj right 
\draw [thick,dash pattern=on 10pt off 3pt on 1pt off 2pt on 1pt off 2pt on 1pt off 3pt,colright,->,>=Straight Barb] (3.77,4.8) -- (3.77,3.4);
\draw [thick,dash pattern=on 10pt off 3pt on 1pt off 2pt on 1pt off 2pt on 1pt off 3pt,colright] (3.77,3.5) -- (3.77,1.375);

\draw [thick,dash pattern=on 10pt off 3pt on 1pt off 2pt on 1pt off 2pt on 1pt off 3pt,colright,->,>=Straight Barb] (2.03,2)-- (2.03,3.5);
\draw [thick,dash pattern=on 10pt off 3pt on 1pt off 2pt on 1pt off 2pt on 1pt off 3pt,colright] (2.03,3.4)-- (2.03,4.8);
\end{tikzpicture}
    \caption{Trajectories when the obstacle $\mathcal{R}$ is \mbox{replaced} by a slit}
\end{subfigure}
\caption{
Replacing the obstacle by a well-chosen slit preserves the trajectory outside of $\mathcal{R}$.
}
\label{fig:ways_to_cross}
\end{center}

{\small
The orange dashed lines are trajectories translated when crossing $\mathcal{R}$, and the (deep and light) blue dotted lines are reversed trajectories around a corner. 
The thin black dotted lines correspond to the singular leaves of the foliation.}
\end{figure}
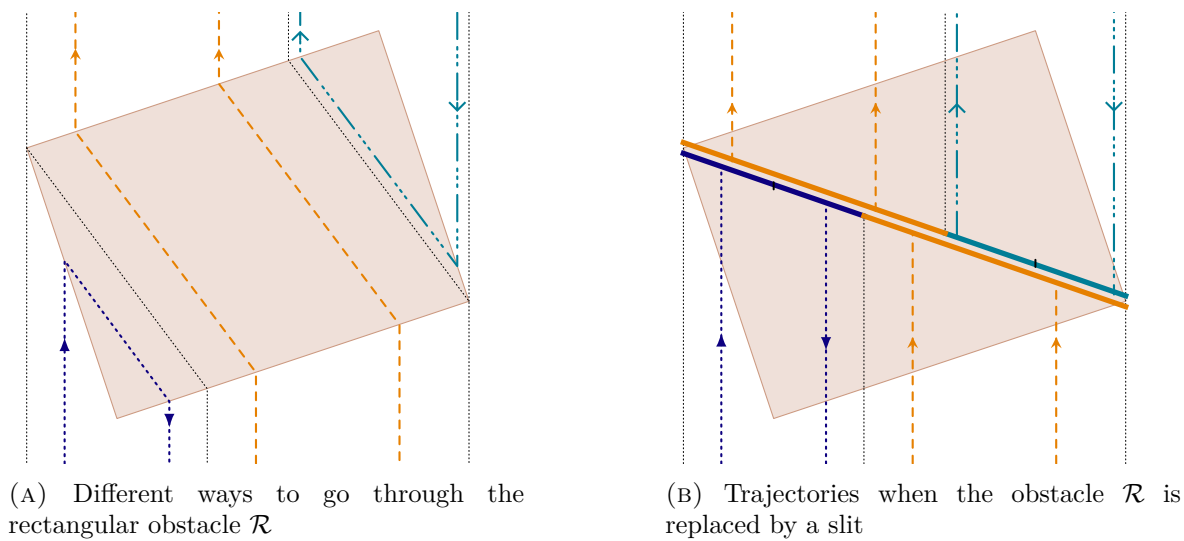 

Since any vertical trajectory has the same three options when crossing an obstacle, we will replace the rectangle by a slit with identifications so that the trajectory outside of $\mathcal{R}$ remains the same as shown on the right in Figure~\ref{fig:ways_to_cross}. To do this, we first let $\mathfrak{c}$ denote the corner of $\mathcal{R}$ with the smallest $y$-coordinate. Such a choice of $\mathfrak{c}$ is unique since we excluded the case where $\theta = 0$ with horizontal and vertical sides. Now consider $\mathfrak{d}$ to be the diagonal of the rectangle that does not contain $\mathfrak{c}$. Note that $\mathfrak{d}$ cannot be vertical (otherwise it should contain the lowest corner of the rectangle $\mathfrak{c}$). In other words, the diagonal $\mathfrak{d}$ is transverse to the vertical direction, so we can encode the way the trajectory crosses the obstacle on this diagonal. This diagonal with side identifications is what we call a slit. Constructing the same slit each copy of $\mathcal{R}$ on the lattice $\Lambda$ defines an infinite half-translation surface. In Section~\ref{subsec:finite_HTS} we give a compact quotient, but we spend the remainder of this section giving exact computations and a proof that the trajectories in the plane with the obstacles and in the infinite half-translation surface are equivalent.

In this new setting, the trajectories of our problem become geodesics in an infinite half-translation surface. The surface in question is defined by Definition~\ref{def:surf_S} and Lemma~\ref{lem:equivalent_traj} states that the geodesics in the new surface are \emph{equivalent} (see Definition~\ref{def:equiv_traj}) to the trajectories of the system $W(\Lambda,a,b,\theta)$.

\begin{figure}[h!]
\begin{center}
\begin{subfigure}[b]{0.475\textwidth}
\centering
\definecolor{colangle}{rgb}{0.7,0,0}
\definecolor{coldel}{rgb}{0,0,0}
\def\eps{0.05}%1pt
\def\xa{3.9}
\def\ya{1.3}
\def\xb{-1}
\def\yb{3}
\def\xp{-0.72}
\def\xpp{3.62}
\def\S{-20.86} %approx sloap of the diagonal (in degrees)
\begin{tikzpicture}[>=stealth,x=.750cm,y=.750cm]
\clip(\xb-4,-1.2) rectangle (\xa+0.5,4.5);
\draw[color=colR,opacity=0.5,fill=colR,fill opacity=0.15] (0,0) -- (3.9,1.3) -- (2.9,4.3) -- (-1,3) -- cycle;

% feuilles singulières
\draw [densely dotted,coldel] (\xb,-0.15) -- (-1,\yb);
\draw [densely dotted,coldel] (\xa,-0.15) -- (3.9,\ya);

% delim horizontales
\draw [densely dotted,coldel] (\xb-0.5,3) -- (\xb,3);
\draw [densely dotted,coldel] (\xb-0.5,2.30612-\eps) -- (1,2.30612-\eps);
\draw [densely dotted,coldel] (\xb-0.5,1.3) -- (3.9,1.3);

\begin{footnotesize}
\draw [color=colangle,fill=colangle,fill opacity=0.1] (0,0) -- (0.4,0) arc (0:18.43:0.4) -- cycle;
\draw [colangle,right] (0.4,0) node {$\theta$};
\draw [<->,shift={(0.06*\xb,0.06*\yb)}] (\xb,\yb)--(\xb+\xa,\yb+\ya);
\draw [above=4] (\xa/2+\xb,\ya/2+\yb) node {$a$};
\draw [<->,shift={(0.05*\xa,0.05*\ya)}] (\xa,\ya)--(\xb+\xa,\yb+\ya);
\draw [right=4] (\xa+\xb/2,\ya+\yb/2) node {$b$};
\draw [<->,shift={(0,-0.2)}] (\xb,0)--(0,0);
\draw [below=4] (\xb/2,0) node {$b\sin\theta$};
\draw [<->,shift={(0,-0.2)}] (0,0)--(\xa,0);
\draw [below=4] (\xa/2,0) node {$a\cos\theta$};

\draw [<->,shift={(-0.6,-\eps)}] (\xb,\ya)--(\xb,2.30612);
\draw (\xb-2.4,2.6-\eps) node {$\frac{2b\sin\theta}{\tan(\theta+\arctan(a/b))}$};
\draw [<->,shift={(-0.6,-\eps)}] (\xb,2.30612)--(\xb,\yb);
\draw (\xb-2.4,1.7-\eps) node {$\frac{a\cos\theta-b\sin\theta}{\tan(\theta+\arctan(a/b))}$};

\draw [<->,shift={(-0.03*\ya+0.03*\yb,0.03*\xa-0.03*\xb)}] 
(-1,3) -- (1.9,1.99388); %(\xa,\ya) ++ (2*\xb, 2*\xb*tan{\S}) -- (\xb,\yb) ;
\draw [shift={(0,0.36)}] (\xb+\xb/2+\xa/2,\xb*tan{\S}+\yb/2+\ya/2)node {$x$};

\draw [<->,shift={(-0.03*\ya+0.03*\yb,0.03*\xa-0.03*\xb)}] (1.9,1.99388) -- (3.9,1.3); 
\draw [shift={(0,0.36)}]   (\xa+\xb,\ya+\xb*tan{\S}) node {$y$};

\draw [color=colangle,fill=colangle,fill opacity=0.1,shift={(0,0)}] (\xb,\yb) --++ (0,-0.4) arc (-90:-20.86:0.4) -- cycle;
\draw [colangle,below=1.5, right=2] (\xb,\yb-0.4) node {$\eta$};
\end{footnotesize}

% slit
\draw [color=colleft,shift={(0,-\eps)}, line width= 1pt] (-1,3) -- (1,2.30612);
\draw [line width=0.75pt] (0,2.65-\eps-0.05)--(0,2.65-\eps+0.03);
\draw [color=colright,shift={(0,+\eps)}, line width= 1pt] (1.9,1.99388) -- (3.9,1.3);
\draw [line width=0.75pt] (2.9,1.65+\eps-0.03)--(2.9,1.65+\eps+0.05);
\draw [color=colmid,shift={(0,+\eps)}, line width= 1pt] (-1,3) -- (1.9,1.99388);
\draw [color=colmid,shift={(0,-\eps)}, line width= 1pt] (1,2.30612) -- (3.9,1.3);
\end{tikzpicture}
\caption{Case when $b\sin\theta<a\cos\theta$}\label{fig:rect_to_slit_first_case}
\end{subfigure}
\hfill
\begin{subfigure}[b]{0.475\textwidth}
\centering
\definecolor{colangle}{rgb}{0.7,0,0}
\definecolor{coldel}{rgb}{0,0,0}
\def\eps{0.05}
\def\xa{0.5}
\def\ya{3}
\def\xb{-6}
\def\yb{1}
\def\xp{-0.72}
\def\xpp{3.62}
\def\T{80.538} %theta is approx. equal to 80.538 degrees
\def\E{107.103} %eta is approx equal to 107.103
\def\S{17.103} %approx. sloape of the diagonal
\begin{tikzpicture}[>=stealth,x=.750cm,y=.750cm]
\clip(\xb-0.5,-0.65) rectangle (\xa+4,4.5);
\draw[color=colR,opacity=0.5,fill=colR,fill opacity=0.15] (0,0) --++ (\xa,\ya) --++ (\xb,\yb) --++ (-\xa,-\ya) -- cycle;

% feuilles singulières
\draw [densely dotted,coldel] (\xb,-0.15) -- (\xb,\yb);
\draw [densely dotted,coldel] (\xa,-0.15) -- (\xa,\ya);

% delim horizontales
\draw [densely dotted,coldel] (\xa,\ya) --++ (0.5,0);
\draw [densely dotted,coldel]  (-\xa,\ya-2*\xa*tan{\S}-\eps) -- (\xa+0.5,\ya-2*\xa*tan{\S}-\eps);
\draw [densely dotted,coldel] (\xb,\yb-\eps) -- (\xa+0.5,\yb-\eps);

\begin{footnotesize}
\draw [color=colangle,fill=colangle,fill opacity=0.1] (0,0) -- (0.4,0) arc (0:\T:0.4) -- cycle;
\draw [colangle,above right] (0.4,0) node {$\theta$};

\draw [<->,shift={(0.03*\xb,0.03*\yb)}] (\xb,\yb)--(\xb+\xa,\yb+\ya);
\draw [left=4] (\xa/2+\xb,\ya/2+\yb) node {$a$};
\draw [<->,shift={(0.08*\xa,0.08*\ya)}] (\xa,\ya)--(\xb+\xa,\yb+\ya);
\draw [above=4] (\xa+\xb/2,\ya+\yb/2) node {$b$};
\draw [<->,shift={(0,-0.2)}] (\xb,0)--(0,0);
\draw [below=4] (\xb/2,0) node {$b\sin\theta$};
\draw [<->,shift={(0,-0.2)}] (0,0)--(\xa,0);
\draw [below=4] (\xa/2,0) node {$a\cos\theta$};

\draw [<->,shift={(0.6,0)}] (\xa,\ya)--(\xa,\ya-2*\xa*tan{\S}-\eps);
\draw [shift={(2.35,0)}]  (\xa,\ya-\xa*tan{\S}-\eps) node {$\frac{2a\cos\theta}{\tan(\theta+\arctan(a/b))}$};
\draw [<->,shift={(0.6,0)}] (\xa,\ya-2*\xa*tan{\S}-\eps)--(\xa,\yb);
\draw [shift={(2.35,-0.2)}]  (\xa,\yb/2+\ya/2-\xa*tan{\S}) node {$\frac{b\sin\theta-a\cos\theta}{\tan(\theta+\arctan(a/b))}$};

\draw [<->,shift={(-0.025*\ya+0.025*\yb,0.025*\xa-0.025*\xb)}] (\xb,\yb) --++ (2*\xa, 2*\xa*tan{\S}) ;
\draw [shift={(0,0.36)}]  (\xb+\xa,\yb+\xa*tan{\S}) node {$y$};

\draw [<->,shift={(-0.025*\ya+0.025*\yb,0.025*\xa-0.025*\xb)}] (\xa,\ya) --++ (\xa+\xb, -tan{\S}*5.5);
\draw [shift={(0,0.36)}]  (\xa+\xa/2+\xb/2,\ya-5.5*tan{\S}/2) node {$x$};

\draw [color=colangle,fill=colangle,fill opacity=0.1] (\xb,\yb) --++ (0,-0.4) arc (-90:\E-90:0.4) -- cycle;
\draw [colangle,below right] (\xb,\yb-0.4) node {$\eta$};
\end{footnotesize}

% slit
\draw [color=colleft,shift={(0,+\eps)}, line width= 1pt] (\xb,\yb) --++ (2*\xa, 2*\xa*tan{\S});
\draw [line width=0.75pt] (\xb,\yb) ++ (\xa,\xa*tan{\S}) --++ (0,0.1);
\draw [color=colright,shift={(0,-\eps)}, line width= 1pt] (\xb,\yb) ++ (-\xb-\xa, tan{\S}*5.5) -- (\xa,\ya);
\draw [line width=0.75pt] (\xa,\ya) ++ (-\xa,-\xa*tan{\S}) --++ (0,-0.1);

\draw [color=colmid,shift={(0,-\eps)}, line width= 1pt] (\xb,\yb) --++ (-\xb-\xa, tan{\S}*5.5); 
\draw [color=colmid,shift={(0,+\eps)}, line width= 1pt] 
(\xa,\ya) --++ (\xb+\xa, -tan{\S}*5.5);

\end{tikzpicture}
\caption{Case when $b\sin\theta>a\cos\theta$}\label{fig:rect_to_slit_second_case}
\end{subfigure}
\caption{From a rectangular obstacle to a slit with identifications }
\label{fig:rect_to_slit}
\end{center}
\end{figure}
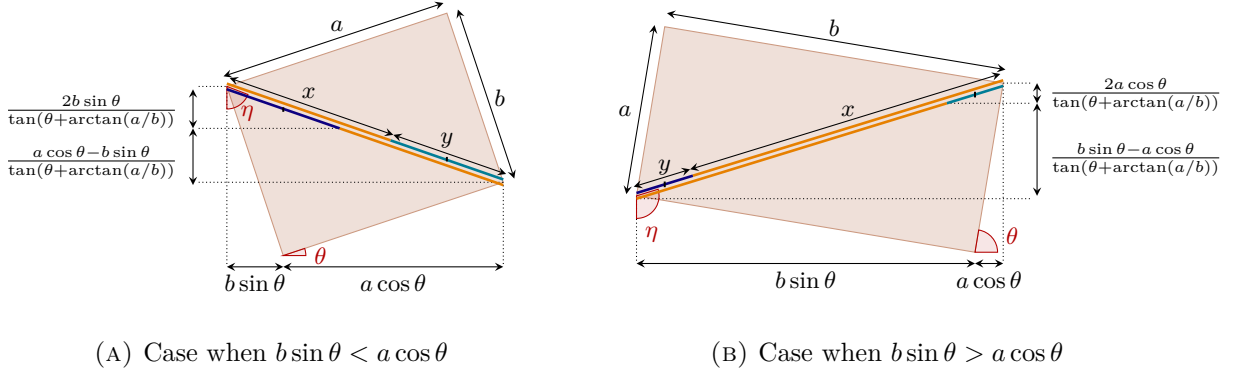

\begin{dfn}\label{def:surf_S}
    Let $S_{(\Lambda, a, b, \theta)}$ denote the infinite half-translation surface defined as follows. Centered on each lattice points $\Lambda$ in the plane $\RR^2$, there is a slit making an angle $\eta=\theta+\arctan(a/b)$ with the {\textbf{vertical}}, and of length $\sqrt{a^2+b^2}$. 
    
    The side identifications for the slit are given by one of the three cases.
    \begin{enumerate}
        \item\label{item:def_S_case1} If $b\sin(\theta) < a\cos(\theta)$, then the top left (of size $x$) of the slit is identified by translation with the bottom right of the slit, the top right (of size $y$) of the slit is identified to itself with a rotation, and so is the bottom left of the slit see \Cref{fig:rect_to_slit_first_case}, where the lengths are \begin{align*}
            x&=(a\cos\theta-b\sin\theta) \sqrt{1 + \left(\frac{1}{\tan(\theta+\arctan(a/b))}\right)^2}\\
            y&=(2b\sin\theta) \sqrt{1 + \left(\frac{1}{\tan(\theta+\arctan(a/b))}\right)^2}.
        \end{align*}
        \item\label{item:def_S_case2}  If $b\sin(\theta) > a\cos(\theta)$, then the top right (of size $x$) of the slit is identified by translation with the bottom left of the slit, the top left (of size $y$)  of the slit is identified to itself with a rotation, and so is the bottom right of the slit see \Cref{fig:rect_to_slit_second_case}, where the lengths are \begin{align*}
            x&=(b\sin\theta-a\cos\theta) \sqrt{1 + \left(\frac{1}{\tan(\theta+\arctan(a/b))}\right)^2}\\
            y&=(2a\cos\theta) \sqrt{1 + \left(\frac{1}{\tan(\theta+\arctan(a/b))}\right)^2}.
        \end{align*}
        \item\label{item:def_S_case3}  If $a\cos(\theta) = b\sin(\theta)$, then the top of the slit is identified to itself with a rotation, and so is the bottom of the slit (there is no part identified by translation).
    \end{enumerate}
\end{dfn}

Case (\ref{item:def_S_case1}) (resp. (\ref{item:def_S_case2})) holds when the singular leaves go through the interior of a side of length $a$ (resp. $b$).
Moreover, when case (\ref{item:def_S_case1}) (resp. (\ref{item:def_S_case2})) holds, every trajectory entering the obstacle by a side of length $b$ (resp. $a$) goes out through an adjacent side.
The special case when the slit is only identified by translation (and hence $S_{(\Lambda,a,b,\theta)}$ is a flat plane, without conical singularity) corresponds to the excluded case when $\theta=0$. 

Finally, let us highlight that the map $(\Lambda,a,b,\theta)\mapsto S_{(\Lambda,a,b,\theta)}$ is a bijection from the set of (non necessarily admissible) parameters $(\Lambda,a,b,\theta)$ to the set of infinite half-translation surfaces obtained as a plane with periodical arrangement of same slits, as in \Cref{fig:infinite_surf}.
\begin{figure}[h!]
\centering
\def\xk{1.2}
\def\yk{0.5}
\def\xl{-0.1}
\def\yl{1}
\def\lx{0.6}
\def\ly{0.4}
\begin{tikzpicture}[line cap=round,line join=round,>=stealth,x=1.0cm,y=2.0cm,rotate=-25]

\draw [gray] (0.9*\xk+0.25*\xl,0.9*\yk+0.25*\yl)--(0.9*\xk+1.25*\xl,0.9*\yk+1.25*\yl)--(1.9*\xk+1.25*\xl,1.9*\yk+1.25*\yl)--(1.9*\xk+0.25*\xl,1.9*\yk+0.25*\yl)--cycle;

\draw [dotted] (\xk*3.25+0.5*\xl,3.25*\yk+0.5*\yl) -- (3.75*\xk+0.5*\xl,3.75*\yk+0.5*\yl); 
\draw [dotted] (-\xk+\xl,-\yk+\yl) -- (-0.5*\xk+\xl,-0.5*\yk+\yl); 
\draw [dotted] (1.5*\xk+2.75*\xl,1.5*\yk+2.75*\yl) -- (1.5*\xk+2.25*\xl,1.5*\yk+2.25*\yl); 
\draw [dotted] (1.5*\xk-1.25*\xl,1.5*\yk-1.25*\yl) -- (1.5*\xk-0.75*\xl,1.5*\yk-0.75*\yl); 

\foreach \k in {0,1,2}%,3,4,5,6,7,8,9}
\foreach \l in {0,1,2}%,3,4,5,6,7,8,9}
{
% top
\draw [line cap=rect,line join=rect,colmid, line width=1pt] (\xk*\k+\xl*\l,\yk*\k+\yl*\l) -- (\xk*\k+\xl*\l+\lx,\yk*\k+\yl*\l); 
\draw [line cap=rect,line join=rect,colright, line width=1pt] (\xk*\k+\xl*\l+\lx,\yk*\k+\yl*\l) -- (\xk*\k+\xl*\l+1,\yk*\k+\yl*\l); 
\draw [line width=0.75pt] (\xk*\k+\xl*\l+\lx+\ly/2,\yk*\k+\yl*\l-0.02) -- (\xk*\k+\xl*\l+\lx+\ly/2,\yk*\k+\yl*\l+0.02);

% bot
\draw [line cap=rect,line join=rect,colleft, line width=1pt] (\xk*\k+\xl*\l,\yk*\k+\yl*\l-0.05) -- (\xk*\k+\xl*\l+\ly,\yk*\k+\yl*\l-0.05);
\draw [line width=0.75pt] (\xk*\k+\xl*\l+\ly/2,\yk*\k+\yl*\l-0.03) -- (\xk*\k+\xl*\l+\ly/2,\yk*\k+\yl*\l-0.07);
\draw [line cap=rect,line join=rect,colmid, line width=1pt] (\xk*\k+\xl*\l+\ly,\yk*\k+\yl*\l-0.05) -- (\xk*\k+\xl*\l+1,\yk*\k+\yl*\l-0.05);
}
\end{tikzpicture}
\caption{The surface $S_{(\Lambda, a, b, \theta)}$}
\label{fig:infinite_surf}
\small
A fundamental domain of $\Lambda$ is drawn in gray (the vectors generating its two sides are in $\Lambda$).
\end{figure}
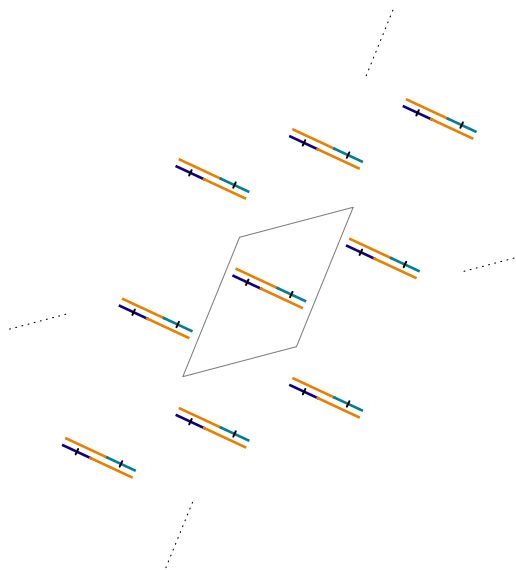
We note that the slits could not fit inside a fundamental domain of $\Lambda$. This will not be a problem.

In the following of this article, we can and will ignore case (\ref{item:def_S_case3}) as it is a special case: it has two "fake" singularities of conical angle $2\pi$ per fundamental domain, and only two "true" singularities (which have conical angle $3\pi$), whereas the surfaces in Case (\ref{item:def_S_case1}) and (\ref{item:def_S_case2}) have four "true" singularities per fundamental domain, as states the following lemma.
\begin{lem}
\label{lem:singularities_S} 
The surface
$S_{(\Lambda, a, b, \theta)}$ is $\Lambda$-periodic and has four singularities per fundamental domain of $\Lambda$: two of conical angle $\pi$ and two of conical angle $3\pi$.
\end{lem}

\begin{proof}
All the singularities of $S$ come from the identifications in the slit.
\begin{figure}[h!]
\centering
\begin{subfigure}[b]{0.45\textwidth}
\centering
\definecolor{colsa}{rgb}{0.8,0.2,0} % for disk sing
\definecolor{colsb}{rgb}{0.6,0,0.6} % for circle sing

\definecolor{colsc}{rgb}{0.4,0.6,0} % for filled rect sing
\definecolor{colsd}{rgb}{0.3,0.3,0.3} % for rect sing

\def\xk{1.2}
\def\yk{0.5}
\def\xl{-0.1}
\def\yl{1}

\def\lx{2.1}
\def\ly{0.8}
\begin{tikzpicture}[line cap=round,line join=round,>=stealth,x=2.0cm,y=2.0cm]
% top
\draw [->,colright, line width=1pt,shift={(\lx,0)}] (0,0) -- (\ly/3,0);
\draw [colright, line width=1pt,shift={(\lx,0)}] (\ly/4,0) -- (3*\ly/4,0);
\draw [<-,colright, line width=1pt,shift={(\lx,0)}] (2*\ly/3,0) -- (\ly,0); 
\draw [->>,colmid, line width=1pt] (0,0) -- (\lx/2,0); 
\draw [colmid, line width=1pt] (\lx/3,0) -- (\lx,0);

% bot
\draw [->>,colmid, line width=1pt,shift={(\ly,0)}] (0,-0.1) -- (\lx/2,-0.1);
\draw [colmid, line width=1pt,shift={(\ly,0)}] (\lx/3,-0.1) -- (\lx,-0.1);
\draw [->,>=latex,colleft, line width=1pt] (0,-0.1) -- (\ly/3,-0.1);
\draw [colleft, line width=1pt] (\ly/4,-0.1) -- (3*\ly/4,-0.1);
\draw [<-,>=latex,colleft, line width=1pt] (2*\ly/3,-0.1) -- (\ly,-0.1);

% sing
\fill [colsa] (\lx,0) circle (0.04);
\fill [colsa] (\lx+\ly,0) circle (0.04);
\fill [colsa] (\lx+\ly,-0.1) circle (0.04);

\draw [thick,colsb] (0,0) circle (0.04);
\draw [thick,colsb] (0,-0.1)  circle (0.04);
\draw [thick,colsb] (\ly,-0.1)  circle (0.04);

\fill [colsc, shift={(\lx+\ly/2,0)}] (-0.035,-0.035) rectangle (0.035,0.035) ;
\draw [colsd, thick, shift={(\ly/2,-0.1)}] (-0.035,-0.035) rectangle (0.035,0.035) ;
\end{tikzpicture}
\caption{If $a\cos\theta>b\sin\theta $}
\end{subfigure}
\begin{subfigure}[b]{0.45\textwidth}
\centering
\definecolor{colsa}{rgb}{0.8,0.2,0} % for disk sing
\definecolor{colsb}{rgb}{0.6,0,0.6} % for circle sing

\definecolor{colsc}{rgb}{0.4,0.6,0} % for filled rect sing
\definecolor{colsd}{rgb}{0.3,0.3,0.3} % for rect sing

\def\xk{1.2}
\def\yk{0.5}
\def\xl{-0.1}
\def\yl{1}

\def\lx{2.1}
\def\ly{0.8}
\begin{tikzpicture}[line cap=round,line join=round,>=stealth,x=2.0cm,y=2.0cm]
% top
\draw [->,colright, line width=1pt] (0,0) -- (\ly/3,0);
\draw [colright, line width=1pt] (\ly/4,0) -- (3*\ly/4,0);
\draw [<-,colright, line width=1pt] (2*\ly/3,0) -- (\ly,0); 
\draw [->>,colmid, line width=1pt] (\ly,0) -- (\ly+\lx/2,0); 
\draw [colmid, line width=1pt] (\ly+\lx/3,0) -- (\ly+\lx,0);

% bot
\draw [->>,colmid, line width=1pt] (0,-0.1) -- (\lx/2,-0.1);
\draw [colmid, line width=1pt] (\lx/3,-0.1) -- (\lx,-0.1);
\draw [->,>=latex,colleft, line width=1pt] (\lx,-0.1) -- (\lx+\ly/3,-0.1);
\draw [colleft, line width=1pt] (\lx+\ly/4,-0.1) -- (\lx+3*\ly/4,-0.1);
\draw [<-,>=latex,colleft, line width=1pt] (\lx+2*\ly/3,-0.1) -- (\lx+\ly,-0.1);

% sing
\fill [colsa] (0,0) circle (0.04);
\fill [colsa] (\ly,0) circle (0.04);
\fill [colsa] (0,-0.1) circle (0.04);

\draw [thick,colsb] (\lx+\ly,0) circle (0.04);
\draw [thick,colsb] (\lx,-0.1) circle (0.04);
\draw [thick,colsb] (\lx+\ly,-0.1) circle (0.04);

\fill [colsc, shift={(\ly/2,0) }, thick] (-0.035,-0.035) rectangle (0.035,0.035) ;

\draw [colsd, shift={(\lx+\ly/2,-0.1)}, thick]  (-0.035,-0.035) rectangle (0.035,0.035) ;
\end{tikzpicture}
\caption{If $a\cos\theta<b\sin\theta $}
\end{subfigure}
\caption{Singularities on the slit. The ones labeled with squares have conical angle $\pi$, the ones with circles or disks have conical angle $3\pi$}\label{fig:sing_slit}
\end{figure}
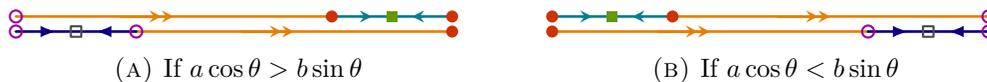
One can check on \Cref{fig:sing_slit} that $S$ has four singularities per slit, hence per fundamental domain that includes the slit. This holds as well for any fundamental domain because any fundamental domain contains the same number of singularities.
Squares correspond to singularities with conical angle $\pi$ and circles (or disks) to singularities with angle $3\pi$.
\end{proof}

\subsection{Equivalent trajectories}\label{subsec:equiv-traj}
We denote by $\mathcal{I}\subset\RR^2$ the set of slits without identifications.
Note that the projection $\pi_S : \RR^2\setminus\mathcal{I} \longrightarrow S_{(\Lambda, a, b, \theta)}$ is well-defined. 

\begin{dfn}\label{def:equiv_traj}
Let $\tilde{\varphi}\subset\RR^2$ and $\varphi\subset S_{(\Lambda, a, b, \theta)}$ be two curves. Let $\tilde{F}\subset \RR^2$ such that $\mathcal{I} \subset \tilde{F}$ and $F = \overline{\pi_S\left(\tilde F\setminus\mathcal{I}\right)}\subset S_{(\Lambda, a, b, \theta)}$.
We say that $\tilde{\varphi}$ and $\varphi$ coincide outside of $\tilde F$ if the images of the curves outside $\tilde{F}$ and $F$ are respectively the same: 
$$\pi_S\left(\tilde{\varphi}\cap\left(\RR^2\setminus \tilde{F}\right)\right) =\varphi\cap\left( S_{(\Lambda, a, b, \theta)}\setminus F \right).$$
\end{dfn}

We have defined $S_{(\Lambda, a, b, \theta)}$ so that each vertical trajectory in the wind-tree model coincide with a vertical geodesic on $S_{(\Lambda, a, b, \theta)}$, outside the set of obstacles.
This is what establishes the following lemma. 

\begin{lem}
\label{lem:equivalent_traj}
Let $\tilde{R}\subset\RR^2$ be the set of rectangular obstacles in the plane.
For every $\tilde p\in\RR^2\setminus \tilde{R}$, the vertical trajectory $\tilde{\varphi}$ starting at point $\tilde{p}$ in the wind-tree model of admissible parameters $(\Lambda, a, b, \theta)$ coincides outside of $\tilde R$ with the vertical trajectory starting at point $p = \pi_S(\tilde p)$ in $S_{(\Lambda, a, b, \theta)}$.

\end{lem}

\begin{proof}
We remind that because of the right angles of the obtstacles, every trajectory in $W(\Lambda,a,b,\theta)$ is parallel to itself outside of the set of obstacles $\tilde{R}$.

We show that the first entering points in $\tilde{R}$ and in $R=\pi_S(\tilde{R})$, and then exiting points, are the same in both models. The lemma follows by induction.

Let $\tilde p\in\RR^2\setminus \tilde{R}$, and $\tilde{\varphi}$ be the vertical trajectory starting at point $\tilde{p}$. Let $p=\pi_S(\tilde{p})$ and $\varphi$ be the vertical trajectory on $S_{(\Lambda,a,b,\theta)}$ starting at point $p$.
Let $\tilde{q}$ be the first intersection point between $\tilde{\varphi}$ and $\tilde{R}$, and $q$ be the first intersection point between $\varphi$ and $R$. 
We denote $\tilde{\varphi_0}$ the segment $[\tilde{p},\tilde{q}]$, and $\varphi_0$ the segment $[p,q]$.
By definition of the projection, $\pi_S(\tilde{\varphi_0})=\varphi_0$ and in particular $\pi_s(\tilde{q})=q$. 

Similarly, we denote $\tilde{\varphi_1}=[\tilde{q},\tilde{r}]$ the maximal segment of $\tilde{\varphi}$ containing $\tilde{q}$ and included in $\tilde{R}$. In other words, $\tilde{r}$ is the existing point of $\tilde{\varphi}$ from $\tilde{R}$. We define analogously $\varphi_1=[q,r]$ on $S_{(\Lambda,a,b,\theta)}$. The segment $\pi_S(\tilde{\varphi_1})$ is not vertical on $S_{(\Lambda,a,b,\theta)}$, on the contrary to $\varphi_1$. However, by construction of the slit (see \Cref{fig:rect_to_slit} and \Cref{def:surf_S}), they agree on $\pi_S(\tilde{r})=r$.

We conclude, with a direct induction, that both trajectories $\tilde{\varphi}$ and $\varphi$ coincide outside of $\tilde{R}$.

\end{proof}
\begin{rmk}
We highlight that the parameterizations of the curves $\varphi$ and $\tilde{\varphi}$ do not coincide. In fact, the difference in their parameterization may change at each crossing of an obstacle: sometimes $\varphi$ has less distance to cover, sometimes more, see Figure~\ref{fig:gamma_longer} and \ref{fig:gamma_shorter} below.
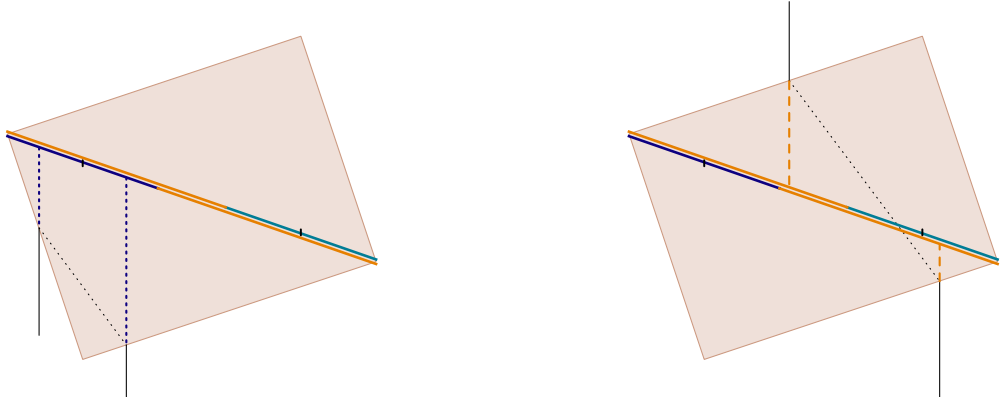
\begin{figure}[h!]
\begin{subfigure}[b]{0.47\textwidth}
\def\eps{0.03}
\def\xa{3.9}
\def\ya{1.3}
\def\xb{-1}
\def\yb{3}
\def\xp{-0.72}
\def\xpp{3.62}
\begin{tikzpicture}[line cap=round,line join=round,>=triangle 45,x=1.0cm,y=1.0cm]
\clip(-2.5,-0.5) rectangle (4,4.8);
\draw[color=colR,opacity=0.5,fill=colR,fill opacity=0.15] (0,0) -- (3.9,1.3) -- (2.9,4.3) -- (-1,3) -- cycle;

% traj left
\draw (-0.58,0.32)-- (-0.58,1.75);
\draw (0.58,0.19)-- (0.58,-0.5);
\draw [dotted] (-0.58,1.75)-- (0.58,0.19);
\draw [thick,dotted,color=colleft] (-0.58,2.85-\eps) --  (-0.58,1.75);
\draw [thick,dotted,color=colleft] (0.58,2.45-\eps) -- (0.58,0.19);

% slit
\draw [line cap=rect,line join=rect,color=colleft,shift={(0,-\eps)}, line width= 1pt] (-1,3) -- (1,2.30612);
\draw [line width=0.75pt] (0,2.65-\eps-0.05)--(0,2.65-\eps+0.03);
\draw [line cap=rect,line join=rect,color=colright,shift={(0,+\eps)}, line width= 1pt] (1.9,1.99388) -- (3.9,1.3);
\draw [line width=0.75pt] (2.9,1.65+\eps-0.03)--(2.9,1.65+\eps+0.05);
\draw [line cap=rect,line join=rect,color=colmid,shift={(0,+\eps)}, line width= 1pt] (-1,3) -- (1.9,1.99388);
\draw [line cap=rect,line join=rect,color=colmid,shift={(0,-\eps)}, line width= 1pt] (1,2.30612) -- (3.9,1.3);
\end{tikzpicture}
\caption{$\varphi$ (in plain black and dotted deep blue) may be longer than $\tilde{\varphi}$ (in plain and dotted black line)}
\label{fig:gamma_longer}
\end{subfigure}
\hfill
\begin{subfigure}[b]{0.47\textwidth}
\definecolor{colsg}{rgb}{0,0,0}
\def\eps{0.03}
\def\xa{3.9}
\def\ya{1.3}
\def\xb{-1}
\def\yb{3}
\def\xp{-0.72}
\def\xpp{3.62}
\begin{tikzpicture}[line cap=round,line join=round,>=triangle 45,x=1.0cm,y=1.0cm]
\clip(-2.5,-0.5) rectangle (4,4.8);
\draw[color=colR,opacity=0.5,fill=colR,fill opacity=0.15] (0,0) -- (3.9,1.3) -- (2.9,4.3) -- (-1,3) -- cycle;

% traj middle 2
\draw (3.13,1.04)-- (3.13,-0.5);
\draw [dotted] (1.13,3.71)-- (3.13,1.04);
\draw (1.13,4.76)-- (1.13,3.71);
\draw [thick,dashed,color=colmid] (1.13,2.3+\eps)--(1.13,3.71);
\draw [thick,dashed,color=colmid] (3.13,1.04)-- (3.13,1.55-\eps);

% slit
\draw [line cap=rect,line join=rect,color=colleft,shift={(0,-\eps)}, line width= 1pt] (-1,3) -- (1,2.30612);
\draw [line width=0.75pt] (0,2.65-\eps-0.05)--(0,2.65-\eps+0.03);
\draw [line cap=rect,line join=rect,color=colright,shift={(0,+\eps)}, line width= 1pt] (1.9,1.99388) -- (3.9,1.3);
\draw [line width=0.75pt] (2.9,1.65+\eps-0.03)--(2.9,1.65+\eps+0.05);
\draw [line cap=rect,line join=rect,color=colmid,shift={(0,+\eps)}, line width= 1pt] (-1,3) -- (1.9,1.99388);
\draw [line cap=rect,line join=rect,color=colmid,shift={(0,-\eps)}, line width= 1pt] (1,2.30612) -- (3.9,1.3);
\end{tikzpicture}
\caption{$\varphi$ (in plain black and dashed orange) may be shorter than $\tilde{\varphi}$ (in plain and dotted black line)}
\label{fig:gamma_shorter}
\end{subfigure}
\caption{Changing the surface from $\RR^2$ to $S_{(\Lambda,a,b,\theta)}$ changes the parameterization of the trajectory}
\end{figure}
 \end{rmk}

\subsection{The finite half-translation surface we will work with}
\label{subsec:finite_HTS}

Since $S_{(\Lambda, a, b, \theta)}$ is a periodic surface, we can study a quotient of it instead of itself. This is more convenient because we will get a finite half-translation surface. As a corollary of \Cref{lem:singularities_S}, we get the following.
\begin{lem}
\label{lem:stratum_X}
The quotient surface
$X_{(\Lambda, a, b, \theta)} = \sfrac{S_{(\Lambda, a, b, \theta)}}{\Lambda}$  is in the stratum $\mathcal{Q}(-1^2,1^2)$.
\end{lem}
We call $X_{(\Lambda, a, b, \theta)}$ the \emph{surface associated} with the tuple $(\Lambda,a,b,\theta)$. When the context is clear, we will simply write $X$ instead of $X_{(\Lambda, a, b, \theta)}$, and call it the associated surface.

\begin{rmk}
We could have considered first the quotient of the plane with obstacles by the lattice
and then have replaced the obstacle by a slit, we would have got the same surface $X_{\Lambda,a,b,\theta}$. 
\end{rmk}
Figure~\ref{fig:equiv_traj_on_torus} below illustrates the equivalence of trajectories between the torus with an obstacle and the slit torus. We kept the rectangle on the figure to allow the reader to compare and check that the trajectories outside the rectangle are the same.
\begin{figure}[h!]
         \centering
         \begin{subfigure}[b]{0.45\textwidth}
         \centering
\includegraphics[scale=0.2]{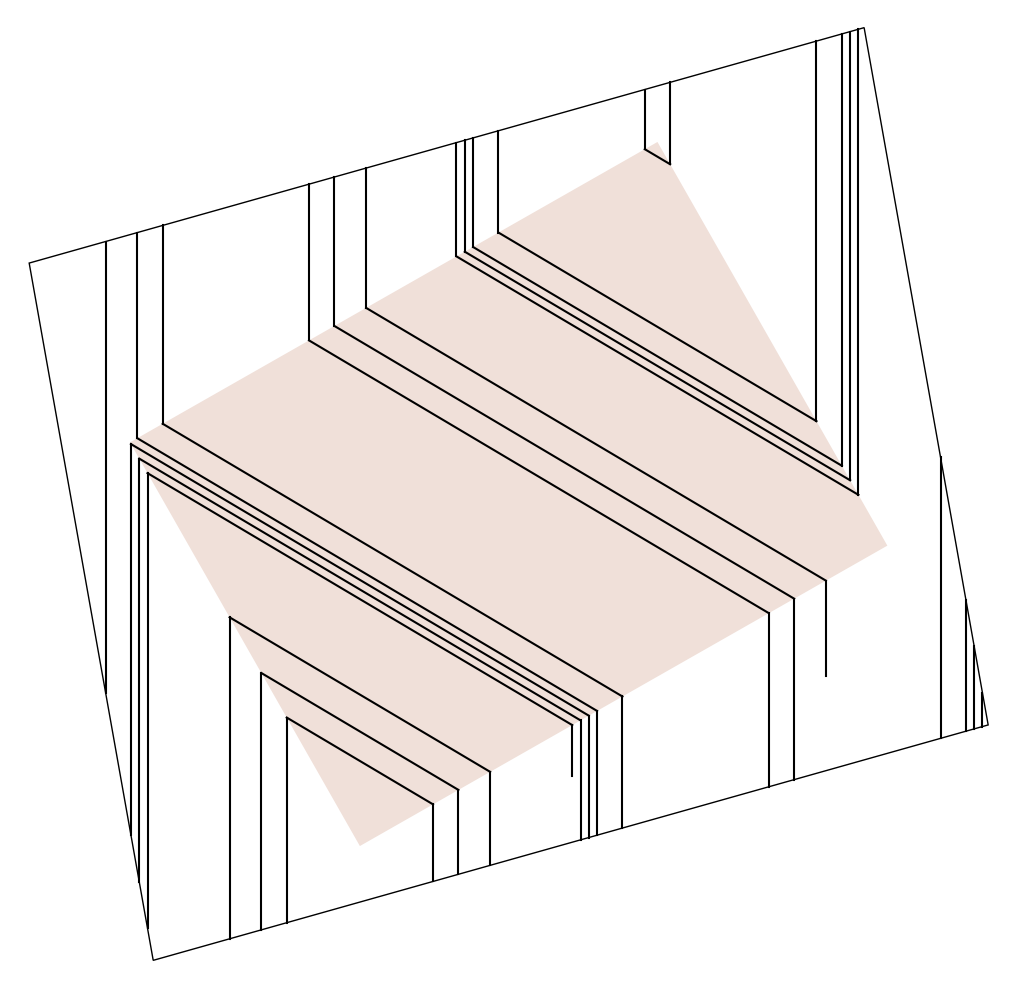}
\caption{On a torus with an obstacle}
\end{subfigure}
\hfill
\begin{subfigure}[b]{0.45\textwidth}
\centering
\includegraphics[scale=0.2]{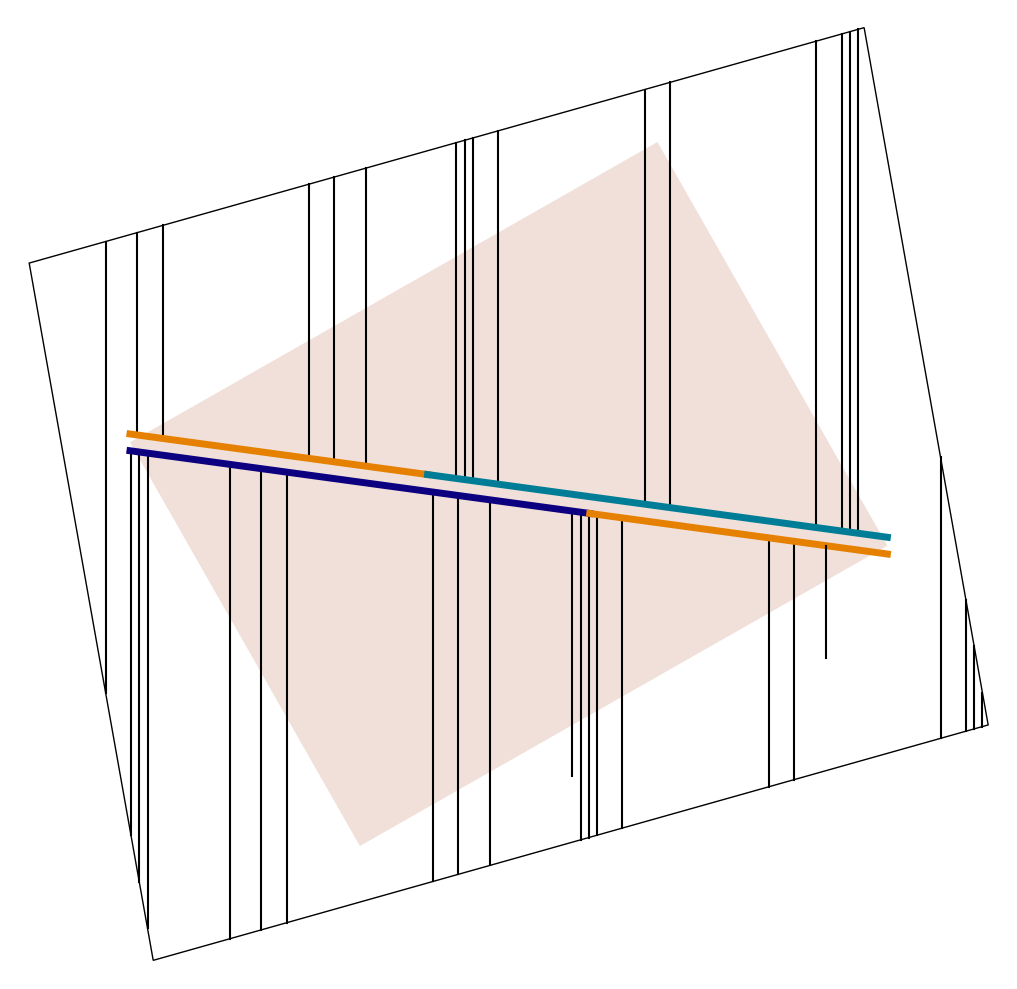}
\caption{On a torus with a slit}
\label{fig:slit_torus_with_traj}
\end{subfigure}
\caption{Two equivalent trajectories}
\label{fig:equiv_traj_on_torus}
\end{figure}

\begin{rmk}
In our context, the natural way to picture $X$ is as a slit torus, like in Figure~\ref{fig:slit_torus_with_traj}. We will also give another representation in Figure~\ref{fig:def_curves}, and utilize both representations for different purposes.
\end{rmk}

We denote the canonical projections by $ \pi :S_{(\Lambda, a, b, \theta)} \longrightarrow X_{(\Lambda, a, b, \theta)} $ and 
$$\pi_X = \pi\circ\pi_S : \RR^2\setminus\mathcal{I} \longrightarrow X_{(\Lambda, a, b, \theta)} .$$

\subsection{How generic is the surface?}\label{subsec:generic_surf}

In the previous section, we associated a surface $X_{(\lambda,a,b,\theta)}$ with each system $W(\lambda,a,b,\theta)$ where $(\lambda,a,b,\theta)$ is admissible. 
 We denote the set of such surfaces corresponding to a system $W(\Lambda,a,b,\theta)$ by
 $$\mathcal{A}=\left\{\left. X_{(\Lambda,a,b,\theta)} \right| (\Lambda,a,b,\theta) \text{ is admissible}\right\}\subset\mathcal{Q}(-1^2,1^2).$$
 Note that we can define $X_{(\Lambda,a,b,\theta)}$ (as a torus $\RR^2/\Lambda$ with a slit defined by $a,b$ and $\theta$) even if $(\Lambda,a,b,\theta)$ is not admissible. We denote the bigger set of surfaces obtained in this way by 
 $$\mathcal{X}=\left\{\left. X_{(\Lambda,a,b,\theta)} \right| (\Lambda,a,b,\theta) \text{ possibly non admissible } \right\}\subset\mathcal{Q}(-1^2,1^2).$$
 
We highlight that $\mathcal{A}$ is an open set of $\mathcal{X}$ because the non-overlapping condition is open.
Lastly, we introduce the set $\mathcal{S}=\mathcal{S}_1\sqcup\mathcal{S}_2$ of tori with a cut in the shape of a parallelogram with identifications, as in \Cref{fig:slit_general}. We denote the boundary of this parallelogram-shaped cut $\mathcal{P}$.
 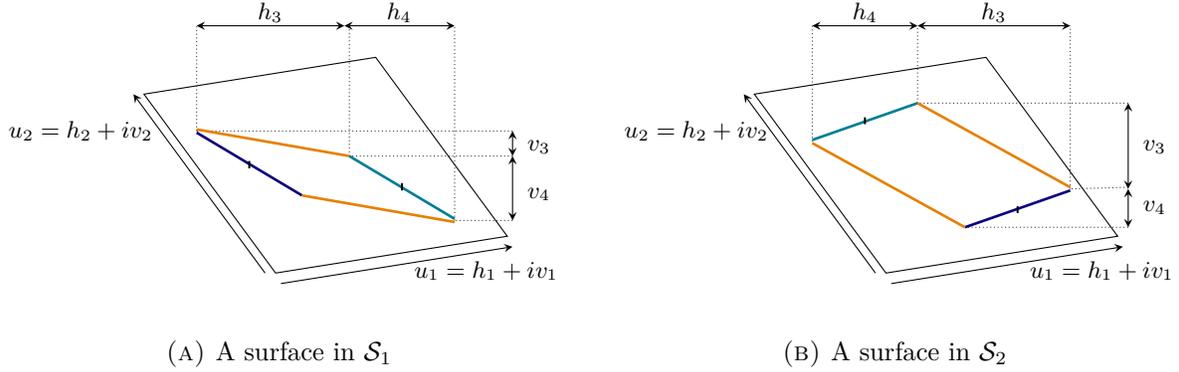
\begin{figure}[h!]
\begin{center}
\begin{subfigure}[b]{0.475\textwidth}
\definecolor{colhor}{rgb}{0.2,0.2,0.2}
\def\eps{0.03}
\def\xa{3.9}
\def\ya{1.3}
\def\xb{-1}
\def\yb{3}
\def\xp{-0.72}
\def\xpp{3.62}
\begin{tikzpicture}[>=stealth,x=.70cm,y=.70cm]

\draw (-2,3.7)--(0.5,0.3)--(4.9,1)--(2.4,4.4)--cycle;
% delim horizontales
\draw [densely dotted,colhor] (5,3) -- (-1,3);
\draw [densely dotted,colhor] (5,2.49388+\eps) -- (1.9,2.49388+\eps);
\draw [densely dotted,colhor] (5,1.3) -- (3.9,1.3);

\draw [densely dotted,colhor] (-1,3)--(-1,5);
\draw [densely dotted,colhor] (1.9,2.49388+\eps) -- (1.9,5);
\draw [densely dotted,colhor] (3.9,1.3) -- (3.9,5);

% slit
\draw [color=colleft,shift={(0,-\eps)}, line width= 1pt] (-1,3) -- (1,1.80612);
\draw [line width=0.75pt] (0,2.4-\eps-0.08)--(0,2.4-\eps+0.06);
\draw [color=colright,shift={(0,+\eps)}, line width= 1pt] (1.9,2.49388) -- (3.9,1.3);
\draw [line width=0.75pt] (2.9,1.9+\eps-0.06)--(2.9,1.9+\eps+0.08);
\draw [color=colmid,shift={(0,+\eps)}, line width= 1pt] (-1,3) -- (1.9,2.49388);
\draw [color=colmid,shift={(0,-\eps)}, line width= 1pt] (1,1.80612) -- (3.9,1.3);

\begin{footnotesize}

\draw [<->,shift={(0,5)}] (\xb,0)--(1.80612,0);
\draw (\xb/2+0.9,5.25) node {$h_3$};
\draw [<->,shift={(0,5)}] (1.80612,0)--(\xa,0);
\draw (\xa/2+0.9,5.25) node {$h_4$};

\draw [->,shift={(0.1,-0.2)}] (0.5,0.3)--(4.9,1);
\draw (4.5,0.35) node {$u_1=h_1+iv_1$};
\draw [->,shift={(-0.2,0)}] (0.5,0.3)--(-2,3.7);
\draw (-3.2,3) node {$u_2=h_2+iv_2$};

\draw [<->,shift={(0,0)}] (5,\ya)--(5,2.49388+\eps);
\draw (5.5,2.7) node {$v_3$};
\draw [<->,shift={(0,0)}] (5,2.49388+\eps)--(5,\yb);
\draw (5.5,1.8) node {$v_4$};
\end{footnotesize}
\end{tikzpicture}
    \caption{A surface in $\mathcal{S}_1$} %First combinatorics of the cut}
\end{subfigure}
\hfill
\begin{subfigure}[b]{0.475\textwidth}
%Warning: here colleft is on the right and colright is on the left
\definecolor{colhor}{rgb}{0.2,0.2,0.2}
\def\eps{0.03}
\def\xa{3.9}
\def\ya{1.3}
\def\xb{-1}
\def\yb{3}
\def\xp{-0.72}
\def\xpp{3.62}
\def\xc{-1}
\def\yc{2.8}
\def\xd{1}
\def\yd{3.5}
\def\xe{3.9}
\def\ye{1.9}
\def\xf{1.9}
\def\yf{1.2}
\begin{tikzpicture}[>=stealth,x=.70cm,y=.70cm]

\draw (-2.1,3.7)--(0.4,0.3)--(4.8,1)--(2.3,4.4)--cycle;
% delim horizontales
\draw [densely dotted,colhor] (5,\yd+\eps) -- (\xd,\yd+\eps);
\draw [densely dotted,colhor] (5,\ye+\eps) -- (\xe,\ye);
\draw [densely dotted,colhor] (5,\yf-\eps) -- (\xf,\yf-\eps);

\draw [densely dotted,colhor] (\xc,\yc)--(\xc,5);
\draw [densely dotted,colhor] (\xd,\yd+\eps) -- (\xd,5);
\draw [densely dotted,colhor] (\xe,\ye) -- (\xe,5);

% slit
\draw [color=colright,shift={(0,+\eps)}, line width= 1pt] (\xc,\yc) -- (\xd,\yd);
\draw [line width=0.75pt] (\xc/2+\xd/2,\yc/2+\yd/2+\eps-0.05)--++(0,0.13);
\draw [color=colleft,shift={(0,-\eps)}, line width= 1pt] (\xe,\ye) -- (\xf,\yf);
\draw [line width=0.75pt] (\xe/2+\xf/2,\ye/2+\yf/2-\eps-0.08)--++(0,0.13);
\draw [color=colmid,shift={(0,-\eps)}, line width= 1pt] (\xc,\yc) -- (\xf,\yf);
\draw [color=colmid,shift={(0,+\eps)}, line width= 1pt] (\xd,\yd) -- (\xe,\ye);

\begin{footnotesize}

\draw [<->,shift={(0,5)}] (\xc,0)--(\xd,0);
\draw (\xc/2+\xd/2,5.25) node {$h_4$};
\draw [<->,shift={(0,5)}] (\xd,0)--(\xe,0);
\draw (\xe/2+\xd/2,5.25) node {$h_3$};

\draw [->,shift={(0,-0.2)}] (0.5,0.3)--(4.9,1);
\draw (4.5,0.35) node {$u_1=h_1+iv_1$};
\draw [->,shift={(-0.3,0)}] (0.5,0.3)--(-2,3.7);
\draw (-3.2,3) node {$u_2=h_2+iv_2$};

\draw [<->,shift={(0,0)}] (5,\ye)--(5,\yd+\eps);
\draw (5.5,\ye/2+\yd/2) node {$v_3$};
\draw [<->,shift={(0,0)}] (5,\yf-\eps)--(5,\ye);
\draw (5.5,\yf/2+\ye/2) node {$v_4$};
\end{footnotesize}
\end{tikzpicture}
    \caption{A surface in $\mathcal{S}_2$} %Second combinatorics of the cut}
\end{subfigure}
\caption{Parameters of the surfaces in $\mathcal{S}$}
\label{fig:slit_general}
\end{center}
\end{figure}
An element in $\mathcal{S}$ is described by 8 real parameters:
\begin{itemize}
    \item $h_1,v_1,h_2,v_2\in \RR$ define the shape of the torus
    \item $h_3,h_4,v_3,v_4>0$ define the slit
\end{itemize}
For clarity of the pictures, we took  $h_3+h_4<h_1+h_2$ and $v_3+v_4<v_1+v_2$ so that the cut is included in the parallelogram-shaped fundamental domain of $\Lambda$ defined by $u_1=h_1+iv_1$ and $u_2=h_2+iv_2$. But we highlight that the cut could be longer, as soon as it does not overlap with itself. Remark that the geometry of the surface does not depend on where the cut is placed in a fundamental domain of $\Lambda$.
Since $\mathcal{S}\subset \mathcal{Q}(-1^2,1^2)$ are both 8 (real) dimensional, $\mathcal{S}$ is locally equal to $\mathcal{Q}(-1^2,1^2)$. In contrast, the surfaces in $\mathcal{A}$ and $\mathcal{X}$ have real dimension $7$ and are therefore specific.
\begin{itemize}
    \item 4 real parameters for the torus
    \item 2 real parameters for the sloap and length of the slit
    \item 1 real parameter for the length of the side associated with translation (the length of the other side of the slit is deduced from the length of the slit)
\end{itemize} 
 However, we show in this section that almost every $X\in\mathcal{X}$ is generic, meaning that:
\begin{enumerate}
\item
the vertical flow on the surface is uniquely ergodic 
the $g_t$-orbit of the surface is recurrent in moduli space (i.e. the surface is \emph{Masur generic}),
\item
the surface admits an Oseledets flag for the Kontsevich--Zorich cocycle (i.e. the surface is \emph{Oseledets generic}).
\end{enumerate} 
The "almost ever" is to understand with respect to the Masur-Veech measure, which is finite on the subset of surfaces of area at most 1.
We defined the Kontsevich--Zorich cocycle in \Cref{sec:background_KZ_geometric} and Oseledets flag in \Cref{subsec:Oseledets}.

\begin{lem}
\label{lem:X_generic}
For almost every admissible tuple $(\Lambda,a,b,\theta)$, the associated half-translation surface $X$ is Masur and Oseledets generic.
\end{lem}
\begin{proof}
Let $\Omega\subset\mathcal{A}$ be an open set in $\mathcal{A}$. We will show that $\Omega$ contains a Masur and Oseledets generic surface. Since $\Omega$ is arbitrary, this proves the lemma.
Up to shrinking $\Omega$, we assume that $\Omega$ is included in either $\mathcal{S}_1$ or $\mathcal{S}_2$. We now thicken $\Omega$ to get an open set of $\mathcal{S}$. To do so, let $\mathcal{U}$ be an open set of $\mathcal{S}$ containing $\Omega$ and contained in either $\mathcal{S}_1$ or $\mathcal{S}_2$. For every $Y\in\Omega$, we define the set $F(Y)\subset\mathcal{S}$:
$$F(Y) = \left\{  Z\in\mathcal{U}  \left| \begin{array}{cl}
     \forall i=1,2, & u_i(Z)=u_i(Y), \\
     \forall i=3,4, & h_i(Z) = h_i(Y), \\
    &  v_3(Z)+v_4(Z) =v_3(Y)+v_4(Y)
\end{array} \right. \right\},$$
and we thicken $\Omega$ with them:
$$ \mathcal{V}= \bigsqcup_{Y\in\Omega} F(Y).$$
From the coordinates of the elements in $\mathcal{S}$, we see that it $\mathcal{V}$ is an open set in $\mathcal{S}$, ie. in $\mathcal{Q}(-1^2,1^{2})$. Hence, it contains a surface, say $Z_0\in F(Y_0)$, that is both Masur and Oseledets generic.

The set $F(Y_0)$ is one (real) dimensional: all the surfaces in $F(Y_0)$ have the same coordinates as $Y_0$ except for $v_3$ and $v_4$ for which we have one degree of freedom. Moreover the surfaces in $F(Y_0)$ differ only in some vertical coordinates (ie. they are in the same local stable manifold) and therefore have the same forward asymptotic behavior under $g_t$ in $\mathcal{Q}(-1^2,1^2)$. 
Since Masur and Oseledets genericities depend only on the asymptotic behavior, $Y_0$ is Masur and Oseledets generic like $Z_0$.
\end{proof}

\subsection{The new question}\label{subsec:new_question}
When we quotient the surface $S_{(\Lambda,a,b,\theta)}$ by the lattice $\Lambda$, we forget the position of the trajectory in the plane $\RR^2$. But we can recover this information using the algebraic intersection number of our curve with a basis of the homology. This is the purpose of \Cref{lem:transition}. Before stating it, we formalize how we get a closed loop from a trajectory (between time 0 and time $T$).

Let $(e_1,e_2)$ be a basis of $\Lambda$. The vectors $e_1$ and $e_2$ induce two curves, respectively $\xi_1$ and $\xi_2$, say based at a point $x_0\in X$.
Then, $(\xi_1,\xi_2)$ generates the fundamental group $\pi_1(X,x_0)$ of $X$.
Their homology classes $(\zeta_1,\zeta_2)=([\xi_1],[\xi_2])$ form a symplectic basis of $H_1(X,\RR)$.
Moreover, $X$ is topologically a torus and $\xi_1$ and $\xi_2$ two closed (non null homotopic) curves intersecting once, so $X\setminus\{\xi_1,\xi_2\}$ is topologically a disk and in particular is path connected.
We conclude that any two points in $X$ can be connected by a path that does not intersect $\xi_1$ nor $\xi_2$. 
In particular, we can close the trajectory between time 0 (starting at point $p$) and time $T$ into a loop, which we will denote by $\gamma_T$.
\begin{figure}[h!]
\centering
\begin{subfigure}[b]{0.4\textwidth}
\includegraphics[scale=0.1]{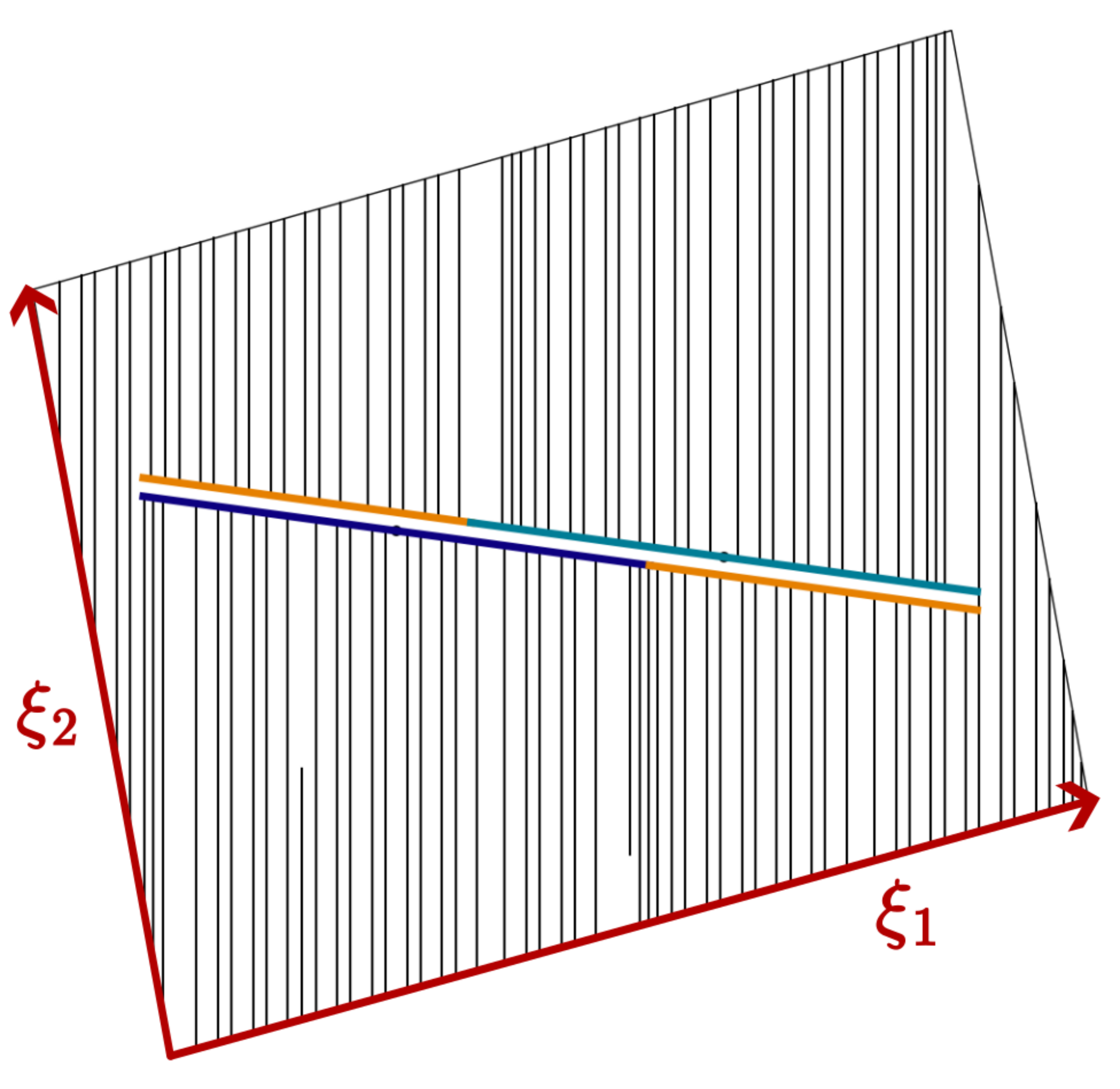}
\caption{Cycles $\zeta_1$ and $\zeta_2$}
\label{fig:torus_slit_hv}
\end{subfigure}
\hspace{1cm}
\begin{subfigure}[b]{0.4\textwidth}
\includegraphics[scale=0.1]{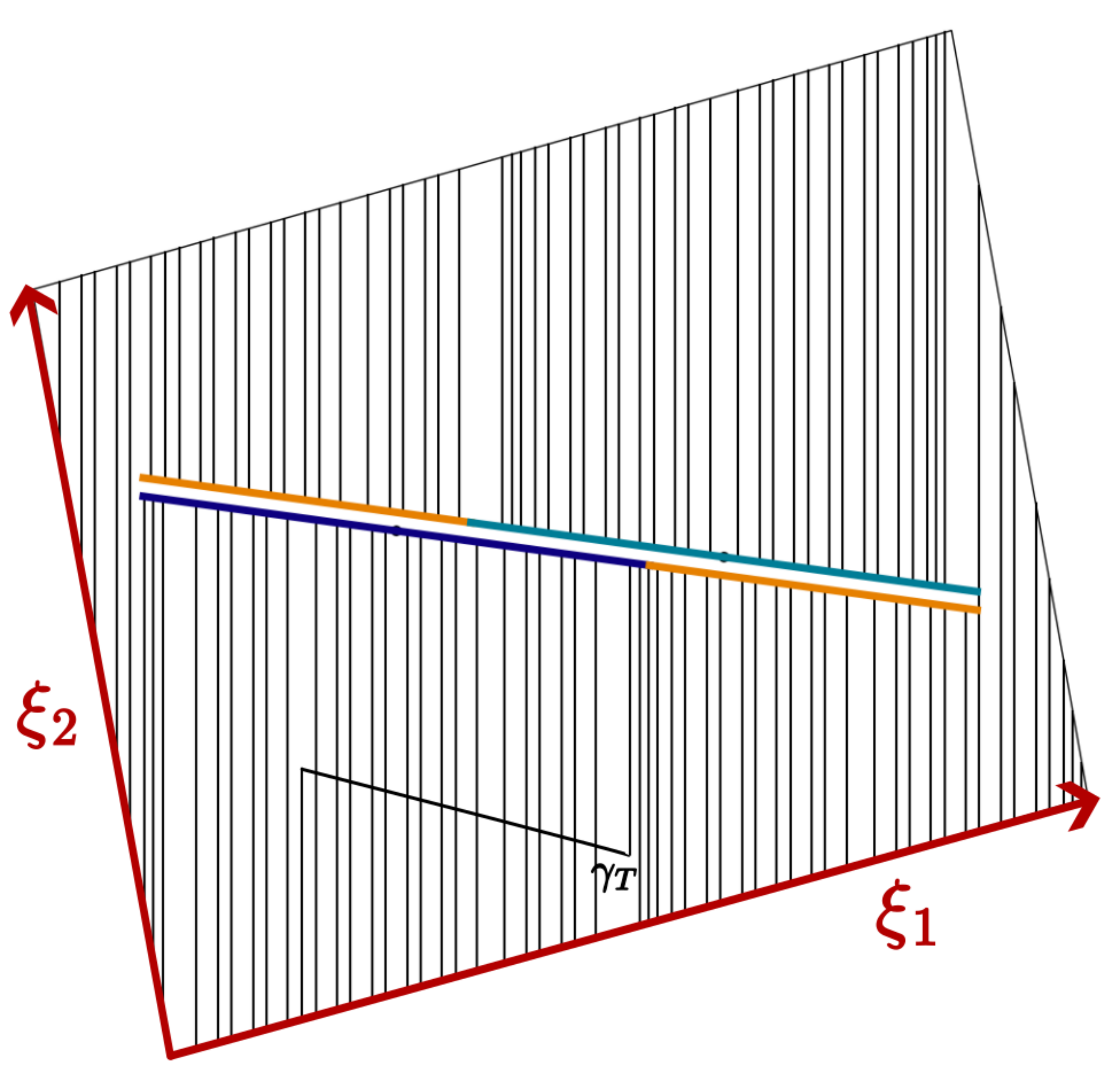}
\caption{Cycle $\gamma_T$}
\label{fig:torus_slit_gamma}
\end{subfigure}
\caption{The cycles that we consider}
\label{fig:torus_slit_hv_gamma}
\end{figure}
More precisely, we denote by 
$$\Phi : \begin{array}{ccc}
 X\times\RR_+ & \longrightarrow & X \\
(x,s) & \mapsto & \Phi_s(x)
\end{array}
$$ 
the vertical flow in $X$. In particular, for every $x$, $\Phi_0(x)=x$. For $s<t$, we denote by $\varphi_s^t(x)$ the part (of length $t-s$) of the vertical foliation of $x$ that is between $\Phi_s(x)$ and $\Phi_t(x)$.
Let $T_1,T_2>0$, and denote by $\gamma_{T_1,T_2}(p)$ the closed curve made of the union of $\varphi_{T_1}^{T_2}(p)$ and a path between $\Phi_{T_2}(p)$ and $\Phi_{T_1}(p))$ that does not intersects $\xi_1$ nor $\xi_2$.
For $T>0$, we define 
$$
\gamma_T(p) = \gamma_{0,T}(p)
$$
the cycle that starts at $p$, follows the vertical flow until time $T$ and that comes back to $p$ without intersecting $\xi_1$ nor $\xi_2$.
Finally, we denote by $$\Psi : 
 \begin{array}{ccc}
\RR^2\times\RR_+ & \longrightarrow &  \RR^2 \\
(x,s) & \mapsto & \Psi_s(x)
\end{array}
$$
the flow of the wind-tree tiling billiards.
We denote by $\iota$ the intersection form on $H_1(X,\RR^2)$. As an example, $\gamma=n_1\zeta_1+n_2\zeta_2\in H_1(X,\RR^2)$ satisfies
\begin{align*}
\iota(\zeta_1,\gamma)&=\iota(\zeta_1,n_1\zeta_1+n_2\zeta_2)=n_2\\
\iota(\zeta_2,\gamma)&=\iota(\zeta_2,n_1\zeta_1+n_2\zeta_2)=-n_1.
\end{align*}

\begin{lem}\label{lem:transition}
Let $X\in\mathcal{S}$ and $v\in H_1(X,\RR)$. If
$$ 
\exists C>0, \,\forall T>0, \qquad |\iota(v,\gamma_T)| < C
$$
then
$$
\exists C'>0, \, \exists z\in\RR^2\setminus\{0\}, \, \forall T>0, \qquad |\langle z,\Psi_T(\tilde{p})-\tilde{p}\rangle| < C',
$$
where $\langle\cdot ,\cdot\rangle$ denotes the Euclidean scalar product and $\tilde{p}\in\pi^{-1}(p)$.
\end{lem}
The vector $\Psi_T (\tilde{p})-\tilde{p}$ corresponds to the displacement of the tiling billiards trajectory in the plane, between times $0$ and $T$.
In other words, the wind-tree tiling billiards trajectory is trapped in a strip of width $2C'/||z||$ and orthogonal to the vector $z$.

The idea of the proof is simple. Imagine for a second that $v=\zeta_1$, the algebraic intersection number $k$ between $\gamma_T$ and $\zeta_1$ gives us information on the position of the trajectory at time $T$ in the plane: it lies at a point $xe_1+ye_2$ where $\lfloor y \rfloor = k$. Therefore, the trajectory must remain in a strip in the direction of $e_1$. The complete proof must first take into account the reparameterization between the trajectory in the wind-tree model and the one on the half-translation surface, and secondly, treat the general case of $v\in H_1(X,\RR)$.

\begin{proof}
We recall that the parameterization of $\Psi_s$ does not necessarily coincide with the one of $\Phi_s$, i.e. a priori $\pi_X(\Psi_s(\tilde{p})) \neq \Phi_s(p)$, for $\tilde{p}\in\RR^2\setminus\tilde{R}$ and $p=\pi_X(\tilde{p})\in X$, where the set $\tilde{R}$ project onto the set of rectangles in the wind-tree model (see Lemma~\ref{lem:equivalent_traj}).
Let ${\mathcal{T}_r = \{s \in\RR_+ | \Psi_s(\tilde{p}) \notin R \}}$ be the set of times when the tiling billiards trajectory in the plane lies outside the rectangles $R$.
Let $r:\RR_+ \longrightarrow \RR_+$ be a reparameterization function such that $\forall s\in\mathcal{T}_r, \, \pi_X(\Psi_s(\tilde{p})) = \Phi_{r(s)}(p)$. The values of $r$ outside the set $\mathcal{T}_r$ do not matter.
Let $T\in r(\mathcal{T}_r)$, so that $\Phi_T(p)\notin R=\pi_X(\tilde{R})$, and let $ s=r^{-1}(T)$.
We have $\pi_X(\Psi_{s}(\tilde{p}))=\Phi_T(p)$.

Complete $v$ into a basis $(v,w)$ of $H_1(X,\RR)$. Let $v=v_1\zeta_1+v_2\zeta_2$ and $w=w_1\zeta_1+w_2\zeta_2$ be their decomposition in the basis $(\zeta_1,\zeta_2)$ of $H_1(X,\RR)$ induced by the basis $(e_1,e_2)$ of $\Lambda$.
We consider $\tilde v = v_1e_1+v_2e_2$ and $\tilde w = w_1e_1+w_2e_2$ which form a basis $(\tilde{v},\tilde{w})$ of $\RR^2$. 
We decompose the loop $\gamma_T = x v+yw =(xv_1+yw_1)\zeta_1+(xv_2+yw_2)\zeta_2$ in both bases $(v,w)$ and $(\zeta_1,\zeta_2)$. Note that $(xv_1+yw_1)$ and $(xv_2+yw_2)$ are integers.
The end point in $\RR^2$ of the lift $\tilde\gamma_T$ of $\gamma_T$ to $\RR^2\setminus\mathcal{P}$ is 
$$\tilde q =\tilde p + x\tilde v + y \tilde w =\tilde{p}+(xv_1+yw_1)e_1+(xv_2+yw_2)e_2.$$ 
It is at distance at most 
$\max \{ || u_1e_1+u_2e_2|| \, | \, u_1,u_2\in[0,1] \}$
away from $\Psi_s(\tilde{p})$ because we obtained $\gamma_T$ by closing the trajectory $\varphi_0^T$ without crossing $\xi_1$ nor $\xi_2$.
Let $z$ be the projection of $\tilde w$ to the orthogonal subspace to $\tilde v$. It is non zero since $\tilde{w}$ is not in the linear span of $\tilde{v}$. 
Then
\begin{align*}
    |\langle z,\Psi_{s}(\tilde{p}) - \tilde{p}\rangle| &\leqslant |\langle z, \Psi_T(\tilde{p})-\tilde q \rangle| + | \langle z,\tilde{q}-\tilde{p}\rangle|  \\
    &\leqslant m + |y| |\langle z, w\rangle |
\end{align*}
where $m=\max \{\langle z, u_1e_1+u_2e_2 \rangle \, | \, u_1,u_2\in[0,1] \}$.
Moreover, $|y||\iota(v,w)|=|\iota(v,\gamma_T)|<C$ (by assumption) so $|y|<C/|\iota(v,w)|$.
The constant 
$$C' = m + C|\langle z,w\rangle|/|\iota(v,w)|$$
depends only on $v$ and $w$, not on $p$ nor on $T$, and satisfies $ |\langle z,\Psi_{s}(\tilde{p}) - \tilde{p}\rangle| \leqslant  C'$. 

For a time $T\notin\mathcal{T}_r$, we consider the closest time $s\in\mathcal{T}_r$ to $T$ such that $T<s$. Then the difference vector $\epsilon = \left(\Psi_{T}(\tilde{p}) - \tilde{p} \right) - \left(\Psi_{s}(\tilde{p}) - \tilde{p} \right)$ is uniformly bounded. 
Therefore, up to take a larger constant $C'$, we also have $ |\langle z,\Psi_T(\tilde{p})-\tilde{p}\rangle| < C'$, 
which ends the proof.
\end{proof}

Lemma~\ref{lem:transition} shows that Theorem~\ref{thm:bounded_intersection} implies Theorem~\ref{thm:tb_in_wt}. 
We have now rephrased the existence of a trapping strip into the existence of a homology class that has a bounded algebraic intersection number with $\gamma_T$.

To answer this question, the idea is to use the Teichmüller flow $g_t$ to "shorten" the obtained curve $\gamma_T$ so that we have fewer intersections to count with a new basis. But when we closed the path to obtain $\gamma$, we (\textit{a priori}) added a horizontal component, which will be extended by $g_t$. To be able to still control the length of the curve after applying $g_t$, we decompose $\gamma_T$ into several closed curves $\gamma^{(k)}$, which we will each study at a different scale of time, chosen so that the horizontal component of $\gamma^{(k)}$ is bounded. A different scale of time means that we change from the basis $(\zeta_1,\zeta_2)$ of $H_1(X,\RR)$ to a basis $(\zeta_1^{(k)},\zeta_2^{(k)})$, image by the flow. As a result, $\gamma^{(k)}$ intersects many fewer times $(\zeta_1^{(k)},\zeta_2^{(k)})$ than $(\zeta_1,\zeta_2)$ (in fact, a bounded number of times).
Finally, the Kontsevich--Zorich cocycle gives us the relation between the basis $(\zeta_1,\zeta_2)$ and $(\zeta_1^{(k)},\zeta_2^{(k)})$.
Hence, we get the algebraic intersection number between $\gamma^{(k)}$ and both $\zeta_1$ and $\zeta_2$ from the algebraic intersection number between $\gamma^{(k)}$ and both $\zeta_1^{(k)}$ and $\zeta_2^{(k)}$, the latter being bounded.

We explain the decomposition of $\gamma_T$ into the $\gamma^{(k)}$'s in the next section.

\section{Decomposition of the trajectory}\label{sec:decompo_traj}
In order to control the decomposition of $\gamma_T$ presented in \Cref{subsec:decompo}, it is easier to restrict ourselves in a nice open set of $\mathcal{S}$, which we introduce in \Cref{subsec:prefered-open-set}. In \Cref{subsec:bounds}, we establish uniform bounds on a neighborhood of $X$.

\subsection{A preferred open set}\label{subsec:prefered-open-set}
Let $\epsilon >0$ and consider the following conditions on the coordinates of a surface in $\mathcal{S}$  (cf. parameterizations of surfaces in $\mathcal{S}$ in \Cref{fig:slit_general}).
\begin{align}
& 0< \arg u_1 < \pi/2 < \arg u_2 < \pi \label{cond1} \\
&  h_3+h_4 < h_1 - h_2 \label{cond2} \\
& v_3+v_4 < v_2 - v_1 \label{cond3} \\
& \max(v_3, v_4) + \varepsilon < \min(v_1, v_2)  \label{cond4}
\end{align}
We define the open set of $\mathcal{S}$
$$\mathcal{O}_\varepsilon = \{X\in\mathcal{S} \, | \, X \textrm{ satisfies conditions (\ref{cond1}) to (\ref{cond4})} \}.$$
This open set corresponds to tori with a parallelogram fundamental domain (given by $u_1$ and $u_2$) that contains the slit (conditions (\ref{cond1}) and (\ref{cond2}). In other words, the slit is not too long. This is to avoid technical difficulties. The conditions (\ref{cond1}) and (\ref{cond4}) will be useful to prove \Cref{lem:lower_bound_vert_length}. Condition (\ref{cond1}) 
ensures that an "upward" going loop intersects positively $\zeta_1$ and negatively $\zeta_2$. The meaning and importance of the last condition will appear in the proof of \Cref{lem:lower_bound_vert_length}.

Let $X\in\mathcal{O}_\varepsilon$ be a half-translation surface that is Masur and Oseledets generic. 
We assume moreover that $X$ has no vertical nor horizontal saddle connection. Note that to have a horizontal or vertical saddle connection, the eight real parameters describing $X$ 
have to satisfy a linear equation. This means that almost every surface in $\mathcal{O}_\varepsilon$ does not have a horizontal or vertical saddle connection.

Let $p\in X$ be a regular point for the vertical flow in $X$ (almost every point is regular). 
We consider the vertical trajectory in $X$ starting at $p$. 

\begin{rmk}
The same argument works for p not regular, it suffices to 
restrict ourselves to the times before the first time the vertical flow hits a singularity.
\end{rmk}

\subsection{Uniform bounds in a neighborhood of the surface \texorpdfstring{$X$}{X}}\label{subsec:bounds}
\begin{dfn}
We call \emph{vertical length} of a path $\gamma$ its integral against the $1$-form $\textrm{d}z$. 
We denote it by $l_{\textrm{v}}$:
$$
l_{\textrm{v}}(\gamma) = \int | Im \gamma' (t)| \textrm{d}t= \int | \gamma_y' (t)| \textrm{d}t.
$$
if $\gamma=\gamma_x +i\gamma_y$.
\end{dfn}

The following lemma is about the non algebraic intersection number between curves of the basis of homology and a vertical path.
\begin{lem}
\label{lem:cte_K}
There exist a relatively compact neighborhood $\mathcal{U}$ of $X$ in $\mathcal{Q}(-1^2,1^2)$ and a constant $K>0$ such that for every $Y\in\overline{\mathcal{U}}$, there exists a segment $I(Y)\subset Y$ such that, for $i,j\in\{1,2\}$:
\begin{itemize}
\item
every path $\gamma$ on $Y$ of \emph{vertical
length} at least $K$ intersects at least once the segment $I(Y)$, 
\item
every path $\gamma$ on $Y$ that intersects at least twice $I(Y)$ has \emph{vertical length} at least $K^{-1}$.
\end{itemize}
Moreover, for every $Y\in\overline{\mathcal{U}}$, the homology group $H_1(Y,\RR)$ can be identified to $H_1(X,\RR)$, and the basis $(\zeta_1(Y),\zeta_2(Y))$ of $H_1(Y,\RR)$ corresponding to the basis $(\zeta_1,\zeta_2)$ of $H_1(X,\RR)$ is such that:
$$
K^{-1}\leqslant l_{\text{v}}\left(\zeta_j(Y) \right) \leqslant K
$$
where $l_{\text{v}}\left(\zeta \right)$ denotes the vertical length of the shortest curve in the class of $\zeta$.
\end{lem}
\begin{proof}
Let $\mathcal{U}$ a relatively compact neighborhood of $X$ in $\mathcal{Q}(-1^2,1^2)$, small enough so that every surface in $\mathcal{U}$, seen as a polygon with identified sides, has the same combinatorial data as $X$ (i.e. the sides are identified in the same order for every surface of $\mathcal{U}$). 
From this point of view, each singularity of a surface is a vertex of the polygon. Since all polygons have the same combinatorial type,
there is a natural identification of the singularities of a surface of $\mathcal{U}$ to the singularities of any other surface in $\mathcal{U}$.
Fix $s$ a singularity of $X$. Choose a horizontal separatrix going out of $s$ and let $I$ be a subsegment of this separatrix, starting at $s$. 
Follow forward and backward each vertical separatrix from each singularity (these two directions are the same when the singularity has conical angle $\pi$), stopping at the first intersection with $I$. We get a finite collection of points: $s_1, \dots, s_n$ (in this order when we follow $I$ from $s$).
Reduce $I$ so that it starts at $s$ and ends at the last intersection $s_n$.
We can build a zippered rectangles representation of the surface above $I$ and the $s_i$ correspond to the base points between the rectangles and the points  between the identifications of the top of the rectangles (see \Cref{subsec:background_zip-rect} for background).

For every $Y\in\mathcal{U}$, denote by $s(Y)\in Y$ the singularity corresponding to $s\in X$. Up to replacing $\mathcal{U}$ by a smaller neighborhood of $X$, we can assume that every surface $Y\in\mathcal{U}$ has no horizontal or vertical saddle connection. Hence, the horizontal separatrix going out from $s(Y)$ is not a saddle connection. Then we can build a segment $I(Y)\subset Y$ in $Y$ as  we have built $I$ in $X$.

Consider the zippered rectangles presentation of $Y$. Up to considering a smaller neighborhood, we can assume that all the zippered rectangles in $\mathcal{U}$ have the same combinatorial data.
For each surface, there is an upper and a lower bound for the height of the rectangles, say $K(Y)$ and $K(Y)^{-1}$. Since 
these bounds depend continuously on the surface $Y$,
 the constant $K=\underset{Y\in\overline{\mathcal{U}}}{\max} \, K(Y) \in (0,\infty)$ is well defined and satisfies the first two points of the lemma, together with $I(Y)$.
Moreover, the curves $(\zeta_1,\zeta_2)$ are not horizontal since $X$ admits no horizontal saddle connection. Up to increasing $K$, we have also $K^{-1} <  l_{\text{v}}\left(\zeta_j(Y)\right) < K $ for every $Y\in\mathcal{U}$ and $j\in\{1,2\}$. 
\end{proof}

Denote by $(t_n)_{n\in\NN}$ the sequence of return times of $g_t.X$ in $\mathcal{U}$, with the convention $t_0 = 0$.
We highlight that the sequence $(t_n)$ is well defined and infinite because the orbit of $X$ is recurrent (we pick $X$ to be Masur generic).
We note that the segment $I$ can be taken as short as we want (up to increasing $K$). We choose it so that it is always possible to connect two points in $X$ by a path without intersecting $\xi_1$, $\xi_2$ nor $I$.

\begin{dfn}
\label{def:zetak}
For $k\in\NN$, we define 
$$
\zeta_i^{(k)} = A^{KZ}(t_k,X).\zeta_i . 
$$
\end{dfn}
In particular $\left(\zeta_1^{(0)},\zeta_2^{(0)}\right)=\left(\zeta_1,\zeta_2\right)$.

\begin{lem}\label{lem:len_zeta}
For any $k\in\NN$ and $j\in\{1,2\}$, 
$$K^{-1} e^{t_k} \leqslant l_{\text{v}}(\zeta_j^{(k)}) \leqslant K e^{t_k}.$$
\end{lem}
\begin{proof}
This follows directly from the fact that
 $
 l_{\text{v}}\left(\zeta_j^{(k)}\right) = e^{t_k}  l_{\text{v}}\left(\zeta_j\right).
 $
\end{proof}

\subsection{Decomposition of the cycle}\label{subsec:decompo}

Before being able to prove the inequalities on the coefficients of the decomposition of \Cref{lem:decomp_traj}, we need to estimate the vertical length of a curve depending on its coefficients (in homology).
\begin{lem}\label{lem:lower_bound_vert_length}
    There exist $c>0$ such that for every $\gamma \in H_1(X,\RR) $ with coordinates $(n_1,n_2)$ in the basis $(\zeta_1,\zeta_2)$, we have
    $$ l_{\textrm{v}}(\gamma) \geqslant c\sum_{i=1}^{2} |n_{i}|l_{\textrm{v}}(\zeta_i).$$

    This constant $c$ can be chosen uniformly for all surfaces in $\mathcal{U}$ of \Cref{lem:cte_K}.
\end{lem}

\begin{proof}
    Consider the infinite half-translation surface $S$ covering of $X$. It is $\RR^2\setminus\mathcal{P}$ with identifications as in $X$, see \Cref{fig:S}. In order to work with an orientable vertical foliation, we define $\tilde S$, a translation cover of $S$, as follows. We take two copies of $\RR^2\setminus\mathcal{P}$, one being rotated by half a turn compared to the other. We call the first one the positive copy and the second one the negative copy. A path going up in the positive (resp. negative) copy would project in the plane to a trajectory going up (resp. down) for the wind-tree tilling billiards. Then, we identify the slits as in \Cref{fig:S_tilde}. We get an infinite translation surface. 
Each of these slits is in a parallelogram shaped fundamental domain given by $u_1$ and $u_2$. The set of these fundamental domains is in bijection with $\ZZ^2\times\ZZ/2\ZZ$: each element of $\ZZ/2\ZZ$ corresponds to a copy of $S$, and the element of $\ZZ^2$ is given as follows. Assign $(0,0)$ to one fundamental domain. 
Let $\gamma$ be a loop on $X$ and $z\in\ZZ^2$ be its algebraic intersection number with $\zeta_1,\zeta_2$, let  $\tilde\gamma$ be a lift of $\gamma$ to $\tilde S$ that starts in the fundamental domain $(0,0)$. Assign $z$ to the fundamental domain where $\tilde\gamma$ stops.  See \Cref{fig:bij_fd_Z2}.
     In other words, the surface $\tilde S$ is a $\ZZ^2\times\ZZ/2\ZZ$ ramified cover of $X$ (the ramification points are exactly the singularities).

         \begin{figure}[h]
        \centering
             \def\eps{0.03}
% corners of the slit
\def\xa{3.9}
\def\ya{1.3}
\def\xb{-1}
\def\yb{3}
\def\xp{-0.72}
\def\xpp{3.62}
% basis of Lamda
\def\ux{4.4} % x coord of vector u in \Lambda
\def\uy{0.7} % y coord of vector u in \Lambda
\def\vx{-2.5} % x coord of vector v in \Lambda
\def\vy{3.4} % x coord of vector v in \Lambda
\def\lx{0.6}
\def\ly{0.4}
\begin{tikzpicture}[>=stealth,x=.35cm,y=.35cm]
\foreach \k in {0,1,2}
\foreach \l in {0,1,2}
{\begin{scope}[shift={(\k*\ux+\l*\vx,\k*\uy+\l*\vy)}]
% slit
\draw [color=colleft,shift={(0,-\eps)}, line width= 1pt] (-1,3) -- (1,1.80612);
\draw [line width=0.75pt] (0,2.4-\eps-0.15)--(0,2.4-\eps+0.05);
\draw [color=colright,shift={(0,+\eps)}, line width= 1pt] (1.9,2.49388) -- (3.9,1.3);
\draw [line width=0.75pt] (2.9,1.9+\eps-0.05)--(2.9,1.9+\eps+0.15);
\draw [color=colmid,shift={(0,+\eps)}, line width= 1pt] (-1,3) -- (1.9,2.49388);
\draw [color=colmid,shift={(0,-\eps)}, line width= 1pt] (1,1.80612) -- (3.9,1.3);
\fill [gray!25] (\xa,\ya) -- (1,1.80612) -- (\xb,\yb) -- (1.9,2.49388) -- cycle;
\end{scope}
}

\draw [densely dotted,shift={(1.5*\vx,+1.5*\vy)}] (0.5,0.3) --++ (-0.3*\ux,-0.3*\uy); %dots on the left
\draw [densely dotted,shift={(3*\ux+1.5*\vx,3*\uy+1.5*\vy)}] (0.5,0.3) --++ (0.3*\ux,0.3*\uy); %dots on the right
\draw [densely dotted,shift={(1.5*\ux,+1.5*\uy)}] (0.5,0.3) --++ (-0.3*\vx,-0.3*\vy); %dots below
\draw [densely dotted,shift={(3*\vx+1.5*\ux,3*\vy+1.5*\uy)}] (0.5,0.3) --++ (0.3*\vx,0.3*\vy); %dots above
\end{tikzpicture}
        \caption{The infinite half-translation surface $S$ covering $X$}\label{fig:S}
{\small Within each cut, the top side is identified by translation with the bottom side, the left side is identified with itself by a half-translation and the right side as well. The gray part is cut off the plane, it does not belong to the surface $S$. }
    \end{figure}
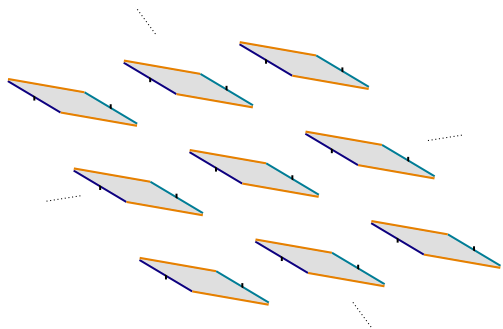
        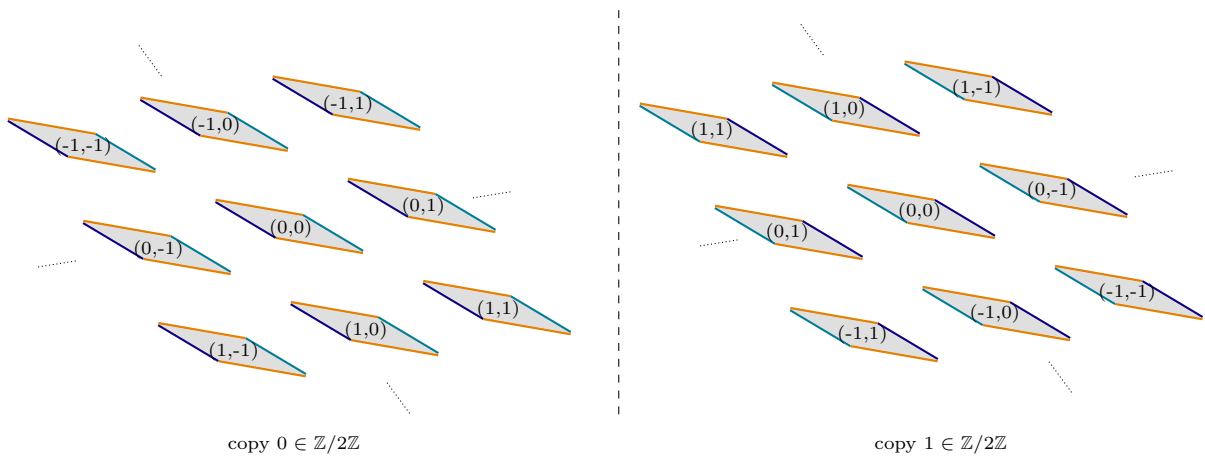
\begin{figure}[h]
        \centering
             %\definecolor{xdxdff}{rgb}{0.49,0.49,1}
\def\eps{0.03}
% corners of the slit
\def\xa{3.9}
\def\ya{1.3}
\def\xb{-1}
\def\yb{3}
\def\xp{-0.72}
\def\xpp{3.62}
% basis of Lamda
\def\ux{4.4} % x coord of vector u in \Lambda
\def\uy{0.7} % y coord of vector u in \Lambda
\def\vx{-2.5} % x coord of vector v in \Lambda
\def\vy{3.4} % x coord of vector v in \Lambda
\def\lx{0.6}
\def\ly{0.4}
\begin{tikzpicture}[>=stealth,x=.4cm,y=.4cm]

\foreach \k in {0,1,2}
\foreach \l in {0,1,2}
{\begin{scope}[shift={(\k*\ux+\l*\vx,\k*\uy+\l*\vy)}]
% slit
\draw [color=colleft,shift={(0,-\eps)}, line width= 1pt] (-1,3) -- (1,1.80612);
\draw [color=colright,shift={(0,+\eps)}, line width= 1pt] (1.9,2.49388) -- (3.9,1.3);
\draw [color=colmid,shift={(0,+\eps)}, line width= 1pt] (-1,3) -- (1.9,2.49388);
\draw [color=colmid,shift={(0,-\eps)}, line width= 1pt] (1,1.80612) -- (3.9,1.3);
\fill [gray!25] (\xa,\ya) -- (1,1.80612) -- (\xb,\yb) -- (1.9,2.49388) -- cycle;
\end{scope}
}

\foreach \k in {-1,0,1}%,3,4,5,6,7,8,9}
\foreach \l in {1}%,3,4,5,6,7,8,9}
{\begin{scope}[shift={(\k*\ux+\ux+\vx+\l*\vx,\k*\uy+\uy+\vy+\l*\vy)}]
\draw (1.5,2.1) node {{\tiny(-\l,\k)}};
\end{scope}
}
\foreach \k in {-1,0,1}
\foreach \l in {0,1}
{\begin{scope}[shift={(\k*\ux+\ux+\vx-\l*\vx,\k*\uy+\uy+\vy-\l*\vy)}]
\draw (1.5,2.1) node {{\tiny(\l,\k)}};
\end{scope}
}

\draw [densely dotted,shift={(1.5*\vx,+1.5*\vy)}] (0,0) --++ (-0.3*\ux,-0.3*\uy); %dots on the left
\draw [densely dotted,shift={(3*\ux+1.5*\vx,3*\uy+1.5*\vy)}] (0,0) --++ (0.3*\ux,0.3*\uy); %dots on the right
\draw [densely dotted,shift={(1.5*\ux,+1.5*\uy)}] (0,0) --++ (-0.3*\vx,-0.3*\vy); %dots below
\draw [densely dotted,shift={(3*\vx+1.5*\ux,3*\vy+1.5*\uy)}] (0,0) --++ (0.3*\vx,0.3*\vy); %dots above

\draw[dashed] (3.25*\ux,0) --++ (0,4*\vy);
% copy for the translation cover
\begin{scope}[shift={(2*\ux,-2*\uy)},rotate around={180:(3*\ux+1.5*\vx,3*\uy+1.5*\vy))}]
\foreach \k in {0,1,2}
\foreach \l in {0,1,2}
{\begin{scope}[shift={(\k*\ux+\l*\vx,\k*\uy+\l*\vy)}]
% slit
\draw [color=colleft,shift={(0,-\eps)}, line width= 1pt] (-1,3) -- (1,1.80612);
\draw [color=colright,shift={(0,+\eps)}, line width= 1pt] (1.9,2.49388) -- (3.9,1.3);
\draw [color=colmid,shift={(0,+\eps)}, line width= 1pt] (-1,3) -- (1.9,2.49388);
\draw [color=colmid,shift={(0,-\eps)}, line width= 1pt] (1,1.80612) -- (3.9,1.3);
\fill [gray!25] (\xa,\ya) -- (1,1.80612) -- (\xb,\yb) -- (1.9,2.49388) -- cycle;
\end{scope}
}

\foreach \k in {-1,0,1}
\foreach \l in {1}
{\begin{scope}[shift={(\k*\ux+\ux+\vx+\l*\vx,\k*\uy+\uy+\vy+\l*\vy)}]
\draw (1.5,2.15) node {{\tiny(-\l,\k)}};
\end{scope}
}
\foreach \k in {-1,0,1}
\foreach \l in {0,1}
{\begin{scope}[shift={(\k*\ux+\ux+\vx-\l*\vx,\k*\uy+\uy+\vy-\l*\vy)}]
\draw (1.5,2.15) node {{\tiny(\l,\k)}};
\end{scope}
}

\draw [densely dotted,shift={(1.5*\vx,+1.5*\vy)}] (0,0) --++ (-0.3*\ux,-0.3*\uy); %dots on the left
\draw [densely dotted,shift={(3*\ux+1.5*\vx,3*\uy+1.5*\vy)}] (0,0) --++ (0.3*\ux,0.3*\uy); %dots on the right
\draw [densely dotted,shift={(1.5*\ux,+1.5*\uy)}] (0,0) --++ (-0.3*\vx,-0.3*\vy); %dots below
\draw [densely dotted,shift={(3*\vx+1.5*\ux,3*\vy+1.5*\uy)}] (0,0) --++ (0.3*\vx,0.3*\vy); %dots above
\end{scope}

\draw (3.5,-1) node {{\tiny copy $0\in\ZZ/2\ZZ$}};
\draw (25,-1) node {{\tiny copy $1\in\ZZ/2\ZZ$}};
\end{tikzpicture}
        \caption{The translation cover $\tilde S$ of $S$}\label{fig:S_tilde}
        {\small The surface $\tilde S$ is obtained as the gluing of two planes with each countably many cuts. The cuts of each plane are in bijection with $\ZZ^2$. The bottom side of each cut if identified by translation to the top side of the same cut. The left side of the cut $z\in\ZZ^ 2$ on one plane is identified by translation to the right side of the cut $z$ in the other plane. }
    \end{figure}

    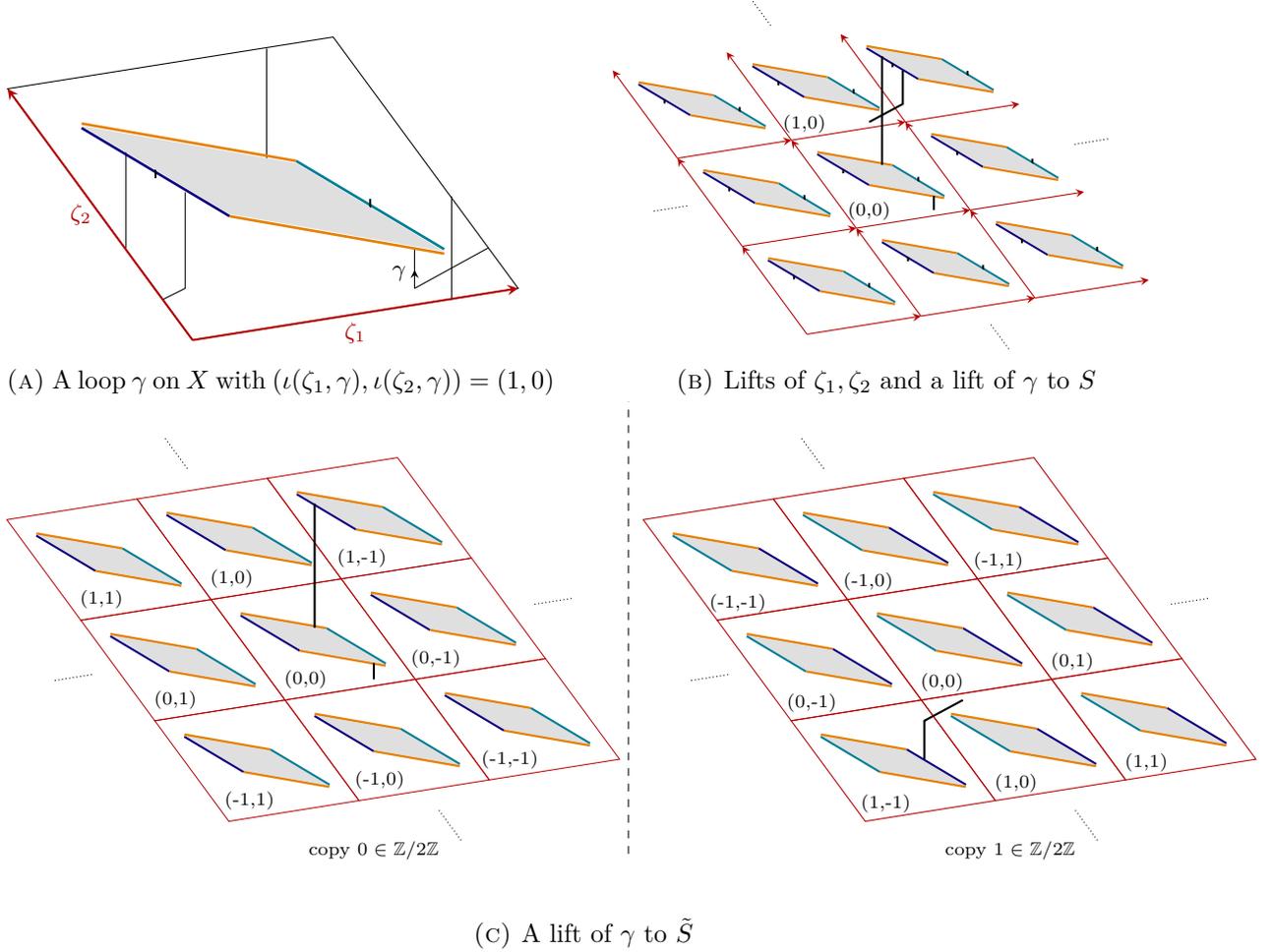
\begin{figure}[h]
        \centering
        \begin{subfigure}[b]{0.475\textwidth}
            \definecolor{colangle}{rgb}{0.7,0,0}
\def\eps{0.03}
% corners of the slit
\def\xa{3.9}
\def\ya{1.3}
\def\xb{-1}
\def\yb{3}
%vectors defining the slit
\def\slopeTopx{2.9}
\def\slopeTopy{0.50612}
\def\slopeLeftx{-2}
\def\slopeLefty{1.19388}
%coordinate of the left corner of the fd
\def\cx{0.5}
\def\cy{0.1}
\def\xp{-0.72}
\def\xpp{3.62}
% basis of Lamda
\def\ux{4.4} % x coord of vector u in \Lambda
\def\uy{0.7} % y coord of vector u in \Lambda
\def\vx{-2.5} % x coord of vector v in \Lambda
\def\vy{3.4} % x coord of vector v in \Lambda
\begin{tikzpicture}[>=stealth,x=1cm,y=1cm]

%traj
\draw[->] (\cx,\cy) ++ (3,0.7) --++ (0,0.25); 
\draw (\cx,\cy) ++ (3,0.9) --++ (0,0.35) ++ (\slopeLeftx,\slopeLefty) --++ (0,1.5) ++  (-\vx,-\vy) --++ (0,1.4) + + (-\ux,-\uy) --++ (0,1.3) ;
\draw (0.4,2.1) --++ (0,-1.3) --++ (-0.3,-0.15) ++ (\ux,\uy) -- (3+\cx,0.7+\cy);

%torus
\draw (\cx,\cy) --++ (\ux,\uy) --++ (\vx,\vy) --++ (-\ux,-\uy) -- cycle;

% slit
\draw [color=colleft,shift={(0,-\eps)}, line width= 1pt] (-1,3) -- (1,1.8061);
\draw [line width=0.75pt] (0,2.4-\eps-0.08)--(0,2.4-\eps+0.05);
\draw [color=colright,shift={(0,+\eps)}, line width= 1pt] (1.9,2.49388) -- (3.9,1.3);
\draw [line width=0.75pt] (2.9,1.9+\eps-0.05)--(2.9,1.9+\eps+0.08);
\draw [color=colmid,shift={(0,+\eps)}, line width= 1pt] (-1,3) -- (1.9,2.49388);
\draw [color=colmid,shift={(0,-\eps)}, line width= 1pt] (1,1.80612) -- (3.9,1.3);
\fill [gray!25] (\xa,\ya) -- (1,1.80612) -- (\xb,\yb) -- (1.9,2.49388) -- cycle;

\draw [colangle,thick,->] (\cx,\cy) --++ (\ux,\uy);
\draw [colangle,thick,->] (\cx,\cy) --++ (\vx,\vy);
\begin{footnotesize}
    \draw [colangle,below] (\cx+\ux/2,\cy+\uy/2) node {$\zeta_1$};
    \draw [colangle,left] (\cx+\vx/2,\cy+\vy/2) node {$\zeta_2$};
    \draw [left] (\cx+3,\cy+0.9) node {$\gamma$};
\end{footnotesize}
\end{tikzpicture}
            \caption{A loop $\gamma$ on $X$ with $(\iota(\zeta_1,\gamma),\iota(\zeta_2,\gamma))=(1,0)$}\label{fig:gamma_on_X}
        \end{subfigure}
        \hfill
        \begin{subfigure}[b]{0.475\textwidth}
            \definecolor{colangle}{rgb}{0.7,0,0}
\def\eps{0.03}
% corners of the slit
\def\xa{3.9}
\def\ya{1.3}
\def\xb{-1}
\def\yb{3}
\def\xp{-0.72}
\def\xpp{3.62}
%coordinate of the left corner of the fd
\def\cx{0.5}
\def\cy{0.1}
% basis of Lamda
\def\ux{4.4} % x coord of vector u in \Lambda
\def\uy{0.7} % y coord of vector u in \Lambda
\def\vx{-2.5} % x coord of vector v in \Lambda
\def\vy{3.4} % x coord of vector v in \Lambda
\def\lx{0.6}
\def\ly{0.4}
%vectors defining the slit
\def\slopeTopx{2.9}
\def\slopeTopy{0.50612}
\def\slopeLeftx{-2}
\def\slopeLefty{1.19388}
\begin{tikzpicture}[>=stealth,x=.35cm,y=.35cm]

\foreach \k in {0,1,2}
\foreach \l in {0,1,2}
{\begin{scope}[shift={(\k*\ux+\l*\vx,\k*\uy+\l*\vy)}]
\draw [colangle,->] (\cx,\cy) --++ (\ux,\uy) ;
\draw [colangle,->] (\cx,\cy) --++ (\vx,\vy);
% slit
\draw [color=colleft,shift={(0,-\eps)}, line width= 1pt] (-1,3) -- (1,1.80612);
\draw [line width=0.75pt] (0,2.4-\eps-0.15)--(0,2.4-\eps+0.05);
\draw [color=colright,shift={(0,+\eps)}, line width= 1pt] (1.9,2.49388) -- (3.9,1.3);
\draw [line width=0.75pt] (2.9,1.9+\eps-0.05)--(2.9,1.9+\eps+0.15);
\draw [color=colmid,shift={(0,+\eps)}, line width= 1pt] (-1,3) -- (1.9,2.49388);
\draw [color=colmid,shift={(0,-\eps)}, line width= 1pt] (1,1.80612) -- (3.9,1.3);
\fill [gray!25] (\xa,\ya) -- (1,1.80612) -- (\xb,\yb) -- (1.9,2.49388) -- cycle;
\end{scope}
}

\begin{scope}[shift={(\ux+\vx,\uy+\vy)}]
%traj
\draw [thick] (\cx,\cy) ++ (3,0.7) --++ (0,0.55) ++ (\slopeLeftx,\slopeLefty) --++ (0,1.5) --++ (0,1.4) --++ (0,1.3) ;
\draw [thick,shift={(\ux+\vx,\uy+\vy)}] (0.4,2.1) --++ (0,-1.3)  -- (3+\cx-\ux,0.7+\cy-\uy);
%\draw [red,shift={(\ux+\vx,\uy+\vy)}] (0.4,2.1) ++ (0,-1.3) ++ (-0.3,-0.15) --++(1,0);
\draw (1,0.75) node {{\tiny(0,0)}};
\draw (1+\vx,0.75+\vy) node {{\tiny(1,0)}};
\end{scope}

\draw [densely dotted,shift={(1.5*\vx,+1.5*\vy)}] (0.5,0.1) ++ (-0.2*\ux,-0.2*\uy) --++ (-0.3*\ux,-0.3*\uy); %dots on the left
\draw [densely dotted,shift={(3*\ux+1.5*\vx,3*\uy+1.5*\vy)}] (0.5,0.1) ++ (0.2*\ux,0.2*\uy) --++ (0.3*\ux,0.3*\uy); %dots on the right
\draw [densely dotted,shift={(1.5*\ux,+1.5*\uy)}] (0.5,0.1) ++ (-0.2*\vx,-0.2*\vy) --++ (-0.3*\vx,-0.3*\vy); %dots below
\draw [densely dotted,shift={(3*\vx+1.5*\ux,3*\vy+1.5*\uy)}] (0.5,0.1) ++ (0.2*\vx,0.2*\vy) --++ (0.3*\vx,0.3*\vy); %dots above

\end{tikzpicture}
            \caption{Lifts of $\zeta_1,\zeta_2$ and a lift of $\gamma$ to $S$}\label{fig:gamma_on_S}
        \end{subfigure}
        
        \begin{subfigure}[b]{1\textwidth}
             \definecolor{colangle}{rgb}{0.7,0,0}
\def\eps{0.03}
% corners of the slit
\def\xa{3.9}
\def\ya{1.3}
\def\xb{-1}
\def\yb{3}
\def\xp{-0.72}
\def\xpp{3.62}
%coordinate of the left corner of the fd
\def\cx{0.5}
\def\cy{0.1}
% basis of Lamda
\def\ux{4.4} % x coord of vector u in \Lambda
\def\uy{0.7} % y coord of vector u in \Lambda
\def\vx{-2.5} % x coord of vector v in \Lambda
\def\vy{3.4} % x coord of vector v in \Lambda
\def\lx{0.6}
\def\ly{0.4}
%vectors defining the slit
\def\slopeTopx{2.9}
\def\slopeTopy{0.50612}
\def\slopeLeftx{-2}
\def\slopeLefty{1.19388}
\begin{tikzpicture}[>=stealth,x=.4cm,y=.4cm]
\foreach \k in {-1,0,1}
\foreach \l in {-1,0,1}
{\begin{scope}[shift={(\k*\ux+\l*\vx,\k*\uy+\l*\vy)}]
\draw [colangle,very thin] (\cx,\cy) --++ (\ux,\uy) --++ (\vx,\vy) --++ (-\ux,-\uy) -- cycle;
% slit
\draw [color=colleft,shift={(0,-\eps)}, line width= 1pt] (-1,3) -- (1,1.80612);
\draw [color=colright,shift={(0,+\eps)}, line width= 1pt] (1.9,2.49388) -- (3.9,1.3);
\draw [color=colmid,shift={(0,+\eps)}, line width= 1pt] (-1,3) -- (1.9,2.49388);
\draw [color=colmid,shift={(0,-\eps)}, line width= 1pt] (1,1.80612) -- (3.9,1.3);
\fill [gray!25] (\xa,\ya) -- (1,1.80612) -- (\xb,\yb) -- (1.9,2.49388) -- cycle;
\end{scope}
}

\foreach \k in {1}
\foreach \l in {-1,0,1}
{\begin{scope}[shift={(\k*\ux+\l*\vx,\k*\uy+\l*\vy)}]
\draw (1.2,0.8) node {{\tiny(\l,-1)}};
\end{scope}
}
\foreach \k in {0}
\foreach \l in {-1,0,1}
{\begin{scope}[shift={(\k*\ux+\l*\vx,\k*\uy+\l*\vy)}]
\draw (1.2,0.8) node {{\tiny(\l,\k)}};
\end{scope}
}
\foreach \k in {-1}
\foreach \l in {-1,0,1}
{\begin{scope}[shift={(\k*\ux+\l*\vx,\k*\uy+\l*\vy)}]
\draw (1.2,0.8) node {{\tiny(\l,1)}};
\end{scope}
}

\draw [densely dotted,shift={(0.5*\vx-\ux,+0.5*\vy-\uy)}] (0.5,0.1) ++ (-0.2*\ux,-0.2*\uy) --++ (-0.3*\ux,-0.3*\uy); %dots on the left
\draw [densely dotted,shift={(2*\ux+0.5*\vx,2*\uy+0.5*\vy)}] (0.5,0.1) ++ (0.2*\ux,0.2*\uy) --++ (0.3*\ux,0.3*\uy); %dots on the right
\draw [densely dotted,shift={(0.5*\ux-\vx,+0.5*\uy-\vy)}] (0.5,0.1) ++ (-0.2*\vx,-0.2*\vy) --++ (-0.3*\vx,-0.3*\vy); %dots below
\draw [densely dotted,shift={(2*\vx+0.5*\ux,2*\vy+0.5*\uy)}] (0.5,0.1) ++ (0.2*\vx,0.2*\vy) --++ (0.3*\vx,0.3*\vy); %dots above

\draw (3.5,-5) node {{\tiny copy $0\in\ZZ/2\ZZ$}};

%traj first part
\draw [thick] (\cx,\cy) ++ (3,0.7) --++ (0,0.55) ++ (\slopeLeftx,\slopeLefty) --++ (0,1.5) --++ (0,1.4) --++ (0,1.3) ;

\draw[dashed] (2.75*\ux,-1.5*\vy) --++ (0,4.5*\vy);

% copy for the translation cover
\begin{scope}[rotate around={180:(\cx+0.5*\ux+0.5*\vx,\cy+0.5*\uy+0.5*\vy))},shift={(\cx-5*\ux,0)}]
\foreach \k in {-1,0,1}
\foreach \l in {-1,0,1}
{\begin{scope}[shift={(\k*\ux+\l*\vx,\k*\uy+\l*\vy)}]
\draw [colangle,thin] (\cx,\cy) --++ (\ux,\uy) --++ (\vx,\vy) --++ (-\ux,-\uy) -- cycle;
% slit
\draw [color=colleft,shift={(0,-\eps)}, line width= 1pt] (-1,3) -- (1,1.80612);
\draw [color=colright,shift={(0,+\eps)}, line width= 1pt] (1.9,2.49388) -- (3.9,1.3);
\draw [color=colmid,shift={(0,+\eps)}, line width= 1pt] (-1,3) -- (1.9,2.49388);
\draw [color=colmid,shift={(0,-\eps)}, line width= 1pt] (1,1.80612) -- (3.9,1.3);
\fill [gray!25] (\xa,\ya) -- (1,1.80612) -- (\xb,\yb) -- (1.9,2.49388) -- cycle;
\end{scope}
}

\draw [densely dotted,shift={(0.5*\vx-\ux,+0.5*\vy-\uy)}] (0.5,0.1) ++ (-0.2*\ux,-0.2*\uy) --++ (-0.3*\ux,-0.3*\uy); %dots on the right
\draw [densely dotted,shift={(2*\ux+0.5*\vx,2*\uy+0.5*\vy)}] (0.5,0.1) ++ (0.2*\ux,0.2*\uy) --++ (0.3*\ux,0.3*\uy); %dots on the lett
\draw [densely dotted,shift={(0.5*\ux-\vx,+0.5*\uy-\vy)}] (0.5,0.1) ++ (-0.2*\vx,-0.2*\vy) --++ (-0.3*\vx,-0.3*\vy); %dots above
\draw [densely dotted,shift={(2*\vx+0.5*\ux,2*\vy+0.5*\uy)}] (0.5,0.1) ++ (0.2*\vx,0.2*\vy) --++ (0.3*\vx,0.3*\vy); %dots below

\foreach \k in {1}
\foreach \l in {-1,0,1}
{\begin{scope}[shift={(\k*\ux+\l*\vx,\k*\uy+\l*\vy)}]
\draw (1.7,3.6) node {{\tiny(\l,-1)}};
\end{scope}
}
\foreach \k in {0}
\foreach \l in {-1,0,1}
{\begin{scope}[shift={(\k*\ux+\l*\vx,\k*\uy+\l*\vy)}]
\draw (1.7,3.6) node {{\tiny(\l,\k)}};
\end{scope}
}
\foreach \k in {-1}
\foreach \l in {-1,0,1}
{\begin{scope}[shift={(\k*\ux+\l*\vx,\k*\uy+\l*\vy)}]
\draw (1.7,3.6) node {{\tiny(\l,1)}};
\end{scope}
}
%traj second part
\draw [thick,shift={(\ux+\vx,\uy+\vy)}] (0.4,2.1) --++ (0,-1.3)  -- (3+\cx-\ux,0.7+\cy-\uy);
\end{scope}

\draw (25,-5) node {{\tiny copy $1\in\ZZ/2\ZZ$}};
\end{tikzpicture}
             \caption{A lift of $\gamma$ to $\tilde{S}$}\label{fig:gamma_on_S_tilde}
        \end{subfigure}
        \caption{The bijection between the translated fundamental domains of $\tilde S$ and $\ZZ^2\times\ZZ/2\ZZ$}\label{fig:bij_fd_Z2}
    \end{figure}

    Let $\tilde\gamma$ be a lift of $\gamma$ starting in the fundamental domain $(0,0,0)$. It is a union of vertical segments, with a break each time that $\tilde\gamma$ intersect a slit, see \Cref{fig:gamma_on_S_tilde}. When $\tilde\gamma$ intersects the part of the slit illustrated in orange in the figures, it stays in the same copy of $\RR^2\setminus\mathcal{P}$. When it intersects a (deep or light) blue part, it changes from a copy to the other one.
    Note that the length of $\tilde\gamma$ does not change if we translate some parts of $\tilde\gamma$. 
    We translate if necessary the segments of $\tilde \gamma$ to ensure that whenever $\tilde \gamma$ comes back to the previous copy, it comes back in the previous fundamental domain. See \Cref{fig:modified_lift}. We translate also the segments of $\tilde\gamma$ in the negative copy of $\RR^2 \setminus \mathcal{P}$ so that the first crossed fundamental domain of this copy is $(0,0,1)$. The projection to $S$ is not a path anymore, but a union of paths, with in total the same vertical length. 
     We denote by $(p,q,0)\in \ZZ^2\times \ZZ/2\ZZ$ the last fundamental domain crossed by $\tilde\gamma$ in the positive copy of $\RR^\setminus \mathcal{P}$, and $(p',q',1)$ the one in the negative copy. Note that the coordinates $(n_1,n_2)$ of $\gamma$ in the basis $(\zeta_1,\zeta_2)$ are such that $n_1=p+p'$ and $n_2 = q+q'$. See \Cref{fig:modified_lift}.

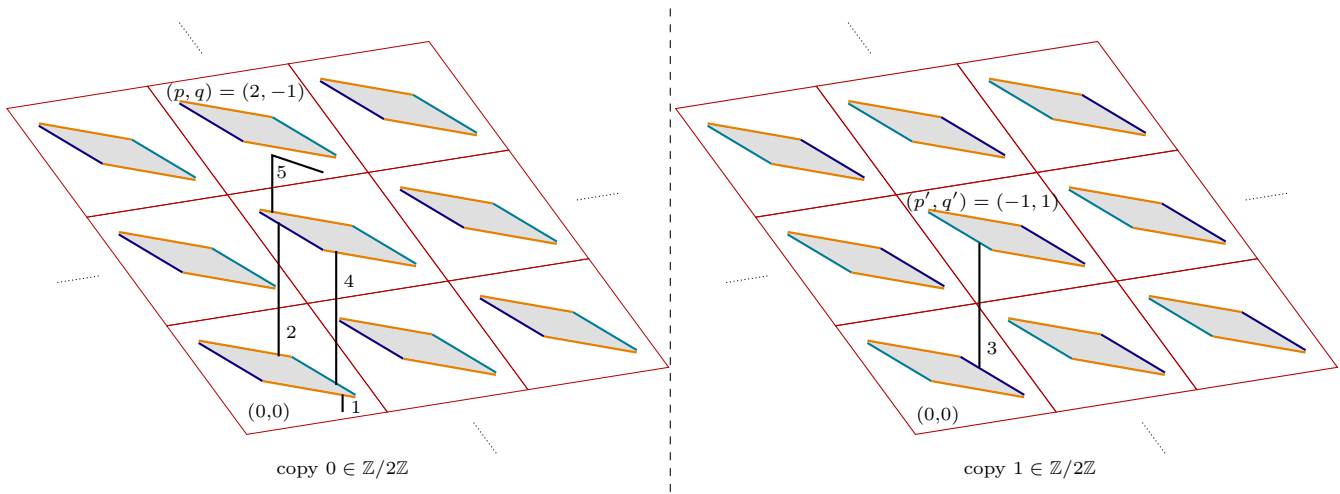
\begin{figure}[h]
    \centering
    \definecolor{colangle}{rgb}{0.7,0,0}
\def\eps{0.03}
% corners of the slit
\def\xa{3.9}
\def\ya{1.3}
\def\xb{-1}
\def\yb{3}
\def\xp{-0.72}
\def\xpp{3.62}
%coordinate of the left corner of the fd
\def\cx{0.5}
\def\cy{0.1}
% basis of Lamda
\def\ux{4.4} % x coord of vector u in \Lambda
\def\uy{0.7} % y coord of vector u in \Lambda
\def\vx{-2.5} % x coord of vector v in \Lambda
\def\vy{3.4} % x coord of vector v in \Lambda
\def\lx{0.6}
\def\ly{0.4}
%vectors defining the slit
\def\slopeTopx{2.9}
\def\slopeTopy{-0.50612}
\def\slopeLeftx{-2}
\def\slopeLefty{1.19388}
\begin{tikzpicture}[>=stealth,x=.425cm,y=.425cm]

\foreach \k in {0,1,2}
\foreach \l in {0,1,2}
{\begin{scope}[shift={(\k*\ux+\l*\vx,\k*\uy+\l*\vy)}]
\draw [colangle,very thin] (\cx,\cy) --++ (\ux,\uy) --++ (\vx,\vy) --++ (-\ux,-\uy) -- cycle;
% slit
\draw [color=colleft,shift={(0,-\eps)}, line width= 1pt] (-1,3) -- (1,1.80612);
\draw [color=colright,shift={(0,+\eps)}, line width= 1pt] (1.9,2.49388) -- (3.9,1.3);
\draw [color=colmid,shift={(0,+\eps)}, line width= 1pt] (-1,3) -- (1.9,2.49388);
\draw [color=colmid,shift={(0,-\eps)}, line width= 1pt] (1,1.80612) -- (3.9,1.3);
\fill [gray!25] (\xa,\ya) -- (1,1.80612) -- (\xb,\yb) -- (1.9,2.49388) -- cycle;
\end{scope}
}
\draw (1.2,0.8) node {{\tiny(0,0)}};
\draw (0.75+\ux+2*\vx,3.3+\uy+2*\vy) node {{\tiny$(p,q)=(2,-1)$}};

\draw [densely dotted,shift={(1.5*\vx,1.5*\vy)}] (0.5,0.1) ++ (-0.2*\ux,-0.2*\uy) --++ (-0.3*\ux,-0.3*\uy); %dots on the left
\draw [densely dotted,shift={(3*\ux+1.5*\vx,3*\uy+1.5*\vy)}] (0.5,0.1) ++ (0.2*\ux,0.2*\uy) --++ (0.3*\ux,0.3*\uy); %dots on the right
\draw [densely dotted,shift={(1.5*\ux,+1.5*\uy)}] (0.5,0.1) ++ (-0.2*\vx,-0.2*\vy) --++ (-0.3*\vx,-0.3*\vy); %dots below
\draw [densely dotted,shift={(3*\vx+1.5*\ux,3*\vy+1.5*\uy)}] (0.5,0.1) ++ (0.2*\vx,0.2*\vy) --++ (0.3*\vx,0.3*\vy); %dots above

\draw (3.5,-1) node {{\tiny copy $0\in\ZZ/2\ZZ$}};

%traj 1st part
\draw [thick] (\cx,\cy) ++ (3,0.7) --++ (0,0.55) ++ (\slopeLeftx,\slopeLefty) --++ (0,1.5) --++ (0,1.4) --++ (0,1.3) ;
%traj 3rd part
\draw [thick,shift={(\slopeTopx,\slopeTopy)}] (0.4,2.15) --++ (0,4.2) ++ (\slopeLeftx,\slopeLefty) --++ (0,1.8) -- (3+\cx+\ux+2*\vx-\slopeTopx,0.7+\cy+\uy+2*\vy-\slopeTopy);
\begin{tiny}
\draw[thick,right] (\cx,\cy) ++ (3,0.85) node {$1$};
\draw[thick] (1.9,3.3) node {$2$};
\draw[thick] (-0.7+\ux,4.2+\uy) node {$4$};
\draw[thick] (-0.3+\ux+\vx,4.2+\uy+\vy) node {$5$};
\end{tiny}

\draw[dashed] (3.125*\ux,-0.5*\vy) --++ (0,4.5*\vy);

% copy for the translation cover
\begin{scope}[rotate around={180:(\cx+0.5*\ux+0.5*\vx,\cy+1.*\uy+1.*\vy))},shift={(\cx-5.3*\ux,0)}]
\foreach \k in {-1,0,1}
\foreach \l in {-1,0,1}
{\begin{scope}[shift={(\k*\ux+\l*\vx,\k*\uy+\l*\vy)}]
\draw [colangle,thin] (\cx,\cy) --++ (\ux,\uy) --++ (\vx,\vy) --++ (-\ux,-\uy) -- cycle;
% slit
\draw [color=colleft,shift={(0,-\eps)}, line width= 1pt] (-1,3) -- (1,1.80612);
\draw [color=colright,shift={(0,+\eps)}, line width= 1pt] (1.9,2.49388) -- (3.9,1.3);
\draw [color=colmid,shift={(0,+\eps)}, line width= 1pt] (-1,3) -- (1.9,2.49388);
\draw [color=colmid,shift={(0,-\eps)}, line width= 1pt] (1,1.80612) -- (3.9,1.3);
\fill [gray!25] (\xa,\ya) -- (1,1.80612) -- (\xb,\yb) -- (1.9,2.49388) -- cycle;
\end{scope}
}

\draw [densely dotted,shift={(0.5*\vx-\ux,+0.5*\vy-\uy)}] (0.5,0.1) ++ (-0.2*\ux,-0.2*\uy) --++ (-0.3*\ux,-0.3*\uy); %dots on the right
\draw [densely dotted,shift={(2*\ux+0.5*\vx,2*\uy+0.5*\vy)}] (0.5,0.1) ++ (0.2*\ux,0.2*\uy) --++ (0.3*\ux,0.3*\uy); %dots on the lett
\draw [densely dotted,shift={(0.5*\ux-\vx,+0.5*\uy-\vy)}] (0.5,0.1) ++ (-0.2*\vx,-0.2*\vy) --++ (-0.3*\vx,-0.3*\vy); %dots above
\draw [densely dotted,shift={(2*\vx+0.5*\ux,2*\vy+0.5*\uy)}] (0.5,0.1) ++ (0.2*\vx,0.2*\vy) --++ (0.3*\vx,0.3*\vy); %dots below

\draw [shift={(\ux+\vx,\vy+\uy)}] (1.7,3.6) node {{\tiny(0,0)}};
\draw (2.2,1) node {{\tiny$(p',q')=(-1,1)$}};

%traj second part
\draw [thick,shift={(\ux+\vx,\vy+\uy)}] (0.4,2.1) --++ (0,-3.9);
\draw[thick,right] (0.5+\ux+\vx,1.5+\uy+\vy) node {{\tiny$3$}};
%traj 4th part
%\draw [thick,shift={(0,0)}] (0.4,-1.8) ++ (0,3.9) --++ (0,-3);
\end{scope}

\draw (25,-1) node {{\tiny copy $1\in\ZZ/2\ZZ$}};
\end{tikzpicture}
    \caption{Translating each part of the lift to estimate the displacement of the tiling billiards trajectory in $\RR^2$}
    \label{fig:modified_lift}
    {\small The numbers aside the black lines denote the order of the translated part of the lifted loop.}
\end{figure}
    
    In the plane $\RR^2$ tiled by parallelograms whose sides are given by $u_1$ and $u_2$, the vertical length of a path going to the parallelogram $(0,0)$ to the parallelogram $(p,q)$ is at least $|p| v_1 + |q| v_2$.
    
    In our case, we consider $\RR^2\setminus\mathcal{P}$, tiled by the parallelogram shaped fundamental domain whose sides are given by $u_1$ and $u_2$. We know that $\tilde\gamma$ can go through a slit in each fundamental (and at most once per fundamental domain).
    Going through a slit can produce a vertical jump of hight $v_4$ (going through the translation one, in blue on the figures) or at most $v_3$ 
    (when changing from a copy to the other, through the red and orange ones on the figures). We set $\delta=\max(v_3,v_4)$. Hence, the vertical length of $\tilde\gamma$ is at least $|p| (v_1 - \delta) +|q| (v_2-\delta)$.

    Since $\mathcal{O}_\varepsilon$ is bounded, there exist $v>0$ such that for every surface in $\mathcal{O}_\varepsilon$, $\max(v_1,v_2)<v$. Furthermore, $\min(v_1,v_2) > \varepsilon+\max(v_3,v_4)>\varepsilon$. If we take $c$ such that $cv<\varepsilon$, then all surfaces in $\mathcal{O}_\varepsilon$ satisfy that 
    \begin{align*}
       & cl_{\textrm{v}}(\zeta_i)=cv_1 \leqslant cv<\varepsilon\leqslant v_1-\delta \\
        & cl_{\textrm{v}}(\zeta_2)=cv_2\leqslant cv<\varepsilon\leqslant v_2-\delta.\
    \end{align*}
     We conclude that the vertical length of $\tilde\gamma$ on the positive copy is at least
     $|p| (v_1 - \delta) + |q| (v_2-\delta) \geqslant c \big(|p| l_{\textrm{v}}(\zeta_1)+|q|l_{\textrm{v}}(\zeta_2)\big) . $
The same reasoning on the negative copy gives that the vertical length of $\tilde\gamma$ on the negative copy is at least $ |p'| (v_1 - \delta) + |q'| (v_2-\delta) \geqslant c \big(|p'| l_{\textrm{v}}(\zeta_1)+|q'|l_{\textrm{v}}(\zeta_2)\big) . $
In total, we get
     $$  l_{\textrm{v}}(\gamma)  = l_{\textrm{v}}(\tilde\gamma) \geqslant (|p|+|p'|) (v_1 - \delta) + (|q|+|q'|) (v_2-\delta) \geqslant c\sum_{i=1}^2 |n_i| l_{\textrm{v}}(\zeta_i).$$
\end{proof}

The following Lemma is an adaptation to the case of half-translation surfaces ad of the homology surface (with its singularities) of Lemma 9.4 in \cite{Forni} for translation surfaces, rephrased in Lemma 9 of \cite{DHL}. It comes from the so called \emph{suffix-prefix decomposition}.

\begin{lem}\label{lem:decomp_traj}
There exists $T_0=T_0(X)$ such that for every $ T>T_0$, there exist an integer $n$, and a decomposition
$$
\forall p\in X,
\gamma_T(p) = \sum_{k=0}^{n}\sum_{j=1}^{2} m_j^{(k)}(p)\zeta_j^{(k)} \qquad \text{in } H_1(X,\RR)
$$
which satisfies:
\begin{itemize}
\item
for every $k\in\{0,\dots,n\}$, and $j\in\{1,2\}$, the $m_j^{(k)}(p)$ are (possibly negative) integers,
\item
the sum $|m_1^{(n)}(p)|+|m_2^{(n)}(p)|$ is not zero,
\item
for every $k\in\{0,\dots,n\}$, $\sum_{j=1}^{2} | m_j^{(k)}(p)| \leqslant 2 \frac{K^2}{c} e^{t_{k}-t_{k+1}}$,
\end{itemize}
where $K$ and $c$ are constants satisfying \Cref{lem:cte_K} and \Cref{lem:lower_bound_vert_length} respectively. 
\end{lem}

\begin{proof}
For each $k\in\NN$, denote by $I^{(k)}$ the subinterval of $I$ starting at $s$ and of length $e^{-t_k}l(I)$.
With this notation, the segment $I^{(k)}$ corresponds to the segment $I\left(g_{t_{k}}.X\right)$ given by \Cref{lem:cte_K}:
$$g_{t_{k}}(I^{(k)}) = I\left(g_{t_{k}}.X\right)\subset g_{t_k}.X.$$

\begin{figure}[h!]
\centering
\def\colphi{blue}
\def\colgamma{violet}
\def\colgamN{green!40!black}
\def\colp{red}
\def\colsa{orange}
\def\colsb{magenta}
\def\colsc{gray}
\def\colsd{brown}
\definecolor{colhom}{rgb}{0.7,0,0.1}

\definecolor{coli}{rgb}{1,0.3,0} % I(1)
\definecolor{colii}{rgb}{0.6,0.3,0}%{0.2,0.7,0.4} % I(2)

\begin{tikzpicture}[line cap=round,line join=round,>=stealth,x=0.65cm,y=0.65cm]
\draw (0,0) -- (3,1) -- (6.25,4) -- (7.25,4) -- (8.25,4) -- (10,-1) --  (6.75,-4) -- (5.75,-4) -- (4.75,-4) -- (1.75,-5) -- cycle;

%\foreach \x{}
\draw [line width=1pt] (0,0) -- (8,0);
%\draw [red,shift={(0,-0.02)}] (0,0) -- (2,0);
%\draw [\colphi] (5.5,-0.5) -- (5.5,3.5);
%\draw [\colphi] (9.5,-1.5) -- (9.5,0);
%\draw [\colphi] (1.5,-4) -- (1.5,0.5);
%\draw [\colphi] (2.5,-5+1/6) -- (2.5,-1.5);

\draw [colhom, ->] (1.75,-5) -- (0,0);
\draw [colhom, ->] (1.75,-5) -- (10,-1);

\draw [color=\colphi,->] (5.35,-0.83)-- (5.35,0.8);
\draw [color=\colphi] (5.35,0.7)-- (5.35,3.17);
\draw [color=\colphi] (9.1,-1.83)-- (9.1,1.57);
\draw [color=\colphi] (8.1,-2.75)-- (8.1,4);
\draw [color=\colphi,->] (6.4,4)-- (6.4,1.5);
\draw [color=\colphi] (6.4,1.6)-- (6.4,1.02);
\draw [color=\colphi] (4.35,2.25)-- (4.35,-4.13);
\draw [color=\colphi] (2.6,0.87)-- (2.6,-4.72);
\draw [color=\colphi] (0.85,0.28)-- (0.85,-2.43);
\draw [coli,shift={(0,-0.03)},line width=1pt] (0,0)-- (3.1,0);
\draw [colii,,shift={(0,0.03)},line width=1pt] (0,0)-- (0.63,0);

\begin{footnotesize}
\draw [below,line width=1pt] (6.5,0) node {$I$};
\draw [below, coli] (1.6,0) node {$I^{(1)}$};
\draw [above, colii] (0.3,0) node {$I^{(2)}$};
\draw [below left, colhom] (0.25,-0.5) node {$\zeta_2^{(0)}$};
\draw [below, colhom] (7,-2.5) node {$\zeta_1^{(0)}$};

\draw [left, \colphi] (8.1,-1.3) node {$\varphi_0^T(p)$};
\fill [\colp] (6.4,1.02) circle (1pt);
\draw [right] (6.4,1.02) node {$\Phi_T(p)$};
%\fill [\colp] (2.5,-0.5) circle (1pt);
%\draw [above] (2.5,-0.5) node {$\Phi_T(p)^+$};
%\fill [\colp] (1.5,-0.5) circle (1pt);
%\draw [above left] (1.5,-0.5) node {$\Phi_T(p)^-$};
%\draw [\colgamma] (4.5,-3-1/6) -- (4.5,2.5);
%\draw [\colgamma] (2.5,-4+1/6) -- (2.5,5/6);
\fill [\colp] (5.35,-0.83) circle (1pt); %(5.5,-0.5)
\draw [below] (5.35,-0.83) node {$p$};%(5.5,-0.5) node {$p$};

\draw [above] (1.5,0.5) node {$A$};
\draw [above left] (5,3) node {$B$};
\draw [above] (6.75,4) node {$C$};
\draw [above] (7.75,4) node {$C$};
\draw [above right] (9,2) node {$D$};

\draw [below left] (1,-3) node {$D$};
\draw [below] (3.5,-4.5) node {$A$};
\draw [below] (5.25,-4) node {$E$};
\draw [below] (6.25,-4) node {$E$};
\draw [below right] (8.5,-2.5) node {$B$};

\fill [\colsa] (0,0) circle (1pt);
\fill [\colsa] (1.75,-5) circle (1pt);
\fill [\colsa] (6.25,4) circle (1pt);
\fill [\colsa] (8.25,4) circle (1pt);
\fill [\colsa] (10,-1) circle (1pt);

\fill [\colsb] (3,1) circle (1pt);
\fill [\colsb] (4.75,-4) circle (1pt);
\fill [\colsb] (6.75,-4) circle (1pt);

\fill [\colsc] (7.25,4) circle (1pt);

\fill [\colsd] (5.75,-4) circle (1pt);
\end{footnotesize}
\end{tikzpicture}
\hfill
\begin{tikzpicture}[line cap=round,line join=round,>=stealth,x=0.65cm,y=0.65cm]
\draw (0,0) -- (3,1) -- (6.25,4) -- (7.25,4) -- (8.25,4) -- (10,-1) --  (6.75,-4) -- (5.75,-4) -- (4.75,-4) -- (1.75,-5) -- cycle;

%\foreach \x{}
\draw [line width=1pt] (0,0) -- (8,0);
%\draw [red,shift={(0,-0.02)}] (0,0) -- (2,0);
%\draw [\colphi] (5.5,-0.5) -- (5.5,3.5);
%\draw [\colphi] (9.5,-1.5) -- (9.5,0);
%\draw [\colphi] (1.5,-4) -- (1.5,0.5);
%\draw [\colphi] (2.5,-5+1/6) -- (2.5,-1.5);

\draw [colhom, ->] (1.75,-5) -- (0,0);
\draw [colhom, ->] (1.75,-5) -- (10,-1);

\draw [color=\colgamma,->] (5.35,-0.83)-- (5.35,2.8);
\draw [color=\colgamma] (5.35,2.7)-- (5.35,3.17);
\draw [color=\colgamma] (9.1,-1.83)-- (9.1,1.57);
\draw [color=\colgamma] (8.1,-2.75)-- (8.1,4);
\draw [color=\colgamma,->] (6.4,4)-- (6.4,1.5);
\draw [color=\colgamma] (6.4,1.6)-- (6.4,1.02);
\draw [color=\colgamma] (4.35,2.25)-- (4.35,-4.13);
\draw [color=\colgamN] (2.6,0)-- (2.6,-4.72);
\draw [color=\colgamma] (2.6,0.87)-- (2.6,0);
\draw [color=\colgamN] (0.85,0)-- (0.85,0.28);
\draw [color=\colgamma] (0.85,0)-- (0.85,-2.43);
\draw [coli,shift={(0,-0.03)},line width=1pt] (0,0)-- (3.1,0);
\draw [colii,shift={(0,0.03)},line width=1pt] (0,0)-- (0.63,0);
\draw [color=\colgamma] (5.35,-0.83)-- (6.4,1.02);

%closing \gammaN
\draw [color=\colgamN] (0.85,0.05)-- (2.6,0.05);
\draw [color=\colgamma] (0.85,-0.05)-- (2.6,-0.05);

\begin{scriptsize}%{footnotesize}
\draw [below] (6.5,0) node {$I$};
\draw [above, coli,line width=1pt] (1.75,-0.05) node {$I^{(1)}$};
\draw [above, colii,line width=1pt] (0.3,0) node {$I^{(2)}$};
\draw [below left, colhom] (0.25,-0.5) node {$\zeta_2^{(0)}$};
\draw [below, colhom] (7,-2.5) node {$\zeta_1^{(0)}$};

\draw [left, \colgamma] (8.1,-1.3) node {$\gamma_1^{(0)}+\gamma_2^{(0)}$};
\draw [left, \colgamN] (2.6,-2.3) node {$\gamma^{(1)}$};
\fill [\colp] (6.4,1.02) circle (1pt);
\draw [right] (6.4,1.02) node {$\Phi_T(p)$};
\fill [\colp] (5.35,-0.83) circle (1pt); %(5.5,-0.5)
\draw [below] (5.35,-0.83) node {$p$};%(5.5,-0.5) node {$p$};

%\fill [\colp] (6.4,0) circle (1pt);
%\draw [right] (6.4,0) node {$\Phi_T(p)$};
%\fill [\colp] (5.35,0) circle (1pt);

\fill  (0.85,0) circle (1pt); %[\colp]
\draw [below right] (0.85-0.1,0) node {$\Phi_{{T_1^{(1)}}}(p)$};

\fill (2.6,0) circle (1pt); %[\colp] 
\draw [above right] (2.6-0.1,-0.1) node {$\Phi_{{T_2^{(1)}}}(p)$};

\draw [above] (1.5,0.5) node {$A$};
\draw [above left] (5,3) node {$B$};
\draw [above] (6.75,4) node {$C$};
\draw [above] (7.75,4) node {$C$};
\draw [above right] (9,2) node {$D$};

\draw [below left] (1,-3) node {$D$};
\draw [below] (3.5,-4.5) node {$A$};
\draw [below] (5.25,-4) node {$E$};
\draw [below] (6.25,-4) node {$E$};
\draw [below right] (8.5,-2.5) node {$B$};

\fill [\colsa] (0,0) circle (1pt);
\fill [\colsa] (1.75,-5) circle (1pt);
\fill [\colsa] (6.25,4) circle (1pt);
\fill [\colsa] (8.25,4) circle (1pt);
\fill [\colsa] (10,-1) circle (1pt);

\fill [\colsb] (3,1) circle (1pt);
\fill [\colsb] (4.75,-4) circle (1pt);
\fill [\colsb] (6.75,-4) circle (1pt);

\fill [\colsc] (7.25,4) circle (1pt);

\fill [\colsd] (5.75,-4) circle (1pt);
\end{scriptsize}%\end{footnotesize}
\end{tikzpicture}
\caption{Illustration of the different curves} 
\label{fig:def_curves}
\end{figure}
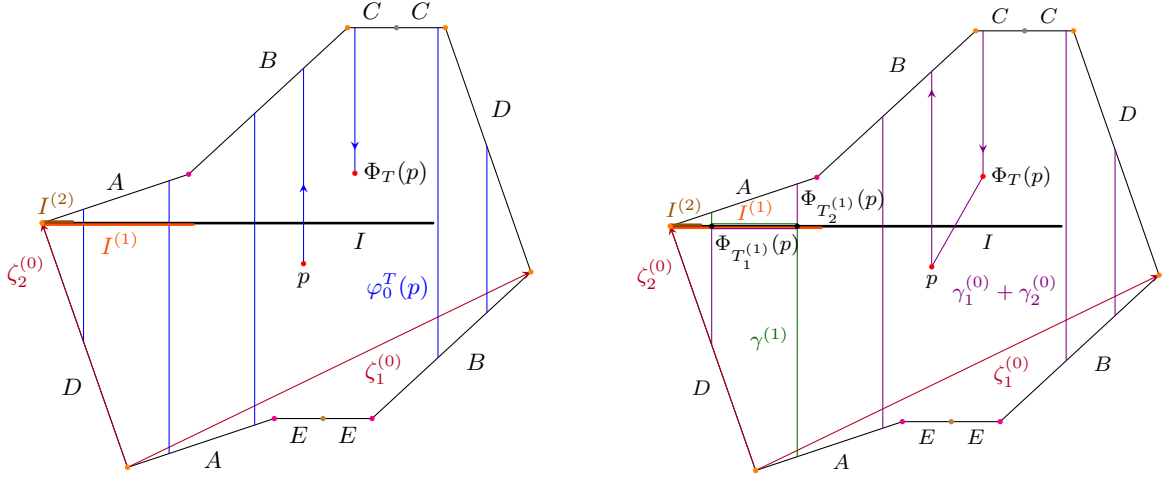

If $T$ is big enough, then the vertical path $\varphi_0^T(p)$ on $X$ intersects the segment $I$ several times by Lemma~\ref{lem:cte_K}. It will also intersect several times the segments $I^{(k)}$ for the first integers $k$. 
Let $n$ be the greatest integer such that $\gamma_T$ intersects at least twice the segment $I^{(n)}$, and let $T_1^{(n)}$ (resp. $T_2^{(n)}$) be the first (resp. last) intersection time between $\varphi_0^T(p)$ and $I^{(n)}$.

Set $\gamma^{(n)}$ to be the closed curve that is the union of $\varphi_{T_1^{(n)}}^{T_2^{(n)}}(p)$ with the horizontal path between $\Phi_{T_1^{(n)}}(p)$ and $\Phi_{T_2^{(n)}}(p)$ (which is a subinterval of $I^{(n)}$). In other words, the curve $\gamma^{(n)}$ is vertical from $\Phi_{T_1^{(n)}}(p)$ to $\Phi_{T_2^{(n)}}(p)$ and then horizontal (along $I^{(n)}$) from  $\Phi_{T_2^{(n)}}(p)$ to $\Phi_{T_1^{(n)}}(p)$.
Define similarly
$$\left\{
\begin{array}{rl}
\rho_1^{(n)}&=\gamma_{0,T_1^{(n)}}(p) \\
\rho_2^{(n)}&=\gamma_{T_2^{(n)},T}(p)
\end{array}
\right. .
$$
Seen as elements of $H_1(X,\RR)$, $\gamma_T$, $\gamma^{(n)}$, $\rho_1^{(n)}$ and $\rho_2^{(n)}$ satisfy: $\gamma_T = \rho_1^{(n)} + \gamma^{(n)} + \rho_2^{(n)}$. 

We now decompose $\rho_1^{(n)}$ and $\rho_2^{(n)}$ as we have decomposed $\gamma$. 
Let $T_1^{(n-1)}$ (resp. $T_2^{(n-1)}$) be the first (resp. last) intersection time between $\rho_1^{(n)}$ (resp. $\rho_2^{(n)}$) and $I^{(n-1)}$. Note that $\rho_1^{(n)}$ (resp. $\rho_2^{(n)}$) ends (resp. starts) in $I^{(n)}\subset I^{(n-1)}$ so $T_1^{(n-1)}$ (resp. $T_2^{(n-1)}$) is well-defined.
Set
$$\left\{
\begin{array}{rl}
\gamma_1^{(n-1)}&=\gamma_{T_1^{(n-1)},T_1^{(n)}}(p) \\
\gamma_2^{(n-1)}&=\gamma_{T_2^{(n)},T_2^{(n-1)}}(p) 
\end{array}
\right.
$$
and
$$\left\{
\begin{array}{rl}
\rho_1^{(n-1)}&=\gamma_{0,T_1^{(n-1)}}(p) \\
\rho_2^{(n-1)}&=\gamma_{T_2^{(n-1)},T}(p) 
\end{array}
\right. .
$$
We may have $T_1^{(n-1)}=T_1^{(n)}$ (resp. $T_2^{(n-1)}=T_2^{(n)}$) and then $\gamma_1^{(n-1)}$ (resp. $\gamma_2^{(n-1)}$) is empty.
By induction, we get an increasing sequence of times 
$$0\leqslant T_1^{(1)}\leqslant\dots\leqslant T_1^{(n-1)} \leqslant T_1^{(n)}\leqslant T_2^{(n-1)}\dots\leqslant T_2^{(1)}\leqslant T$$ 
such that $T_1^{(k)}$ (resp. $T_2^{(k)}$) is the first (resp. last) time in which $\varphi_0^T(p)$ intersects $I^{(k)}$. We close each vertical path $\varphi_{T_1^{k}}^{T_1^{k+1}}$ with a horizontal path along $I^{(k)}$ and call it $\gamma_1^{(k)}$. We define $\gamma_2^{(k)}$ in the same way: we close each vertical path $\varphi_{T_2^{k+1}}^{T_2^{k}}$ by a horizontal path along $I^{(k)}$.
It remains $\rho_1^{(0)}$ and $\rho_2^{(0)}$. We have
$$\gamma = \rho_1^{(0)} + \sum_{k=0}^{n-1} \gamma_1^{(k)}(p)+\gamma^{(n)} +\sum_{k=0}^{n-1} \gamma_2^{(k)} + \rho_2^{(0)},$$ 
where
for any $k\in\{0,\dots,n-1\}$:

\begin{itemize}
\item
$\rho_1^{(0)}=\gamma_{0,T_1^{(0)}}(p)$ starts at $p$, is vertical until $\Phi_{T_1^{(0)}}(p) \in I^{(0)}=I$ and is closed with a path that does not intersect $\xi_1$, $\xi_2$.
\item
$\gamma_1^{(k)}$ is a non empty loop with a vertical part starting in $I^{(k)}$ , ending in $I^{(k+1)}$, that intersects at most once $I^{(k+1)}$, and closed with a horizontal part included in $I^{(k)}$,
\item
$\gamma^{(n)}$ is a loop with a vertical part starting and ending in $I^{(n)}$ that intersects at most once $I^{(n+1)}$, and closed with a horizontal part included in $I^{(n)}$,
\item
$\gamma_2^{(k)}$ is a loop with a vertical part starting in $I^{(k+1)}$, ending in $I^{(k)}$, that intersects at most once $I^{(k+1)}$, and closed with a horizontal part included in $I^{(k)}$,
\item
$\rho_2^{(0)}=\gamma_{T_2^{(0)},T}(p)$ starts at $\Phi_{T_2^{(0)}}(p) \in I^{(0)}=I$, is vertical until $\Phi_T(p) $ and is closed with a path that does not intersect $\xi_1$ nor $\xi_2$.
\end{itemize}
From now on, we replace $\gamma_j^{(0)}$ by $\gamma_j^{(0)}+\rho_j^{(0)}$. This loop still intersects at most once $I^{(1)}$, which is the only property we use about $\gamma_i^{(0)}$.
For $j\in\{1,2\}$, we write the loop $\gamma_j^{(k)}(p)$, seen as an element of $H_1(X,\ZZ)$, in the basis $\left(\zeta_1^{(k)},\zeta_2^{(k)}\right)$ (see Definition~\ref{def:zetak}): 
$$\gamma_j^{(k)} = n_{j,1}^{(k)} \zeta_1^{(k)} + n_{j,2}^{(k)} \zeta_2^{(k)}$$
and
$$\gamma^{(n)} = m_1^{(n)} \zeta_1^{(n)} + m_2^{(n)} \zeta_2^{(n)}.$$
Note that $\gamma^{(n)}$ is not empty so $|m_1^{(n)}|+|m_2^{(n)}|$ does not vanish.
Of course, the curves depends on $p$, hence the coefficients $n_{j,i}^{(k)}$ and $m_j^{(k)}$ depend also on $p$. However, we will bound the sum of their absolute value from above by a constant that does not depend on $p$.
The curve $\gamma_i^{(k)}$ can not be too long because it intersects at most once the segment $I^{(k+1)}$, which is such that $g_{t_{k+1}}.I^{(k+1)} = I\left(g_{t_{k+1}}.X\right)$ (this is why we defined $I^{(k)}$ this way) and Lemma~\ref{lem:cte_K} applied to $g_{t_{k+1}}.\gamma_i^{(k)}\subset g_{t_{k+1}}.X$ gives:
$$
l_{\text{v}}\left(\gamma_i^{(k)}\right) =  e^{-t_{k+1}}  l_{\text{v}}\left(g_{t_{k+1}}.\gamma_i^{(k)}\right) \, \leqslant \,  e^{-t_{k+1}} K.
$$
Moreover, $l_{\text{v}}\left(\gamma_i^{(k)}\right) \geqslant c\sum_{j=1}^{2} \left| n_{i,j}^{(k)} \right| l_{\text{v}}\left(\zeta_j^{(n)}\right)$ by \Cref{lem:lower_bound_vert_length}.
Together with Lemma~\ref{lem:len_zeta}, we get:
$$
l_{\text{v}}(\gamma_i^{(k)}) \geqslant c \sum_{j=1}^{2} \left|n_{i,j}^{(k)}\right| K^{-1} e^{-t_k}.
$$
From these two bounds, we get, for $i=1$ and $i=2$:
$$
\sum_{j=1}^{2} \left|n_{i,j}^{(k)}\right| \leqslant \frac{K^2}{c} e^{t_k-t_{k+1}}.
$$
We define $m_1^{(k)}=n_{1,1}^{(k)}+n_{2,1}^{(k)}$ and $m_2^{(k)}=n_{1,2}^{(k)}+n_{2,2}^{(k)}$ so that:
$$\gamma_1^{(k)}+\gamma_2^{(k)} = m_1^{(k)} \zeta_1^{(k)} + m_2^{(k)} \zeta_2^{(k)}.$$
Then we get the bound announced:
$$
\sum_{j=1}^{2} \left|m_{j}^{(k)}\right| \leqslant \sum_{i,j=1}^{2} \left|n_{i,j}^{(k)}\right| \leqslant \frac{2K^2}{c} e^{t_k-t_{k+1}}.
$$

\vspace{-\belowdisplayskip}\[\]
\end{proof}

\section{Conclusion}\label{sec:ccl}

\subsection{Computation of the algebraic intersection number}

Take $w\in H_1(X,\RR)$ and write $(w_1^{(k)},w_2^{(k)})$ the coordinates of $w$ in the basis $\left(\zeta_1^{(k)},\zeta_2^{(k)}\right)$, then we have, using the decomposition of \Cref{lem:decomp_traj}:

\begin{align}
\iota(\gamma_{T}(p),w) &= \iota\left( \sum_{k=0}^{n} \sum_{i=1}^{2} m_i^{(k)} \zeta_i^{(k)}, w\right)= \sum_{k=0}^{n} \sum_{i,j=1}^{2} m_i^{(k)} \iota\left(\zeta_i^{(k)},w\right)   \notag \\
&= \sum_{k=0}^{n} \sum_{i,j=1}^{2} m_i^{(k)} w_{j}^{(k)} \iota\left(\zeta_i^{(k)},\zeta_{j}^{(k)}\right) \notag  \\
				&= \sum_{k=0}^{n} \left( m_1^{(k)} w_{2}^{(k)} + m_2^{(k)} w_{1}^{(k)} \right).\label{eqn:comput_int_nb}
\end{align}

We will control the $m_j$'s with Lemma~\ref{lem:decomp_traj} and the $w_j^{(k)}$'s by choosing a good $w$ with the Oseledets theorem applied to (a discrete version of) the Kontsevich--Zorich cocycle. 
 
\subsection{Conclusion using the Oseledets theorem}\label{subsec:ccl_Oseledets}

We have everything to prove Theorem~\ref{thm:bounded_intersection}.

\begin{proof}[Proof of \Cref{thm:bounded_intersection}]
Let us first prove the statement for $X\in\mathcal{O}_\varepsilon$. We consider $\mathcal{U}$ a neighborhood of $X$ in $\mathcal{O}_\varepsilon$ where Lemmas~\ref{lem:cte_K}, \ref{lem:lower_bound_vert_length} and~\ref{lem:decomp_traj} hold. We consider the discrete version of the Kontsevich--Zorich cocycle given by the first returns to $\mathcal{U}$.
The surface $X$ is Oseledets generic, so the Oseledets theorem (\Cref{thm:Oseledets}) applied to this discrete version of the Kontsevich--Zorich cocycle gives two Lyapunov exponents $ \theta_1\leqslant\theta_2$ and a flag decomposition $\{0\}=E_0\subset E_1 \subset E_2 = H_1(X,\RR)$ such that
$$
\forall v\in E_i\setminus E_{i-1}, \, \frac{\log\left\|A_{KZ}^{(n)}v\right\|}{n}\underset{n\rightarrow\infty}{\longrightarrow} \theta_i.
$$
Moreover, Bell, Delecroix, Gadre, Gutierez-Romo and Schleimer \cite{bell2021flow} showed that the plus and minus Lyapunov spectra of any component of any stratum of half-translation surfaces are simple, in particular here $\theta_1<\theta_2$. % in the stratum $\mathcal{Q}(1^2,-1^2)$.
Furthermore, the Kontsevich--Zorich cocycle is symplectic, which implies that $\theta_1=-\theta_2$. We conclude that $\theta_1<0$, namely $\lambda := e^{\frac{\theta_1}{2}} < 1$.

Let $w\in H_1(X,\RR)$ be a vector in the contracted direction $E_1$, and let $\delta\in (0,\frac{\theta_2}{2})$. Then there exists $k_0$ such that:
$$ \forall k\geqslant k_0, \, \log\left\|w^{(k)}\right\|= \log\left\|A_{KZ}^{(k)}w\right\|  \leqslant k(\theta_1+\delta) < k\frac{\theta_1}{2}. $$

Then $\left\|w_j^{(k)}\right\| \leqslant \lambda^k \|w\|$ and using \Cref{eqn:comput_int_nb}, we get
\begin{align*}
\left|\iota(\gamma_{0,T}(p),w)\right| &\leqslant \sum_{k=0}^{n} \left\|w^{(k)}\right\| \left(\left| m_1^{(k)}\right|  + \left| m_2^{(k)} \right|  \right) \\
		&\leqslant \sum_{k=0}^{n} \lambda^k \|w\| \left( \left| m_1^{(k)}\right|  + \left| m_2^{(k)} \right|  \right) \\	
		&< \|w\| \sum_{k=0}^{n} \lambda^k \frac{K^2}{c} e^{t_k-t_{k+1}}.
\end{align*}

We write $M=\frac{1}{\mu(\mathcal{U})}$ the inverse of the measure of $\mathcal{U}$ in $\mathcal{Q}(-1^2,1^2)$. By Kac's lemma $ \frac{t_k}{k} \underset{k\rightarrow\infty}{\longrightarrow} M $, so for any $\delta>0$, there exists $k_0\in\NN$ such that
$$
\forall k\geqslant k_0, \quad -M - \delta(2k +1) \leqslant t_k - t_{k+1} \leqslant -M +\delta(2k+1), 
$$
which leads to
$$
\left|\iota(\gamma_{0,T}(p),w)\right| < \|w\|_{\infty}  \left( \sum_{k=0}^{k_0-1} \lambda^k \frac{K^2}{c} e^{t_k-t_{k+1}} + e^{-M+\delta}\sum_{k=k_0}^{n} \lambda^k e^{2k\delta} \right).
$$
For $\delta$ small enough, this is a convergent geometric sum. 
This shows the lemma, for 
$$C= \|w\|_{\infty} \left( \frac{e^{-M+\delta}}{1-\lambda e^{2\delta}} +\sum_{k=0}^{k_0-1} \lambda^k \frac{K^2}{c} e^{t_k-t_{k+1}} \right).$$
Note that this value depends on the surface $X$ (because of the times $t_k$, for $k\in\{0,k_0-1\}$).
This proves the theorem for almost every surface in $\mathcal{O}_\varepsilon$.

Now, remark that if the surface $X\in\mathcal{Q}(-1^2,1^2)$ satisfies the theorem with a certain cycle $w\in H_1(X,\RR)$, then the cycle $g_t w$ (given by the image under $g_t$ of any representative of the homology class $w$) in the homology group of the surface $g_tX$ satisfies: for every $p\in g_t X$, $q = g_{-t}p\in X$ is such that 
$$|\iota(g_t w,\gamma_T(p))|= | \iota(w, \gamma_{e^T}(q)) |<C$$ 
for every $T>0$.
We deduce that the theorem holds for the image of the open set $\mathcal{O}_\varepsilon$ under $g_t$. Since $g_t$ is ergodic, the theorem holds for the entire stratum. 
\end{proof}

This proves on the same occasion Theorem~\ref{thm:tb_in_wt} as a corollary, as stated by Lemma~\ref{lem:transition} . 

Goujard \cite{Goujard} gives the exact value of volumes of strata of small genus. Together with the formula of Eskin, Kontsevich and Zorich \cite{EKZ} on the sum of Lyapunov exponents, we get that the exact values of the Lyapunov exponents for the Teichmüller flow in the stratum $\mathcal{Q}(-1^2,1^2)$ are $\theta_1=-\frac{2}{3}$ and $\theta_2=\frac{2}{3}$. 

\section{Concluding remarks and perspectives}
We used the existence of a contracted direction in the homology by the Kontsevich-Zorich cocycle to show the existence of a trapping strip. But there is also a dilated direction (with Lyapunov exponent $\theta = 2/3$ in our case). This means that a generic trajectory in a generic setting has a logarithmic growth rate of $2/3$:
$$
\underset{n\rightarrow\infty}{\limsup} \frac{\log |\Psi_t(p) - p|}{\log t} = \frac{2}{3}.
$$
Of course, this expansion has to be in the direction of the trapping strip.
This result can be proved using similar methods as in \cite{jay}.

We showed that the vertical trajectories are trapped in a strip for \emph{almost every} lattice $\Lambda$ according to which we place the obstacles. This result reminds the result for Eaton lenses \cite{FS}, which was generalized by Fr{\k{a}}czek , Shi, and Ulcigrai \cite{FSU}. They show that for \emph{every} lattice $\Lambda$, the trajectories in \emph{almost every} direction are trapped in a strip of the plane. 
It is a work in progress with Fr{\k{a}}czek 
to prove the same result for wind-tree tiling billiards. 

\printbibliography

\vspace{0.5cm}

Magali Jay, \textsc{Max Planck Institute for Mathematics in the Sciences, Leipzig, Germany}

E-mail adrdress: \url{magali.jay@mis.mpg.de}
\end{document}